\documentclass[a4paper,12pt]{article}
\usepackage{amsthm,
	    amsmath,
	    amssymb,
	    bbm,
	    geometry,
	    epsfig,
	    hyperref,
	    enumerate,
	    comment,
	    nicefrac,
	    listings,
	    latexsym,
	    mathrsfs,
	    mathtools,
	    rotating,
	    float,
	    subeqnarray,
	    cases,
	    indentfirst,
	    graphicx,
	    subfigure,
	    amsfonts,
	    booktabs,
	    bm,
	    epstopdf,
	    outlines,
	    color,
	    todonotes
}

\renewcommand{\lstlistingname}{{\sc Python} code} 

\newtheorem{lemma}{Lemma}[section]
\newtheorem{remark}[lemma]{Remark}

\newtheorem{algo}[lemma]{Framework}

\providecommand{\N}{{\ensuremath{\mathbb{N}}}}

\providecommand{\R}{{\ensuremath{\mathbb{R}}}}
\providecommand{\B}{\mathcal{B}}

\renewcommand{\P}{\mathbbm{P}}
\providecommand{\bS}{\mathbb{S}}
\renewcommand{\S}{\mathcal{S}}
\providecommand{\E}{{\ensuremath{\mathbbm{E}}}}

\providecommand{\N}{{\ensuremath{\mathbbm{N}}}}

\providecommand{\R}{{\ensuremath{\mathbbm{R}}}}

\providecommand{\E}{{\ensuremath{\mathbb{E}}}}

\newcommand{\F}{{\ensuremath{\mathcal{F}}}}
\newcommand{\bF}{{\ensuremath{\mathbb{F}}}}

\newcommand{\cL}{{\ensuremath{\mathcal{L}}}}

\newcommand{\G}{{\ensuremath{\mathbbm{G}}}}
\newcommand{\A}{{\ensuremath{\mathbbm{A}}}}
\newcommand{\U}{{\ensuremath{\mathbbm{U}}}}
\newcommand{\V}{{\ensuremath{\mathbbm{Z}}}}
\newcommand{\X}{{\ensuremath{\mathcal{X}}}}
\newcommand{\Y}{{\ensuremath{\mathcal{Y}}}}
\newcommand{\cZ}{{\ensuremath{\mathcal{Z}}}}

\title{
Machine learning approximation algorithms \\ 
for high-dimensional fully nonlinear partial \\
differential equations and second-order \\ 
backward stochastic differential equations
}

\author{Christian Beck$^1$, 
	Weinan E$^2$,
	and 
	Arnulf Jentzen$^3$
\medskip
\\
\small{$^1$ETH Zurich (Switzerland), 
e-mail: christian.beck (at) math.ethz.ch}
\smallskip
\\
\small{$^2$Beijing Institute of Big Data Research (China), 
Princeton University (USA),}
\\
\small{and Peking University (China), 
e-mail: weinan (at) math.princeton.edu}
\smallskip
\\
\small{$^3$ETH Zurich (Switzerland), 
e-mail: arnulf.jentzen (at) sam.math.ethz.ch}}	
	
\begin{document}

\maketitle
 \begin{abstract}
 High-dimensional partial differential equations (PDE) appear in a 
 number of \mbox{models} from the financial industry, such as in 
 derivative pricing models, credit valuation adjustment (CVA) models, 
 or portfolio optimization models.  
 The PDEs in such applications are high-dimensional as the dimension 
 corresponds to the number of financial assets in a portfolio.
 Moreover, such PDEs are often fully nonlinear 
 due to the need to incorporate certain nonlinear phenomena in the model 
 such as default risks, transaction costs, volatility uncertainty 
 (Knightian uncertainty), or trading constraints in the model. 
 Such high-dimensional fully nonlinear 
 PDEs are exceedingly difficult to solve as the computational effort 
 for standard approximation methods grows exponentially with the 
 dimension. In this work we propose a new method for solving 
 high-dimensional fully nonlinear second-order PDEs. Our method can in 
 particular be used to sample from high-dimensional nonlinear 
 expectations. The method is based on 
 (i) a connection between fully nonlinear second-order PDEs and 
 second-order backward stochastic differential equations (2BSDEs), 
 (ii) a merged formulation of the PDE and the 2BSDE problem, 
 (iii) a temporal forward discretization of the 2BSDE and a spatial 
 approximation via deep neural nets, 
 and (iv) a stochastic gradient descent-type optimization procedure. 
 Numerical results obtained using {\sc TensorFlow} in {\sc Python} 
 illustrate the efficiency and the accuracy of the method in the cases 
 of a $100$-dimensional Black-Scholes-Barenblatt equation, 
 a $100$-dimensional Hamilton-Jacobi-Bellman equation, 
 and a nonlinear expectation of a $ 100 $-dimensional $ G $-Brownian motion.
\end{abstract}

\begin{center}
\emph{Keywords:} 
deep learning, 
second-order backward stochastic differential\\ 
equation,
2BSDE,
numerical methhod, 
Black-Scholes-Barentblatt equation,
\\
Knightian uncertainty,
Hamiltonian-Jacobi-Bellman equation,
\\
HJB equation,
nonlinear expectation,
$ G $-Brownian motion
\end{center}

\tableofcontents

\section{Introduction}
 Partial differential equations (PDE) play an important role in an  
abundance of \mbox{models} from finance to physics. 
Objects as the wave function associated to a quantum physical system,
the value function describing the fair price of a financial 
derivative in a pricing model,
or the value function describing the expected maximal utility in a 
portfolio optimization problem are often given as solutions to 
nonlinear PDEs. 
Roughly speaking, the nonlinearities in the above mentioned PDEs from 
financial engineering applications appear due to certain nonlinear 
effects in the trading portfolio (the trading portfolio 
for hedging the financial derivative claim in the case of derivative 
pricing problems and the trading portfolio whose utility has to be 
maximized in the case of portfolio optimization problems);
see, e.g.,  
  \cite{Bender2015Primal,
	Bergman1995, 
	ElKarouiPengQuenez1997,
	GobetLemorWarin2005, 
	laurent2014overview,
	LemorGobetWarin2006}
for derivative pricing models with different interest rates for 
borrowing and lending,
see, e.g.,
  \cite{Crepeyetal2013, 
	Labordere2012}
for derivative pricing models incorporating the default risk 
of the issuer of the financial derivative,
see, e.g.,  
  \cite{Bayraktar2009Valuation, 
	Bayraktar2008Pricing,
	windcliff2007hedging}
for models for the pricing of financial derivatives 
on untradable underlyings such as financial 
derivatives on the temperature
or mortality-dependent financial derivatives,
see, e.g.,  
  \cite{Amadori2003Nonlinear,ekren2017portfolio,Moreau2017Trading}
for models incorporating that the trading strategy influences 
the price processes though demand and supply 
(so-called large investor effects),
see, e.g.,
 \cite{ForsythVetzal2001,
       GuyonLabordere2011, 
       Leland1985,Possamai2015Homogenization}       
for models taking transaction costs
in the trading portfolio into account,
and see, e.g., 
\cite{AvellanedaLevyParas1995,GuyonLabordere2011}
for models incorporating
uncertainties in the model parameters 
for the underlyings (Knightian uncertainty).
The PDEs emerging from such models are often 
high-dimensional as the associated trading portfolios frequently contain
a whole basket of financial assets
(see, e.g., \cite{Bender2015Primal,Crepeyetal2013,ForsythVetzal2001}).
These high-dimensional nonlinear PDEs are typically exceedingly 
difficult to solve approximatively. Nonetheless, there is a strong 
demand from the financial engineering industry to approximatively 
compute the solutions of such high-dimensional nonlinear parabolic PDEs 
due to the above mentioned practical relevance of these PDEs.

There are a number of numerical methods for solving nonlinear 
parabolic PDEs approximatively in the literature. 
Some of these methods are deterministic approximation methods and others 
are random approximation methods which rely on suitable probabilistic 
representations of the corresponding PDE solutions 
such as probabilistic representations based on 
backward stochastic differential equations (BSDEs) 
(cf., e.g., \cite{Bismut1973,
      Bouchard2015Lecture,
      geiss2014decoupling,
      PardouxPeng1990, 
      PardouxPeng1992, 
      PardouxTang1999, 
      Peng1991}),
probabilistic representations based on 
second-order backward stochastic differential equations (2BSDEs) 
(cf., e.g., \cite{CheriditoSonerTouziVictoir2007}), 
probabilistic representations based on branching diffusions
(cf., e.g., \cite{Labordere2012, 
		  Labordereetal2016arXiv, 
		  LabordereTanTouzi2014,
		  McKean1975,
		  SkorohodBranchingDiffusion1964, 
		  Watanabe1965Branching}),
and 
probabilistic representations based on extensions of the classical 
Feynman-Kac formula 
(cf., e.g., \cite{LinearScaling, 
		  KaratzasShreve2ndEdition,
		  OksendalSDEs}).
We refer, e.g., 
to 
\cite{chiaramontesolving, 
      dehghan2009numerical, 
      lagaris1998artificial, 
      LeeKang1990, 
      ramuhalli2005finite,
      Tadmor2012,
      Thomee1997,
      PetersdorffSchwab2004}
for deterministic approximation methods for PDEs, 
to 
\cite{BallyPages2003,
      BenderDenk2007, 
      Bender2015Primal, 
      BouchardTouzi2004,
      Chassagneux2014, 
      ChassagneuxCrisan2014,
      ChassagneuxRichou2015, 
      ChassagneuxRichou2016,
      CrisanManolarakis2010,
      CrisanManolarakis2012, 
      CrisanManolarakis2014, 
      CrisanManolarakisTouzi2010, 
      DelarueMenozzi2006,
      DouglasMaProtter, 
      GobetLabart2010, 
      GobetLemor2008Numerical, 
      GobetLemorWarin2005,  
      GobetLopezSalasTurkedjiev2016,
      GobetTurkedjiev2016,
      GobetTurkedjiev2016MathComp,  
      HuijskensRuijterOosterlee2016,
      LabartLelong2013,
      LemorGobetWarin2006, 
      LionnetDosReisSzpruch2015, 	      
      MaProtterSanMartin2002, 
      MaProtterYong1994, 
      MaYong1999, 
      MilsteinTretyakov2006, 
      MilsteinTretyakov2007, 
      Pham2015,
      RuijterOosterlee2015,
      RuijterOosterlee2016,
      Ruszczynski2017Dual,
      Turkedjiev2015,
      Zhang2004}
for probabilistic approximation methods for PDEs based on temporal 
discretizations of BSDEs, 
to 
\cite{EHanJentzen2017, 
      EHanJentzen2017overcoming}
for probabilistic approximation methods for PDEs based on suitable deep learning approximations for BSDEs, 
to 
\cite{FuZhaoZhou,
      ZhangGunzburgerZhao2013} 
for probabilistic approximation methods for BSDEs based on sparse grid approximations, 
to 
\cite{BriandLabart2014, 
      GeissLabart2016}
for probabilistic approximation methods for BSDEs based on Wiener Chaos expansions, 
to
\cite{BouchardElieTouzi2009,
      CheriditoSonerTouziVictoir2007,
      FahimTouziWarin2011,
      GuoZhangZhuo2015,
      KongZhaoZhou2015, 
      KongZhaoZhou2017}
for probabilistic approximation methods for PDEs based on temporal 
discretization of 2BSDEs, 
to
\cite{ChangLiuXiong2016,
      Labordere2012,
      Labordereetal2016arXiv,
      LabordereTanTouzi2014,
      RasulovRaimoveMascagni2010,
      warin2017variations}
for probabilistic approximation methods for PDEs based on 
branching diffusion representations,
and 
to \cite{LinearScaling, MultilevelPicard} 
for probabilistic approximation methods based on extensions of the classical Feynman-Kac formula.

Most of the above named approximation methods are, however, only 
applicable in the case where the PDE/BSDE dimension $ d $ is rather 
small or work exclusively in the case of serious restrictions on the 
parameters or the type of the considered PDE (e.g., small nonlinearities, 
small terminal/initial conditions, semilinear structure of the PDE). 
The numerical solution of a high-dimensional nonlinear PDE thus 
remains an exceedingly difficult task
and there is only a limited number of situations where practical 
algorithms for high-dimensional PDEs have been developed (cf., e.g., 
\cite{Darbon2016Algorithmus, 
      EHanJentzen2017, 
      LinearScaling, 
      MultilevelPicard, 
      EHanJentzen2017overcoming, 
      Labordereetal2016arXiv, 
      LabordereTanTouzi2014,
      SirignanoDGM2017}).
In particular, to the best of our knowledge, at the moment 
there exists no practical algorithm for high-dimensional 
fully nonlinear parabolic PDEs in the scientific literature.

In this work we intend to overcome this difficulty, that is, 
we propose a new algorithm for solving fully-nonlinear PDEs
and nonlinear second-order backward stochastic differential 
equations 2BSDEs. Our method in particular can be used to sample 
from Shige Peng's nonlinear expectation in high space-dimensions
(cf., e.g., \cite{Peng2004GExpectations,
		  Peng2005NonlinearExpMarkov,
		  peng2007g, 
		  Peng2010BookNonlinearExp}).
The proposed algorithm 
exploits a connection between PDEs and 2BSDEs 
(cf., e.g., Cheridito et al.~\cite{CheriditoSonerTouziVictoir2007}) 
to obtain a merged formulation of the PDE and the 2BSDE, 
whose solution is then approximated by combining 
a time-discretization with a neural network based 
deep learning procedure 
(cf., e.g., 
\cite{Bengio2009, 
      cai2017approximating, 
      Carleo2017Solving,
      chiaramontesolving, 
      dehghan2009numerical,
      EHanJentzen2017,
      HanE2016Arxiv,
      EHanJentzen2017overcoming,  
      khoo2017solving,
      Krizhevsky2012,
      lagaris1998artificial, 
      Lecun98, 
      LeCun2015, 
      LeeKang1990, 
      Meade1994Numerical,
      mehrkanoon2015learning,
      ramuhalli2005finite, 
      ruder2016overview}).
Roughly speaking, the merged formulation allows us to formulate the
original PDE problem as a learning problem.  
The random loss function for the deep neural network in our algorithm
is, loosely speaking, given by the error between the 
prescribed terminal condition of the 2BSDE 
and the neural network based forward time discretization of the 2BSDE. 
A related deep learning approximation algorithm 
for PDEs of semilinear type based on 
forward BSDEs has been recently proposed in 
\cite{EHanJentzen2017, EHanJentzen2017overcoming}. 
A key difference between 
\cite{EHanJentzen2017, EHanJentzen2017overcoming}
and the present work is that here we rely on the 
connection between fully nonlinear PDEs and 2BSDEs given by 
Cheridito et al.~\cite{CheriditoSonerTouziVictoir2007}
while \cite{EHanJentzen2017, EHanJentzen2017overcoming} rely on 
the nowadays almost classical connection between PDEs and BSDEs 
(cf., e.g., \cite{PardouxPeng1992, PardouxPeng1990, PardouxTang1999, Peng1991}). 
This is the reason why the method 
proposed in \cite{EHanJentzen2017, EHanJentzen2017overcoming} 
is only applicable to semilinear PDEs while the algorithm 
proposed here allows to treat fully nonlinear PDEs and 
nonlinear expectations.

The remainder of this article is organized as follows.
In Section~\ref{sec:algorithm_main_ideas} 
we derive 
(see Subsections~\ref{subsec:algo_fullyNonlinear}--\ref{subsec:algo_sgd} below)
and formulate
(see Subsection~\ref{subsec:specific_case} below)
a special case of the algorithm proposed in this work.  
In Section~\ref{sec:details} 
the proposed algorithm is derived
(see Subsections~\ref{subsec:details_fullyNonlinear}--\ref{subsec:details_deepandsgd} below) 
and formulated
(see Subsection~\ref{subsec:general_case} below)
in the general case. 
The core idea is most apparent  
in the simplified framework 
in Subsection~\ref{subsec:specific_case}
(see Framework~\ref{def:specific_case} below).
The general framework 
in Subsection~\ref{subsec:general_case}, in turn, 
allows for employing more sophisticated machine learning 
techniques (see Framework~\ref{def:general_algorithm} below). 
In Section~\ref{sec:examples} we present numerical results 
for the proposed algorithm in the case of several high-dimensional PDEs. 
In Subsection~\ref{subsec:example_allen_cahn_plain_sgd_no_bn} 
the proposed algorithm in the simplified framework 
in Subsection~\ref{subsec:specific_case} 
is employed to approximatively 
calculate the solution of 
a $ 20 $-dimensional Allen-Cahn equation. 
In Subsections 
	\ref{subsec:example_bsb}, 
	\ref{subsec:example_square_gradient}, 
	\ref{subsec:example_allen_cahn}, 
and 
	\ref{subsec:example_gbm}
the proposed algorithm in the general framework 
in Subsection~\ref{subsec:general_case}
is used to approximatively calculate 
the solution of a $ 100 $-dimensional Black-Scholes-Barenblatt equation,
the solution of a $ 100 $-dimensional Hamilton-Jacobi-Bellman equation, 
the solution of a $ 50 $-dimensional Allen-Cahn equation, 
and nonlinear expectations of $ G $-Brownian motions in $ 1 $ and $ 100 $ 
space-dimensions. 
{\sc Python} implementations of the algorithms are provided 
in Section~\ref{sec:source_code}.

\section{Main ideas of the deep 2BSDE method}
 \label{sec:algorithm_main_ideas}
 In Subsections~\ref{subsec:algo_fullyNonlinear}--\ref{subsec:algo_sgd} below 
we explain the main idea behind the 
algorithm proposed in this work which we refer to as \emph{deep 2BSDE method}. 
This is done at the expense of a 
rather sketchy derivation and description. 
More precise and more general definitions of the deep 2BSDE method may be found in 
Sections~\ref{subsec:specific_case} 
and \ref{subsec:general_case} below. 
In a nutshell, the main ingredients of the deep 2BSDE method are
\begin{enumerate}[(i)]
 \item a certain connection between PDEs and 2BSDEs, 
 \item a merged formulation of the PDE and the 2BSDE problem,
 \item a temporal forward discretization of the 2BSDE and a spatial 
 approximation via deep neural nets, and
 \item a stochastic gradient descent-type optimization procedure.
\end{enumerate}
The derivation of the deep 2BSDE method is
mainly based on ideas in
E, Han, \& Jentzen~\cite{EHanJentzen2017} 
and 
Cheridito et al.~\cite{CheriditoSonerTouziVictoir2007} 
(cf., e.g., 
\cite[Section~2]{EHanJentzen2017}
and
\cite[Theorem~4.10]{CheriditoSonerTouziVictoir2007}).
Let us start now by describing the PDE problems which 
we want to solve with the deep 2BSDE method. 

\subsection{Fully nonlinear second-order PDEs}
\label{subsec:algo_fullyNonlinear}

Let  
  $d\in\N = \{1,2,3,\ldots\}$, 
  $T\in (0,\infty)$,  
  $ u = (u(t,x))_{t\in [0,T], x\in\R^d}\in C^{1,2}([0,T]\times\R^d,\R) $, 
  $ f \in C( [0,T]\times\R^d\times\R\times\R^d\times\R^{d\times d} , \R ) $,  
  $ g \in C( \R^d, \R ) $ 
satisfy for all $t\in [0,T)$, $x\in\R^d$ 
that $u(T,x) = g(x)$ and 
\begin{align}\label{eq:fullyNonlinearPDETerminalValueProblem}
  \tfrac{\partial u}{\partial t}(t,x) 
  & = 
  f\bigl(t,x,u(t,x),(\nabla_x u)(t,x),(\operatorname{Hess}_x u)(t,x)\bigr). 
\end{align}
The deep 2BSDE method allows us to approximatively compute the function $ u(0,x) $, $ x \in \R^d $. 
To fix ideas we restrict ourselves in this section 
to the approximative computation of the 
real number $ u(0,\xi) \in \R $ for some $\xi\in\R^d$ 
and we refer to Subsection~\ref{subsec:general_case} 
below for the general algorithm.
Moreover, the deep 2BSDE method can easily be extended 
to the case of systems of fully nonlinear second-order parabolic 
PDEs but in order to keep the notational complexity as low as 
possible we restrict ourself to the scalar case in this work
(cf.\ \eqref{eq:fullyNonlinearPDETerminalValueProblem} above). 
Note that the PDE~\eqref{eq:fullyNonlinearPDETerminalValueProblem} 
is formulated as a terminal value problem. We chose the 
terminal value problem formulation 
over the in the PDE literature more common initial 
value problem formulation  
because, on the one hand, the terminal value problem formulation 
seems to be more natural in connection with 
second-order BSDEs (which we are going to use below in the 
derivation of the proposed approximation algorithm) and because, on the other hand, 
the terminal value problem formulation 
shows up naturally in financial engineering applications 
like the Black-Scholes-Barenblatt equation 
in derivative pricing
(cf. Section \ref{subsec:example_bsb}).
Clearly, terminal value problems can be transformed into initial 
value problems and vice versa; see, e.g., 
Remark~\ref{rem:initial_vs_terminal_value} below.

\subsection{Connection between fully nonlinear second-order PDEs and 
2BSDEs}\label{subsec:algorithm_connection_pde_bsde}
Let 
  $(\Omega,\F,\P)$ 
  be a probability space,  
let  
  $W\colon [0,T]\times\Omega\to\R^d$ 
  be a standard Brownian motion on 
  $(\Omega,\F,\P)$ with continuous sample paths, 
let 
  $\bF = (\bF_t)_{t\in [0,T]}$ 
  be the normal filtration on $(\Omega,\F,\P)$ 
  generated by $W$, 
and 
let 
  $Y\colon[0,T]\times\Omega\to\R$, 
  $Z\colon[0,T]\times\Omega\to\R^d$,   
  $\Gamma\colon[0,T]\times\Omega\to\R^{d\times d}$, 
  and 
  $A\colon[0,T]\times\Omega\to\R^d$
  be $\bF$-adapted stochastic processes with continuous 
  sample paths which satisfy that for all $t\in [0,T]$ it 
  holds $\P$-a.s.~that 
  \begin{equation}\label{eq:SDEsforYAndZ_everythingPastedIn:Tag1}
  Y_t = 
  g(\xi + W_T) - 
  \int_t^T
    \bigl( 
    f(s,\xi + W_s, Y_s, Z_s, \Gamma_s) 
    + 
    \tfrac12\operatorname{Trace}(\Gamma_s)
    \bigr)\,ds 
  - 
  \int_t^T
    \langle Z_s, dW_s\rangle_{\R^d} 
  \end{equation}
  and
  \begin{equation}\label{eq:SDEsforYAndZ_everythingPastedIn:Tag2}
  Z_t =  Z_0 
  + 
  \int_0^t A_s\,ds 
  + 
  \int_0^t \Gamma_s\,dW_s. 
  \end{equation}
Under suitable smoothness and regularity hypotheses the fully nonlinear PDE 
\eqref{eq:fullyNonlinearPDETerminalValueProblem} is related to 
the 2BSDE system 
\eqref{eq:SDEsforYAndZ_everythingPastedIn:Tag1}--\eqref{eq:SDEsforYAndZ_everythingPastedIn:Tag2} 
in the sense that for all $t\in [0,T]$ it holds $\P$-a.s.~that
\begin{equation}\label{eq:definition_YZGA}
 Y_t = u(t,\xi + W_t)\in\R, \qquad
 Z_t = (\nabla_x u)(t,\xi+W_t)\in\R^d, 
\end{equation}
\begin{equation}\label{eq:DefinitionGamma}
 \Gamma_t = (\operatorname{Hess}_x u)(t,\xi+W_t)\in\R^{d\times d}, \qquad~\text{and}
\end{equation}
\begin{equation}\label{eq:DefinitionA}
 A_t = (\tfrac{\partial }{\partial t}\nabla_x u)(t,\xi+W_t) 
 + \tfrac12 (\nabla_x\Delta_x u)(t,\xi+W_t)\in\R^d
\end{equation}
(cf., e.g., 
Cheridito et al.~\cite{CheriditoSonerTouziVictoir2007} 
and  Lemma~\ref{lem:applying_ito} below).

\subsection{Merged formulation of the PDE and the 2BSDE}
\label{subsec:algo_mergedFormulation}
In this subsection we derive a merged formulation 
(see~\eqref{eq:SDEforYWithLearnablesPluggedIn} 
and~\eqref{eq:SDEforZWithLearnablesPluggedIn}) 
for the 
PDE~\eqref{eq:fullyNonlinearPDETerminalValueProblem} 
and the 
2BSDE system~\eqref{eq:SDEsforYAndZ_everythingPastedIn:Tag1}--\eqref{eq:SDEsforYAndZ_everythingPastedIn:Tag2}. 
More specifically, observe that 
\eqref{eq:SDEsforYAndZ_everythingPastedIn:Tag1} and 
\eqref{eq:SDEsforYAndZ_everythingPastedIn:Tag2} yield that 
for all $\tau_1,\tau_2\in [0,T]$ with $\tau_1\leq \tau_2$ it holds 
$\P$-a.s.~that 
\begin{equation}
 Y_{\tau_2} = Y_{\tau_1} 
 + \int_{\tau_1}^{\tau_2} 
 \bigl( 
  f(s,\xi + W_s,Y_s,Z_s,\Gamma_s) + \tfrac12\operatorname{Trace}(\Gamma_s)
  \bigr)\,ds 
 + 
 \int_{\tau_1}^{\tau_2}
  \langle Z_s, dW_s\rangle_{\R^d} \label{eq:SDEforY_t1t2} 
\end{equation}
and 
\begin{equation}
 Z_{\tau_2}  =  Z_{\tau_1} + 
 \int_{\tau_1}^{\tau_2} A_s\,ds 
 + 
 \int_{\tau_1}^{\tau_2} \Gamma_s\,dW_s. \label{eq:SDEforZ_t1t2}
\end{equation}
Putting \eqref{eq:DefinitionGamma} and \eqref{eq:DefinitionA} 
into \eqref{eq:SDEforY_t1t2} and \eqref{eq:SDEforZ_t1t2}
demonstrates that for all $\tau_1, \tau_2\in [0,T]$ with $\tau_1\leq \tau_2$ 
it holds $\P$-a.s.~that
\begin{equation}\label{eq:SDEforYWithLearnablesPluggedIn}
 \begin{split}
 &Y_{\tau_2} = Y_{\tau_1} 
 + 
 \int_{\tau_1}^{\tau_2} 
  \langle Z_s, dW_s\rangle_{\R^d} \\
 &+ \int_{\tau_1}^{\tau_2} 
 \Bigl( 
  f\bigl(s,\xi + W_s,Y_s,Z_s,(\operatorname{Hess}_x u)(s,\xi+W_s)\bigr) 
  + 
  \tfrac12\operatorname{Trace}\bigl((\operatorname{Hess}_x u)(s,\xi+W_s)\bigr)
  \Bigr)\,ds 
 \end{split}
\end{equation}
and
\begin{equation}\label{eq:SDEforZWithLearnablesPluggedIn}
\begin{split}
 Z_{\tau_2} & =  Z_{\tau_1} 
 + 
 \int_{\tau_1}^{\tau_2} 
 \bigl( ( \tfrac{\partial }{\partial t} \nabla_x u)(s,\xi + W_s) 
 + 
 \tfrac12 (\nabla_x\Delta_x u)(s,\xi + W_s)
 \bigr)
 \,ds
 \\
 &\quad + 
 \int_{\tau_1}^{\tau_2} 
 (\operatorname{Hess}_x u)(s,\xi+W_s)
 \,dW_s
 .
\end{split}
\end{equation}

\subsection{Forward-discretization of the merged PDE-2BSDE system}
\label{subsec:algo_forward_discretization}
In this subsection we derive a forward-discretization of the merged 
PDE-2BSDE system \eqref{eq:SDEforYWithLearnablesPluggedIn}--\eqref{eq:SDEforZWithLearnablesPluggedIn}. 
Let $t_0,t_1,\ldots,t_N \in [0,T]$ be real numbers with  
\begin{equation}
 0=t_0<t_1<\ldots<t_N=T. 
\end{equation}
such that the mesh size 
  $\sup_{0\leq k\leq N} (t_{k+1}-t_k)$
is sufficiently small. Note that 
\eqref{eq:SDEforYWithLearnablesPluggedIn} 
and 
\eqref{eq:SDEforZWithLearnablesPluggedIn}
suggest that for sufficiently large $N\in\N$ it holds 
for all $n\in\{0,1,\ldots,N-1\}$ that
\begin{equation}\label{eq:algorithm_Yapprox}
 \begin{split}
 Y_{t_{n+1}}
 &\approx
 Y_{t_n} 
 + 
 \Bigl(
 f\bigl(t_n,\xi + W_{t_n},Y_{t_n},Z_{t_n},(\operatorname{Hess}_x u)(t_n, \xi+W_{t_n})\bigr)
 \\
 &\quad+ 
 \tfrac12 \operatorname{Trace}\bigl((\operatorname{Hess}_x u)(t_n, \xi+W_{t_n}) \bigr)
 \Bigr) (t_{n+1}-t_n)
 + \langle Z_{t_n} , W_{t_{n+1}}-W_{t_n}\rangle_{\R^d}
\end{split}
 \end{equation}
 \text{and}\quad
\begin{equation}\label{eq:algorithm_Zapprox}
 \begin{split}
 Z_{t_{n+1}} 
 &\approx 
 Z_{t_n} 
 + 
 \bigl(
 (\tfrac{\partial }{\partial t}\nabla_x u)(t_n,\xi + W_{t_n}) 
 +
 \tfrac12 (\nabla_x\Delta_x u)(t_n,\xi + W_{t_n})
 \bigr)
 \,(t_{n+1}-t_n)  \\
 &\quad+
 (\operatorname{Hess}_x u)(t_n, \xi+W_{t_n} )
 \,(W_{t_{n+1}}-W_{t_n}).
\end{split}
\end{equation}

\subsection{Deep learning approximations}
\label{subsec:deep_learning_approximations}
In the next step we employ for every $n\in\{0,1,\ldots,N-1\}$ 
suitable approximations for the functions 
 \begin{equation}
 \R^d \ni x \mapsto
 (\operatorname{Hess}_x u)(t_n, x) \in \R^{d\times d}
 \end{equation}
 and
 \begin{equation}
 \R^d \ni x \mapsto
 (\tfrac{\partial }{\partial t}\nabla_x u)( t_n, x ) 
 +
 \tfrac12 ( \nabla_x\Delta_x u)(t_n,x)
 \in \R^d
 \end{equation}
in \eqref{eq:algorithm_Yapprox}--\eqref{eq:algorithm_Zapprox} 
but not for the functions
$
  \R^d \ni x \mapsto u(t_n,x) \in \R^d
$ 
and 
$ \R^d \ni x \mapsto (\nabla_x u)(t_n,x) \in \R^d
$
in 
\eqref{eq:definition_YZGA}. 
More precisely, 
let 
$ 
  \nu \in\N \cap [d+1,\infty)
$,
for every
$ \theta\in\R^{\nu}$, $n\in\{0,1,\ldots,N\}$
let 
  $\G^{\theta}_n\colon\R^d\to\R^{d\times d}$
  and
  $\A^{\theta}_n\colon\R^d\to\R^d$
be continuous functions, 
and 
for every $\theta = (\theta_1,\theta_2,\ldots,\theta_{\nu})\in\R^{\nu}$ let 
  $\Y^{\theta}\colon\{0,1,\ldots,N\}\times\Omega\to\R$
  and
  $\cZ^{\theta}\colon\{0,1,\ldots,N\}\times\Omega\to\R^d$  
be stochastic processes which satisfy that 
  $\Y^{\theta}_0 = \theta_1$ 
  and 
  $\cZ^{\theta}_0 = (\theta_2, \theta_3, \ldots,\theta_{d+1})$ 
and which satisfy for all $n\in\{0,1,\ldots,N-1\}$ that 
\begin{equation}\label{eq:deepLearningApprox_Y}
 \begin{split}
  &\Y^{\theta}_{n+1} = \Y^{\theta}_n 
  +
  \langle\cZ^{\theta}_n, W_{t_{n+1}}-W_{t_n}\rangle_{\R^d} \\
  &
  + \Bigl( 
  f\bigl(t_n,\xi+W_{t_n},\Y^{\theta}_n,\cZ^{\theta}_n,\G^{\theta}_n(\xi + W_{t_n})\bigr) 
  + \tfrac12 \operatorname{Trace}\bigl(\G^{\theta}_n(\xi + W_{t_n})\bigr)
  \Bigr)\,(t_{n+1}-t_n)  
\end{split}
\end{equation}
and
\begin{equation} \label{eq:deepLearningApprox_Z}
  \cZ^{\theta}_{n+1} = \cZ^{\theta}_n 
  + \A^{\theta}_n(\xi + W_{t_n})\,(t_{n+1}-t_n) 
  + \G^{\theta}_n(\xi + W_{t_n})\,(W_{t_{n+1}}-W_{t_n})
  .  
\end{equation}
For all suitable $ \theta\in\R^{\nu} $ and all
$n\in\{0,1,\ldots,N\}$ we think of 
$\Y^{\theta}_n\colon\Omega\to\R$ 
as an appropriate approximation 
\begin{equation}\label{eq:appropriateApproximationForY}
  \Y^{ \theta }_n \approx Y_{t_n}
\end{equation}
of 
$
  Y_{t_n}\colon \Omega\to\R
$, 
for all suitable $\theta\in\R^{\nu}$ and 
all $n\in\{0,1,\ldots,N\}$
we think of 
$\cZ^{\theta}_n\colon\Omega\to\R^d
$ 
as an appropriate approximation 
\begin{equation}
  \cZ^{ \theta }_n \approx Z_{ t_n }
\end{equation}
of 
$
  Z_{t_n}\colon \Omega\to\R^d
$,
for all suitable $\theta\in\R^{\nu}$, 
$ x \in \R^d $ 
and all
$ n \in\{0,1,\ldots,N-1\} $ 
we think of 
$
  \G^{\theta}_n(x)
  \in\R^{d\times d}
$ 
as an appropriate approximation 
\begin{equation}
  \G^{\theta}_n(x)
  \approx
  (\operatorname{Hess}_x u)(t_n,x)
\end{equation}
of 
$
  (\operatorname{Hess}_x u)(t_n,x)
  \in \R^{d\times d},
$
and 
for all suitable
  $\theta\in\R^{\nu}$, 
  $x\in\R^d$
  and all
  $ n \in \{0,1,\ldots, N - 1 \} $ 
we think of 
$
  \A^{\theta}_n(x)
  \in \R^d
$ 
as an appropriate approximation 
\begin{equation}
  \A^{\theta}_n(x)
  \approx
  (\tfrac{\partial }{\partial t}\nabla_x u)(t_n,x)+\tfrac12 (\nabla_x \Delta_x u)(t_n,x) 
\end{equation}
of 
$
  (\tfrac{\partial }{\partial t}\nabla_x u)(t_n,x)+\tfrac12 (\nabla_x \Delta_x u)(t_n,x)\in\R^d
$. In particular, we think of $\theta_1$ as an appropriate approximation 
\begin{equation}\label{eq:Theta1ApproxUofXi}
 \theta_1 \approx u(0,\xi)  
\end{equation}
of 
$
u(0,\xi) \in \R,
$ 
and we think of $(\theta_2,\theta_3,\ldots,\theta_{d+1})$ as an 
appropriate approximation 
\begin{equation}\label{eq:appropriateApproximationForNablaAt0}
(\theta_2,\theta_3,\ldots,\theta_{d+1}) \approx (\nabla_x u)(0,\xi)
\end{equation}
of 
$ 
(\nabla_x u)(0,\xi) \in \R^d
$.
We suggest for every $n\in\{0,1,\ldots,N-1\}$ to choose the functions 
$\G^{\theta}_n$
and 
$\A^{\theta}_n$ 
as deep neural networks (cf., e.g., \cite{Bengio2009,LeCun2015}). For example, 
for every $k\in\N$ let 
$ \mathcal{R}_{k} \colon \R^k \to \R^k $ be the function 
which satisfies for all $ x = ( x_1, \dots, x_k ) \in \R^k $ that
\begin{equation}
\label{eq:rectifier}
  \mathcal{R}_k( x ) 
  =
  \left(
    \max\{ x_1, 0 \}
    ,
    \dots
    ,
    \max\{ x_k, 0 \}
  \right)
  ,
\end{equation}
for every 
$ \theta = ( \theta_1, \dots, \theta_{ \nu } ) \in \R^{ \nu } $, 
$ v \in \N_0 = \{0\} \cup \N $,
$ k, l \in \N $
with 
$
  v + k (l + 1 ) \leq \nu
$
let 
$ M^{ \theta, v }_{ k, l } \colon \R^l \to \R^k $ 
be the affine linear function which satisfies for all 
$ x = ( x_1, \dots, x_l ) $ that
\begin{equation}
 M^{ \theta, v }_{ k, l }( x )
  =
  \left(
    \begin{array}{cccc}
      \theta_{ v + 1 }
    &
      \theta_{ v + 2 }
    &
      \dots
    &
      \theta_{ v + l }
    \\
      \theta_{ v + l + 1 }
    &
      \theta_{ v + l + 2 }
    &
      \dots
    &
      \theta_{ v + 2 l }
    \\
      \theta_{ v + 2 l + 1 }
    &
      \theta_{ v + 2 l + 2 }
    &
      \dots
    &
      \theta_{ v + 3 l }
    \\
      \vdots
    &
      \vdots
    &
      \vdots
    &
      \vdots
    \\
      \theta_{ v + ( k - 1 ) l + 1 }
    &
      \theta_{ v + ( k - 1 ) l + 2 }
    &
      \dots
    &
      \theta_{ v + k l }
    \end{array}
  \right)
  \left(
    \begin{array}{c}
      x_1
    \\
      x_2
    \\
      x_3
    \\
      \vdots 
    \\
      x_l
    \end{array}
  \right)
  +
  \left(
    \begin{array}{c}
      \theta_{ v + k l + 1 }
    \\
      \theta_{ v + k l + 2 }
    \\
      \theta_{ v + k l + 3 }
    \\
      \vdots 
    \\
      \theta_{ v + k l + k }
    \end{array}
  \right),
\end{equation}
assume 
that
$
  \nu \geq ( 5 N d + N d^2 + 1 ) (d + 1)
$,
and 
assume 
for all 
$ \theta \in \R^{ \nu } $,
$ n \in \{ m \in \N \colon m < N \} $,
$ x \in \R^d $
that
\begin{equation}
\label{eq:neural_network_for_a}
   \A^{ \theta }_n
    =
    M^{ \theta, [ ( 2N + n ) d + 1] ( d + 1 ) }_{ d, d } 
    \circ 
    \mathcal{R}_d
    \circ 
    M^{ \theta, [ ( N + n ) d + 1] ( d + 1 ) }_{ d, d } \\
    \circ 
    \mathcal{R}_d
    \circ 
    M^{ \theta, (n d + 1) ( d + 1 ) }_{ d, d } 
\end{equation}
and
\begin{equation}\label{eq:neural_network_for_gamma}
  \G^{ \theta }_n
    =
    M^{ \theta, ( 5 N d + n d^2 + 1 ) ( d + 1 ) }_{ d^2, d } 
    \circ 
    \mathcal{R}_d
    \circ 
    M^{ \theta, [ ( 4 N + n) d + 1 ] ( d + 1 ) }_{ d, d }
    \circ 
    \mathcal{R}_d
    \circ 
    M^{ \theta, [ ( 3 N + n) d  + 1 ] ( d + 1 ) }_{ d, d }. 
\end{equation}
The functions in 
\eqref{eq:neural_network_for_a} provide
artifical neural networks with $ 4 $ layers 
($ 1 $ input layer with $ d $ neurons, 
$ 2 $ hidden layers with $ d $ neurons each, 
and $ 1 $ output layer with $ d $ neurons)
and rectifier functions as activation functions
(see \eqref{eq:rectifier}). 
The functions in 
\eqref{eq:neural_network_for_gamma} also provide 
artificial neural networks with $ 4 $ layers
($ 1 $ input layer with $ d $ neurons, 
$ 2 $ hidden layers with $ d $ neurons each, 
and $ 1 $ output layer with $ d^2 $ neurons)
and rectifier functions as activation functions 
(see \eqref{eq:rectifier}).

\subsection{Stochastic gradient descent-type optimization}
\label{subsec:algo_sgd}
We intend to reach 
a \emph{suitable} $\theta\in\R^{\nu}$ in 
\eqref{eq:appropriateApproximationForY}--\eqref{eq:appropriateApproximationForNablaAt0} 
by applying a stochastic gradient descent-type minimization algorithm to the function 
\begin{align}\label{eq:loss_function_mean_square}
 \R^{\nu}\ni\theta\mapsto
 \E\big[
   | \Y^{\theta}_N-g(\xi + W_{t_N})|^2
 \big] \in \R.
\end{align}
Minimizing the function in \eqref{eq:loss_function_mean_square} 
is inspired by the fact that 
\begin{equation}
 \E[|Y_T - g(\xi + W_T)|^2] = 0
\end{equation}
according to 
\eqref{eq:SDEsforYAndZ_everythingPastedIn:Tag1}. 
Applying a stochastic gradient descent-type minimization 
algorithm yields under suitable assumptions random approximations 
\begin{equation}
 \Theta_m = (\Theta_m^{(1)},\Theta_m^{(2)},\ldots,\Theta_m^{(\nu)})
 \colon\Omega\to\R^{\nu}
\end{equation}
for $m\in\N_0$ of a local minimum point of the function in 
\eqref{eq:loss_function_mean_square}. For sufficiently large 
$N, \nu, m\in\N$ we 
use the random variable $\Theta_m^{(1)}\colon\Omega\to\R$ as an 
appropriate approximation 
\begin{equation}
 \Theta_m^{(1)} \approx u(0,\xi)
\end{equation}
of $u(0,\xi)\in\R$ (cf. \eqref{eq:Theta1ApproxUofXi} above). 
In the next subsection the proposed algorithm is described in 
more detail. 

\subsection{Framework for the algorithm in a specific case}
 \label{subsec:specific_case}

In this subsection we describe the 
deep 2BSDE method in the specific case where 
\eqref{eq:fullyNonlinearPDETerminalValueProblem} 
is the PDE under consideration, 
where the standard Euler-Maruyama scheme 
(cf., e.g., \cite{KloedenPlaten1992,Maruyama1955,MilsteinOriginal1974}) is the employed approximation scheme 
for discretizing
\eqref{eq:SDEforYWithLearnablesPluggedIn}
and
\eqref{eq:SDEforZWithLearnablesPluggedIn}
(cf.\ \eqref{eq:deepLearningApprox_Y}
and 
\eqref{eq:deepLearningApprox_Z}),
and where the plain stochastic gradient 
descent with constant learning rate 
$\gamma\in (0,\infty)$ is the employed minimization algorithm. 
A more general description of the deep 2BSDE method, 
which allows to incorporate more sophisticated machine learning approximation techniques such as 
batch normalization or the Adam optimizer, can be found in 
Subsection~\ref{subsec:general_case} below.

\begin{algo}[Special case]
\label{def:specific_case}
Let 
  $T,\gamma\in (0,\infty)$, 
  $d,N\in\N$, 
  $\nu\in \N\cap [d+1,\infty)$,
  $\xi\in\R^d$, 
let
  $f 
  \colon 
  [0,T] \times \R^d \times \R \times \R^d \times \R^{d\times d}
  \to
  \R$ 
  and 
  $g \colon \R^d \to \R$ 
be functions,
let 
$
  ( \Omega, \F, \P )
$ be a probability space, 
let
$
  W^{m}\colon[0,T]\times\Omega\to\R^d
$, 
$ m \in \N_0 $, 
be independent standard Brownian motions on 
  $(\Omega,\F,\P)$,
let $t_0,t_1,\ldots,t_N\in [0,T]$ 
be real numbers with 
  $0=t_0<t_1<\ldots<t_N=T$, 
for every
  $\theta\in\R^{\nu}$, 
  $n\in\{0,1,\ldots,N-1\}$
let
  $\A_n^{\theta}\colon\R^d\to\R^d$ and 
  $\G_n^{\theta}\colon\R^d\to\R^{d\times d}$ be functions, 
for every
  $ m \in \N_0$, 
  $ \theta \in \R^{\nu}$
let
  $
  \Y^{ m, \theta }
  \colon 
  \{0 , 1 , \ldots , N \} \times \Omega \to \R
  $
  and 
  $
  \cZ^{ m , \theta }
  \colon 
  \{ 0 , 1 , \ldots , N \} \times \Omega \to \R^d
  $
be stochastic processes which satisfy 
that 
  $ \Y^{ m , \theta }_0 = \theta_1$ and
  $ \cZ^{ m , \theta }_0 = (\theta_2,\theta_3,\ldots,\theta_{d+1}) $
and which satisfy for all $n\in\{0,1,\ldots,N-1\}$ 
that 
\begin{equation}\label{eq:euler_maruyamaY} 
 \begin{split}
 & \Y^{ m , \theta }_{ n+1 }
 = 
 \Y^{ m , \theta }_{ n } 
 + 
 \langle 
    \cZ^{ m , \theta }_{ n } , W^{ m }_{ t_{n+1} } - W^{ m }_{ t_n } 
  \rangle_{\R^d}  
 \\  
 & + 
 \Bigl( 
  f \bigl( 
    t_n, 
    \xi + W^m_{t_n}, 
    \Y^{ m , \theta }_n , 
    \cZ^{ m , \theta }_n ,
    \G^{ \theta }_n (\xi+W^m_{t_n})
    \bigr) 
  + 
  \tfrac12 \operatorname{Trace}
  \bigl( 
   \G^{ \theta }_n( \xi + W^{ m }_{t_n} ) 
  \bigr)
 \Bigr) 
 ( t_{n+1} - t_n )
\end{split}
\end{equation}
\begin{equation}\label{eq:euler_maruyamaZ}
 \text{and}\qquad\cZ^{ m , \theta }_{ n+1 }
 =
 \cZ^{ m , \theta }_{ n }
 + 
 \A^{ \theta }_{ n } ( \xi + W^{ m }_{ t_n } )\,
 ( t_{ n+1 } - t_{ n } )
 + 
 \G^{ \theta }_{ n }( \xi + W^{ m }_{t_n} )\,
 ( W^{ m }_{ t_{n+1} } - W^{ m }_{ t_n } ), 
\end{equation}
for every $m\in\N_0$ let
  $\phi^{m}\colon \R^{\nu}\times\Omega\to\R$
be the function which satisfies 
  for all $\theta \in \R^{ \nu } $, $\omega\in\Omega$ 
that
  \begin{align}\label{eq:specific_one_sample_of_function_to_minimize}
   \phi^{m}(\theta,\omega) 
   = 
   \bigl|\Y^{m,\theta}_N(\omega)-g\bigl(\xi+W^m_T(\omega)\bigr)\bigr|^2, 
  \end{align}
for every $m\in\N_0$ let
  $\Phi^{m}\colon \R^{\nu}\times\Omega\to\R^{\nu}$
be a function which satisfies for all 
  $\omega\in\Omega$,
  $\theta\in\{\eta\in\R^{\nu}\colon
  ~\phi^{m}(\cdot,\omega)\colon\R^{\nu}\to\R
  ~\text{is differentiable at}~
  \eta\}$
that
  \begin{align}\label{eq:specific_gradient}
   \Phi^{m}(\theta,\omega)
   =
   (\nabla_{\theta}\phi^{m})(\theta,\omega), 
  \end{align}
and let 
  $ \Theta = ( \Theta^{ (1) }, \dots, \Theta^{ (\nu) } ) \colon\N_0\times\Omega\to\R^{\nu}$ 
be a stochastic process which satisfies for all $m\in\N_0$ that
  \begin{align}\label{eq:sgd_step}
   \Theta_{m+1} = \Theta_m - \gamma\cdot\Phi^{m}(\Theta_m). 
  \end{align}
\end{algo}

Under suitable further assumptions, 
we think in the case of sufficiently large  
$N, \nu\in\N$, $m\in\N_0$
and sufficiently small $\gamma\in (0,\infty)$ 
in Framework~\ref{def:specific_case} 
of the random variable 
$\Theta_m = (\Theta_m^{(1)},\ldots,\Theta_m^{(\nu)})\colon\Omega\to\R^{\nu}$ 
as an appropriate approximation of a local minimum point 
of the expected loss function and 
we think in the case of sufficiently large 
$N, \nu\in\N$, $m\in\N_0$ 
and sufficiently small $\gamma\in (0,\infty)$ 
in Framework~\ref{def:specific_case} 
of the random variable 
$\Theta_m^{(1)}\colon\Omega\to\R$ 
as an appropriate approximation of 
the value 
$
  u(0,\xi)\in\R
$
where $u\colon [0,T]\times \R^d\to\R$ is 
an at most polynomially growing continuous function which 
satisfies for all $(t,x)\in [0,T)\times\R^d$ that 
  $u|_{[0,T)\times\R^d}\in C^{1,2}([0,T)\times\R^d,\R)$,
  $u(T,x)=g(x)$, 
  and 
\begin{align}\label{eq:specific_fullyNonlinearPDE}
 \tfrac{\partial u}{\partial t}(t,x) 
 = 
 f\bigl(t, x, u(t,x),(\nabla_x u)(t,x),(\operatorname{Hess}_x u)(t,x)\bigr).
\end{align}
In Subsection~\ref{subsec:example_allen_cahn_plain_sgd_no_bn} 
below an implementation of Framework~\ref{def:specific_case} 
(see {\sc Python} code~\ref{code:deepPDEmethodPlainSGDNoBN} 
in Appendix \ref{subsec:plainSGDCode} below) is 
employed to calculate numerical approximations for 
the Allen-Cahn equation in $20$ space-dimensions 
($d=20$). In Subsection~\ref{subsec:example_allen_cahn} 
below numerical approximations for the 
Allen-Cahn equation in $50$ space-dimensions 
are calculated by means of the algorithm 
in the more general setting in Framework 
\ref{def:general_algorithm} below.

\section{The deep 2BSDE method in the general case}\label{sec:details}
 In this section we extend and generalize the approximation scheme derived and presented in 
Section~\ref{sec:algorithm_main_ideas}. 
The core idea of the approximation scheme in this section remains the same as in 
Section~\ref{sec:algorithm_main_ideas}
but, in contrast to Section~\ref{sec:algorithm_main_ideas}, 
in this section the background dynamics in the approximation scheme may be 
a more general It\^o process than just a Brownian motion
(cf.\ Lemma~\ref{lem:applying_ito} in 
Subsection~\ref{subsec:details_connection} below) 
and, in contrast to Section~\ref{sec:algorithm_main_ideas}, in this section the approximation scheme may employ
more sophisticated machine learning 
techniques (cf.\ Framework~\ref{def:general_algorithm} in 
Subsection~\ref{subsec:general_case} below). 

\subsection{Fully nonlinear second-order PDEs}
\label{subsec:details_fullyNonlinear}
Let 
  $d\in\N$, 
  $T\in (0,\infty)$, 
let $u=(u(t,x))_{t\in [0,T],x\in\R^d}\in C^{1,2}([0,T]\times\R^d,\R)$, 
$f\colon [0,T]\times\R^d\times\R\times\R^d\times\R^{d\times d}\to\R$, 
and $g\colon [0,T]\times\R^d\to\R$ be functions which satisfy for all 
$(t,x)\in [0,T)\times\R^d$ that 
  $u(T,x) = g(x)$ 
  and 
  \begin{align}\label{eq:notes_fully_nonlinear_pde}
   \tfrac{\partial u}{\partial t}(t,x) 
   = 
   f
   \bigl( 
    t, x, u(t,x), (\nabla_x u)(t,x), (\operatorname{Hess}_x u)(t,x) 
   \bigr).
  \end{align}
Our goal is to approximatively compute the solution $ u $ of the PDE~\eqref{eq:notes_fully_nonlinear_pde}
at time $ t = 0 $, that is, our goal is to approximatively calculate the function 
$
  \R^d \ni x \mapsto u( 0, x ) \in \R
$.
For this, 
we make use of the following connection between fully nonlinear second-order PDEs 
and second-order BSDEs. 

\subsection{Connection between fully nonlinear second-order PDEs and 2BSDEs}
\label{subsec:details_connection}

The deep 2BSDE method relies on a 
connection between fully nonlinear second-order PDEs and 
second-order BSDEs; cf., e.g., Theorem~4.10 in Cheridito et al.~\cite{CheriditoSonerTouziVictoir2007} 
and Lemma~\ref{lem:applying_ito} below.

\begin{lemma}[Cf., e.g., Section~3 in Cheridito et al.~\cite{CheriditoSonerTouziVictoir2007}]
\label{lem:applying_ito}
Let 
  $d\in\N$, 
  $T\in (0,\infty)$, 
let
  $u=(u(t,x))_{t\in [0,T], x\in\R^d}\in C^{1,2}([0,T]\times\R^d,\R)$, 
  $\mu\in C(\R^d,\R^d)$,
  $\sigma\in C(\R^d,\R^{d\times d})$, 
  $f\colon [0,T]\times\R^d\times\R\times\R^d\times\R^{d\times d}\to\R$,  
  and $g\colon \R^d\to\R$ 
  be functions which satisfy for all $t\in [0,T)$, $x\in\R^d$ 
  that 
  $\nabla_x u\in C^{1,2}([0,T]\times\R^d,\R^d)$, 
  $u(T,x) = g(x)$, 
  and 
  \begin{align}\label{eq:lemma_fullyNonlinearPDETerminalValueProblem}
   \tfrac{\partial u}{\partial t}(t,x) 
    & = 
   f\bigl(t,x,u(t,x),(\nabla_x u)(t,x),(\operatorname{Hess}_x u)(t,x)\bigr), 
  \end{align}
let $(\Omega,\F,\P)$ 
  be a probability space,  
let  
  $ W = ( W^{ (1) }, \dots, W^{ (d) } ) \colon [0,T]\times\Omega\to\R^d$ 
  be a standard Brownian motion on $(\Omega,\F,\P)$, 
let 
  $\bF = (\bF_t)_{t\in [0,T]}$ 
  be the normal filtration on $(\Omega,\F,\P)$ 
  generated by $W$, 
let 
  $\xi\colon\Omega\to\R^d$ be an $\F_0/\B(\R^d)$-measurable function, 
let 
  $X=(X^{(1)},\ldots,X^{(d)})\colon[0,T]\times\Omega\to\R^d$ 
  be an $\bF$-adapted stochastic process with continuous 
  sample paths which satisfies that for all $t\in [0,T]$ 
  it holds $\P$-a.s.~that 
  \begin{align}\label{eq:lemma_SDEforX}
   X_t = \xi + \int_0^t \mu(X_s)\,ds + \int_0^t \sigma(X_s)\,dW_s,  
  \end{align}
for every $\varphi\in C^{1,2}([0,T]\times\R^d,\R)$ 
  let $\cL\varphi\colon[0,T]\times\R^d\to\R$ be the function which 
  satisfies for all $(t,x)\in [0,T]\times\R^d$ that 
  \begin{align}\label{eq:dynkinOperator_definition}
   (\cL\varphi)(t, x) = 
   (\tfrac{\partial\varphi}{\partial t})(t,x) 
   + 
   \tfrac12\operatorname{Trace}\bigl(
    \sigma(x)\sigma(x)^{\ast}(\operatorname{Hess}_x\varphi)(t,x)
    \bigr), 
  \end{align}
and let 
  $Y\colon [0,T]\times\Omega\to\R$, 
  $Z=(Z^{(1)},\ldots,Z^{(d)})\colon [0,T]\times\Omega\to\R^d$, 
  $\Gamma=(\Gamma^{(i,j)})_{(i,j)\in\{1,\ldots,d\}^2}\colon [0,T]\times\Omega\to\R^{d\times d}$, 
  and 
  $A=(A^{(1)},\ldots,A^{(d)})\colon [0,T]\times\Omega\to\R^d$ 
  be the stochastic processes which satisfy for 
  all $t\in [0,T]$, $i\in\{1,2,\ldots,d\}$ that 
  \begin{align}\label{eq:lemma_YZGAdefinition}
   Y_t = u(t,X_t),\quad 
   Z_t=(\nabla_x u)(t,X_t),\quad
   \Gamma_t = (\operatorname{Hess}_x u)(t,X_t),\quad
   A_t^{(i)}=(\cL(\tfrac{\partial u}{\partial x_i}))(t,X_t). 
  \end{align}
%
  Then it holds that $Y$, $Z$, $\Gamma$, and $A$ are $\bF$-adapted stochastic 
  processes with continuous sample paths which satisfy  
  that for all $t\in [0,T]$ it holds $\P$-a.s.~that 
  \begin{equation}\label{eq:lemma_SDEforY}
   Y_t = 
   g(X_T) - 
   \int_t^T
     \bigl( 
     f(s,X_s, Y_s, Z_s, \Gamma_s) + \tfrac12\operatorname{Trace}(\sigma(X_s)\sigma(X_s)^{\ast}\Gamma_s)
     \bigr)\,ds 
   - 
   \int_t^T
     \langle Z_s, dX_s\rangle_{\R^d}
  \end{equation}
  and
  \begin{equation}\label{eq:lemma_SDEforZ} 
   Z_t =  Z_0 + 
   \int_0^t A_s\,ds 
   + 
   \int_0^t \Gamma_s\,dX_s. 
  \end{equation}
\end{lemma}

\begin{proof}[Proof of Lemma \ref{lem:applying_ito}]
 Note that 
 $u\colon [0,T]\times\R^d\to\R$, 
 $\nabla_x u\colon [0,T]\times\R^d\to\R^d$, 
 $\operatorname{Hess}_x u\colon [0,T]\times\R^d\to\R^{d\times d}$, 
 and 
 $\bigl(\cL\tfrac{\partial u}{\partial x_i}\bigr)\colon [0,T]\times\R^d\to\R$, $i\in\{1,2,\ldots,d\}$, 
 are continuous functions. Combining this and \eqref{eq:lemma_YZGAdefinition} 
 with the continuity of the  sample paths of $X$ shows  
 that $Y$, $Z$, $\Gamma$, and $A$ are $\bF$-adapted 
 stochastic process with continuous sample paths. 
 Next observe that It\^o's lemma and the assumption that 
 $u\in C^{1,2}([0,T]\times\R^d,\R)$ 
 yield that  
 for all $ r \in [0,T]$ it holds $\P$-a.s.~that 
 \begin{equation}\label{eq:SDEforY}
  \begin{split}
 u(T,X_T) &= u(r,X_{r})
 +
 \int_{r}^T \langle(\nabla_x u)(s, X_s), dX_s\rangle_{\R^d} \\
 &\quad
 + \int_{r}^T\Bigl( 
  (\tfrac{\partial u}{\partial t})(s,X_s) 
  + \tfrac12 \operatorname{Trace}
  \bigl(
   \sigma(X_s)\sigma(X_s)^{\ast}(\operatorname{Hess}_x u)(s,X_s)
  \bigr)
  \Bigr)
 \,ds.
  \end{split}
 \end{equation}
 This, 
 \eqref{eq:lemma_fullyNonlinearPDETerminalValueProblem}, 
 and 
 \eqref{eq:lemma_YZGAdefinition}
 yield that for all $ r \in [0,T] $ it holds $\P$-a.s. that
 \begin{equation}
 \begin{split}
  g(X_T) & = 
  Y_{ r } + \int_{ r }^T 
  \langle Z_s, dX_s\rangle_{\R^d}
  \\  
  & \quad   + 
  \int_{ r }^T
  \Bigl( 
  f(s,X_s,Y_s,Z_s,\Gamma_s) 
  + \tfrac12 \operatorname{Trace}
  \bigl(
   \sigma(X_s)\sigma(X_s)^{\ast}\Gamma_s
  \bigr)
  \Bigr)
 \,ds. 
 \end{split}
 \end{equation}
 This establishes \eqref{eq:lemma_SDEforY}. 
 In the next step we note that It\^o's lemma and 
 the hypothesis that $\nabla_x u = ( \frac{ \partial u }{ \partial x_1 } , \dots, \frac{ \partial u }{ \partial x_d } ) \in C^{1,2}([0,T]\times\R^d,\R^d)$
 guarantee that
 for all $i\in\{1,2,\ldots,d\}$, $ r \in [0,T]$ it holds $\P$-a.s.~that
 \begin{equation}\label{eq:prooflem_applyItoToZi}
  \begin{split}
  ( \tfrac{ \partial u }{ \partial x_i } )( r, X_r )
  & = 
  ( \tfrac{ \partial u }{ \partial x_i } )( 0, X_0 )
  + \int_0^{ r } \langle (\nabla_x\tfrac{\partial u}{\partial x_i})
  (s,X_s),\,dX_s\rangle_{\R^d} 
\\
  & \quad + \int_0^{ r } \Bigl( 
  (\tfrac{\partial}{\partial t}\tfrac{\partial u}{\partial x_i})
  (s,X_s) 
  + \tfrac12 \operatorname{Trace}\bigl( 
  \sigma(X_s)\sigma(X_s)^{\ast}(\operatorname{Hess}_x \tfrac{\partial u}{\partial x_i})(s, X_s)\bigr)
  \Bigr)
 \,ds.
  \end{split}
 \end{equation}
This, \eqref{eq:dynkinOperator_definition}, and \eqref{eq:lemma_YZGAdefinition} 
yield that for all 
$i\in\{1,2,\ldots,d\}$, $ r \in [0,T]$ it holds $\P$-a.s.~that
\begin{equation}
 \begin{split}
  Z^{(i)}_{r} 
& =
  Z^{(i)}_0 
  +
  \sum_{j=1}^d 
  \int_0^{r}  (\tfrac{\partial}{\partial x_j}\tfrac{\partial u}{\partial x_i})(s,X_s)\,dX^{(j)}_s 
  \\
&
\quad 
  + 
  \int_0^{r} \Bigl( 
  (\tfrac{\partial}{\partial t}\tfrac{\partial u}{\partial x_i})(s,X_s) 
  + \tfrac12 \operatorname{Trace}\bigl( 
  \sigma(X_s)\sigma(X_s)^{\ast}
  (\operatorname{Hess}_x\tfrac{\partial u}{\partial x_i})(s, X_s)\bigr)
  \Bigr)
 \,ds \\
  &= 
  Z^{(i)}_0 
  + \int_0^{r}  
  A^{(i)}_s \,ds
  + \sum_{j=1}^d
  \int_0^{r} \Gamma^{(i,j)}_s \,dX^{(j)}_s
  .
 \end{split}
\end{equation}
This shows \eqref{eq:lemma_SDEforZ}. The proof of Lemma~\ref{lem:applying_ito} is thus completed. 
\end{proof}

In Subsection~\ref{subsec:algorithm_connection_pde_bsde} above we have employed
Lemma~\ref{lem:applying_ito} 
in the specific situation where 
$ \forall \, x \in \R^d \colon \mu(x) = 0 \in \R^d $
and 
$
  \forall \, x \in \R^d \colon \sigma(x) = \operatorname{Id}_{ \R^d } \in \R^{ d \times d }
$
(cf.\ \eqref{eq:SDEsforYAndZ_everythingPastedIn:Tag1}--\eqref{eq:DefinitionA} 
in Subsection~\ref{subsec:algorithm_connection_pde_bsde}
and \eqref{eq:lemma_SDEforY}--\eqref{eq:lemma_SDEforZ} 
in Lemma~\ref{lem:applying_ito}). 
In the following we proceed 
with the merged formulation, the forward-discretization 
of the merged PDE-2BSDE system, and deep learning approximations
similar as in Section~\ref{sec:algorithm_main_ideas}.

\subsection{Merged formulation of the PDE and the 2BSDE}

In this subsection we derive a merged formulation
(see \eqref{eq:mergedForm_SDEforYPluggedIn} and \eqref{eq:mergedForm_SDEforZPluggedIn})
for the PDE~\eqref{eq:notes_fully_nonlinear_pde} and 
the 2BSDE system \eqref{eq:SDEforYt1t2_general}--\eqref{eq:SDEforZt1t2_general} 
as in Subsection~\ref{subsec:algo_mergedFormulation}. 
To derive the merged formulation of the PDE and the 2BSDE, we 
employ the following hypotheses in addition to the assumptions 
in Subsection~\ref{subsec:details_fullyNonlinear} above (cf.\ Lemma~\ref{lem:applying_ito} above).
Let 
$
  \mu \colon \R^d \to \R^d
$ 
and
$
  \sigma \colon \R^d \to \R^{d\times d}
$ 
be continuous functions, 
let $(\Omega,\F,\P)$ 
  be a probability space,  
let  
  $W\colon [0,T]\times\Omega\to\R^d$ 
  be a standard Brownian motion on $(\Omega,\F,\P)$, 
let 
  $\bF = (\bF_t)_{t\in [0,T]}$ 
  be the normal filtration on $(\Omega,\F,\P)$ 
  generated by $W$, 
let 
  $\xi\colon\Omega\to\R^d$ be an $\F_0/\B(\R^d)$-measurable 
  function, 
let 
  $X=(X^{(1)},\ldots,X^{(d)})\colon[0,T]\times\Omega\to\R^d$ 
  be an $\bF$-adapted stochastic process with continuous 
  sample paths which satisfies that for all $t\in [0,T]$ 
  it holds $\P$-a.s.~that 
  \begin{align}\label{eq:mergedForm_SDEforX}
   X_t = \xi + \int_0^t \mu(X_s)\,ds + \int_0^t \sigma(X_s)\,dW_s,  
  \end{align}
let $ e^{ (d) }_1 = ( 1, 0, \dots, 0 ) $,
$ e^{ (d) }_2 = ( 0, 1, 0, \dots, 0 ) $,
$ \dots $,
$ e^{ (d) }_d = ( 0, \dots, 0, 1 ) \in \R^d $ be the 
standard basis vectors of $\R^d$,
for every $\varphi\in C^{1,2}([0,T]\times\R^d,\R^d)$ 
  let $\cL\varphi\colon[0,T]\times\R^d\to\R^d$ be the function which 
  satisfies for all $(t,x)\in [0,T]\times\R^d$ that 
  \begin{equation}
   (\cL\varphi)(t, x) 
   = 
   \tfrac{\partial\varphi}{\partial t}(t,x) 
   +
   \tfrac{ 1 }{ 2 }
   \textstyle
   \sum\limits_{i=1}^d
   \big( \tfrac{ \partial^2 \varphi }{\partial x^2} \big)(t,x) 
   \big(
     \sigma(x) e^{ (d) }_i ,
     \sigma(x) e^{ (d) }_i 
   \big)
  \end{equation}
and let 
  $Y\colon [0,T]\times\Omega\to\R$, 
  $Z\colon [0,T]\times\Omega\to\R^d$, 
  $\Gamma\colon [0,T]\times\Omega\to\R^{d\times d}$, 
  and 
  $A\colon [0,T]\times\Omega\to\R^d$ 
  be the stochastic processes which satisfy for 
  all $t\in [0,T]$ that 
  \begin{align}\label{eq:mergedForm_YZGAdefinition}
   Y_t = u(t,X_t),\quad 
   Z_t=(\nabla_x u)(t,X_t),\quad
   \Gamma_t = (\operatorname{Hess}_x u)(t,X_t),\quad
   A_t=(\cL(\nabla_x u))(t,X_t). 
  \end{align}
Lemma \ref{lem:applying_ito} implies that 
for all $\tau_1,\tau_2\in [0,T]$ 
with $\tau_1\leq \tau_2$ 
it holds
$\P$-a.s.~that 
\begin{equation}
 X_{\tau_2} = X_{\tau_1} 
 + \int_{\tau_1}^{\tau_2} 
 \mu(X_s)\,ds 
 + \int_{\tau_1}^{\tau_2}
 \sigma(X_s)\,dW_s, \label{eq:notes_SDEforX_t1t2}
\end{equation}
\begin{equation}\label{eq:SDEforYt1t2_general}
 Y_{\tau_2} = Y_{\tau_1} 
 + \int_{\tau_1}^{\tau_2} 
 \bigl( 
  f(s,X_s,Y_s,Z_s,\Gamma_s) + \tfrac12\text{Trace}(\sigma(X_s)\sigma(X_s)^{\ast}\Gamma_s)
  \bigr)\,ds 
 + 
 \int_{\tau_1}^{\tau_2}
  \langle Z_s, dX_s\rangle_{\R^d}, 
\end{equation}
and
\begin{equation}\label{eq:SDEforZt1t2_general}
 Z_{\tau_2} =  Z_{\tau_1} + 
 \int_{\tau_1}^{\tau_2} A_s\,ds 
 + 
 \int_{\tau_1}^{\tau_2 } \Gamma_s\,dX_s. 
\end{equation}
Putting the third and the fourth identity in \eqref{eq:mergedForm_YZGAdefinition} into 
\eqref{eq:SDEforYt1t2_general} and 
\eqref{eq:SDEforZt1t2_general} yields 
that for 
all 
$ \tau_1, \tau_2 \in [0,T] 
$ 
with $\tau_1\leq \tau_2$ 
it holds $\P$-a.s.~that
\begin{equation}\label{eq:mergedForm_SDEforYPluggedIn}
\begin{split}
 &Y_{\tau_2} = Y_{\tau_1} 
 + 
 \int_{\tau_1}^{\tau_2}
  \langle Z_s, dX_s\rangle_{\R^d} \\
 & 
 + \int_{\tau_1}^{\tau_2} 
 \Bigl( 
  f\bigl(s,X_s,Y_s,Z_s,(\operatorname{Hess}_x u)(s,X_s)\bigr) 
  + 
  \tfrac12\text{Trace}\bigl(\sigma(X_s)\sigma(X_s)^{\ast}(\operatorname{Hess}_x u)(s,X_s)\bigr)
  \Bigr)\,ds 
  \end{split}
\end{equation}
and 
\begin{equation}\label{eq:mergedForm_SDEforZPluggedIn}
 Z_{\tau_2} =  Z_{\tau_1} + 
 \int_{\tau_1}^{\tau_2} \bigl(\cL(\nabla_x u)\bigr)(s,X_s)\,ds 
 + 
 \int_{\tau_1}^{\tau_2 } (\operatorname{Hess}_x u)(s,X_s)\,dX_s. 
\end{equation}

\subsection{Forward-discretization of the merged PDE-2BSDE system}
In this subsection we derive a forward-discretization 
of the merged PDE-2BSDE system 
\eqref{eq:mergedForm_SDEforYPluggedIn}--\eqref{eq:mergedForm_SDEforZPluggedIn} 
(cf.\ Subsection~\ref{subsec:algo_forward_discretization}). 
Let $t_0,t_1,\ldots,t_N\in [0,T]$ be real numbers with 
\begin{equation}
 0=t_0<t_1<\ldots<t_N=T
\end{equation}
such that the small mesh size 
  $\sup_{0\leq k\leq N} (t_{k+1}-t_k)$
is sufficiently small. Note that 
\eqref{eq:mergedForm_YZGAdefinition},
\eqref{eq:notes_SDEforX_t1t2}, 
\eqref{eq:mergedForm_SDEforYPluggedIn}, 
and 
\eqref{eq:mergedForm_SDEforZPluggedIn} 
suggest that 
for sufficiently large $N\in\N$ it holds 
for all $n\in\{0,1,\ldots,N-1\}$ that
\begin{equation}
\label{eq:initial_condition}
  X_{ t_0 } 
  = X_0 
  = \xi 
  ,
  \qquad
  Y_{ t_0 } 
  = Y_0 
  = u( 0, \xi )
  ,
  \qquad
  Z_{ t_0 } 
  = Z_0 
  = ( \nabla_x u )( 0, \xi )
  ,
\end{equation}
\begin{equation}
\label{eq:X_approximate_derivation}
 X_{t_{n+1}} 
 \approx 
 X_{t_n} 
 + 
 \mu(X_{t_n})\,(t_{n+1}-t_n) 
 + 
 \sigma(X_{t_n})\,(X_{t_{n+1}}-X_{t_n}), 
\end{equation}
\begin{equation}
\label{eq:notes_Yapprox}
\begin{split}
 &Y_{t_{n+1}}
 \approx
 Y_{t_n}
 + 
 \Bigl(
 f\bigl(t_n,X_{t_n},Y_{t_n},Z_{t_n},(\operatorname{Hess}_x u)(t_n,X_{t_n})\bigr) 
 \\
 & 
 + 
 \tfrac12 \text{Trace}\bigl(\sigma(X_{t_n})\sigma(X_{t_n})^{\ast}(\operatorname{Hess}_x u)(t_n,X_{t_n})\bigr)
 \Bigr)\,(t_{n+1}-t_n) 
 + \langle Z_{t_n} , X_{t_{n+1}}-X_{t_n}\rangle_{\R^d}  
 , 
\end{split}
\end{equation}
and
\begin{equation}\label{eq:notes_Zapprox}
 Z_{t_{n+1}} 
 \approx 
 Z_{t_n} + 
 \bigl(\cL(\nabla_x u)\bigr)(t_n,X_{t_n}) \,(t_{n+1}-t_n) 
 + 
 (\operatorname{Hess}_x u)(t_n,X_{t_n}) \,(X_{t_{n+1}}-X_{t_n})
\end{equation}
(cf.\ 
\eqref{eq:algorithm_Yapprox}--\eqref{eq:algorithm_Zapprox} in Subsection~\ref{subsec:algo_forward_discretization} above).

\subsection{Deep learning approximations}
\label{subsec:details_deepandsgd}
In the next step we employ 
suitable approximations for the functions 
\begin{equation}
 \R^d \ni x \mapsto u(0,x) \in \R
 \qquad\text{and}\qquad
 \R^d \ni x \mapsto (\nabla_x u)(0,x) \in \R^d
\end{equation}
in \eqref{eq:initial_condition}
and we employ 
for every $n\in\{0,1,\ldots,N-1\}$ 
suitable approximations for the functions 
\begin{equation}
 \R^d \ni x \mapsto 
 (\operatorname{Hess}_x u)(t_n,x)\in\R^{d\times d}
\qquad\text{and}\qquad 
 \R^d \ni x \mapsto 
 \bigl(\cL(\nabla_x u)\bigr)(t_n,x)
 \in \R^d 
\end{equation}
in \eqref{eq:notes_Yapprox}--\eqref{eq:notes_Zapprox}.
However, 
we do neither employ approximations for the functions  
$\R^d\ni x\mapsto u(t_n,x)\in\R$, $ n \in \{1,2,\ldots,N-1\} $, 
nor for the functions
$\R^d\ni x\mapsto (\nabla_x u)(t_n,x)\in\R^d$, $n\in\{1,2,\ldots,N-1\}$. 
More formally, 
let 
  $\X\colon\{0,1,\ldots,N\}\times\Omega\to \R^d$
  be a stochastic process 
which satisfies for all $ n \in \{ 0, 1, \ldots, N - 1 \} $ that 
$\X_0 = \xi$
and
\begin{equation}\label{eq:notes_Xapprox}
 \X_{n+1} 
 = 
 \X_n 
 + \mu(\X_n)\,(t_{n+1}-t_n) 
 + \sigma(\X_n)\,(\X_{n+1}-\X_n) , 
\end{equation}  
let 
  $\nu\in\N$, 
for every
$ \theta\in\R^{\nu} $
let
  $\U^{\theta}\colon\R^d\to\R$ and 
  $\V^{\theta}\colon\R^d\to\R^d$ 
  be continuous functions, 
for every
  $\theta\in\R^{\nu}$, $n\in\{0,1,\ldots,N-1\}$
let
  $\G_n^{\theta}\colon\R^d\to\R^{d\times d}$ 
  and 
  $\A_n^{\theta}\colon\R^d\to\R^d$
  be continuous functions, 
and 
for every 
$ \theta=(\theta_1,\ldots,\theta_{\nu})\in\R^{\nu}$ 
let 
$\Y^{\theta}\colon\{0,1,\ldots,N\}\times\Omega\to \R$
and 
$\cZ^{\theta}\colon\{0,1,\ldots,N\}\times\Omega\to \R^d$ 
be stochastic processes 
which satisfy 
  $\Y_0^{\theta} = \U^{\theta}(\xi)$ and 
  $\cZ^{\theta}_0 = \V^{\theta}(\xi)$ and 
which satisfy for all $n\in\{0,1,\ldots,N-1\}$ that 
\begin{equation}
 \begin{split}
  &\Y^{\theta}_{n+1} = \Y_n 
 + \langle \cZ^{\theta}_n, \X_{t_{n+1}}-\X_{t_n} \rangle_{\R^d}
 \\ 
 & + 
 \Bigl(f(t_n,\X_n,\Y^{\theta}_n,\cZ^{\theta}_n,\G^{\theta}_n(\X_n)) 
 + 
 \tfrac12 \operatorname{Trace}\bigl(\sigma(\X_n)\sigma(\X_n)^{\ast}\G^{\theta}_n(\X_n)\bigr)
 \Bigr)\,(t_{n+1}-t_n)
\end{split}
\end{equation}
and
\begin{align}
 \cZ^{\theta}_{n+1} = \cZ^{\theta}_n + \A^{\theta}_n(\X_n)\,(t_{n+1}-t_n) 
 + \G^{\theta}_n(\X_n)\,(\X_{n+1}-\X_n). 
\end{align}
For all suitable $\theta\in\R^{\nu}$  
and all $n\in\{0,1,\ldots,N\}$ 
we think of $\Y^{\theta}_n\colon\Omega\to\R$ as an 
appropriate approximation 
\begin{equation}\label{eq:GeneralAppropriateApproximationForY}
\Y^{\theta}_n \approx Y_{t_n}
\end{equation} 
of $Y_{t_n}\colon\Omega\to\R$, 
for all suitable $\theta\in\R^{\nu}$ 
and all $n\in\{0,1,\ldots,N\}$ 
we think of $\cZ^{\theta}_n\colon\Omega\to\R^d$ as an 
appropriate approximation 
\begin{equation}
 \cZ^{\theta}_n \approx Z_{t_n}
\end{equation}
of $Z_{t_n}\colon\Omega\to\R$, 
for all suitable $ \theta \in \R^{ \nu } $, $ x \in \R^d $
we think of $\U^{\theta}( x ) \in \R $ 
as an appropriate approximation 
\begin{equation}
  \U^{ \theta }( x ) \approx u(0, x)
\end{equation}
of $ u(0,x) \in \R $,
for all suitable $ \theta \in \R^{ \nu } $, $ x \in \R^d $
we think of $ \V^{\theta}( x ) \in \R^d$ 
as an appropriate approximation 
\begin{equation}
 \V^{ \theta }( x ) \approx ( \nabla_x u)( 0, x ) 
\end{equation}
of $ ( \nabla_x u)( 0, x ) \in \R^d $,
for all suitable 
$\theta\in\R^{\nu}$, $x\in\R^d$ and all $n\in\{0,1,\ldots,N-1\}$ 
we think of $\G^{\theta}_n(x)\in\R^{d\times d}$ 
as an approprate approximation
\begin{equation}
 \G^{\theta}_n(x) \approx (\operatorname{Hess}_x u)(t_n,x) 
\end{equation}
of 
$(\operatorname{Hess}_x u)(t_n,x)\in\R^{d\times d},  
$
and for all suitable $\theta\in\R^{\nu}$, 
$x\in\R^d$ and all $n\in\{0,1,\ldots,N-1\}$ 
we think of $\A^{\theta}_n(x)\in\R^d$ 
as an appropriate approximation  
\begin{equation}\label{eq:GeneralAppropriateApproximationForA}
\A^{\theta}_n(x)
\approx
\bigl(
\cL(\nabla_x u)
\bigr)
(t_n,x)
\end{equation}
of
$ 
\bigl(
\cL(\nabla_x u)
\bigr)
(t_n,x)
\in\R^d
$.

\subsection{Stochastic gradient descent-type optimization}
\label{sec:SGD}

As in Subsection~\ref{subsec:algo_sgd} 
we intend to reach 
a \emph{suitable} $\theta\in\R^{\nu}$ 
in 
\eqref{eq:GeneralAppropriateApproximationForY}--\eqref{eq:GeneralAppropriateApproximationForA}
by applying a minimization algorithm 
to the function 
\begin{align}\label{eq:general_mean_square_loss}
 \R^{\nu} \ni \theta \mapsto \E\big[ | \Y^{\theta}_N-g(\X_n) |^2 \big] \in \R. 
\end{align}
Applying a stochastic gradient descent-based 
minimization algorithm yields under suitable assumptions 
random approximations 
$\Theta_m \colon\Omega\to\R^{\nu}$, $m\in\N_0$, 
of a local minimum point of the function in 
\eqref{eq:general_mean_square_loss}. 
For sufficiently large 
$N, \nu, m \in \N$ we use under suitable hypotheses the random function 
\begin{equation}
  \U^{ \Theta_m } \colon \Omega \to C(\R^d,\R)
\end{equation}
as an appropriate 
approximation of the function 
\begin{equation}
  \R^d \ni x \mapsto u(0,x) \in \R
  .
\end{equation}
A more detailed 
description is provided in the next subsection.

\subsection{Framework for the algorithm in the general case}
\label{subsec:general_case}

In this subsection we provide a general framework 
(see Framework~\ref{def:general_algorithm} below)
which covers the deep 2BSDE method 
derived in Subsections~\ref{subsec:details_fullyNonlinear}--\ref{sec:SGD}. 
The variant of the deep 2BSDE method described in Subsection~\ref{subsec:specific_case} 
still remains the core idea of Framework~\ref{def:general_algorithm}. 
However, Framework~\ref{def:general_algorithm} allows more general It\^o processes 
as background dynamics 
(see \eqref{eq:lemma_SDEforX}, \eqref{eq:mergedForm_SDEforX}, \eqref{eq:X_approximate_derivation}, and \eqref{eq:notes_Xapprox} above
and \eqref{eq:FormalXapprox} below)
than just Brownian motion (see Framework~\ref{def:specific_case} in Subsection~\ref{subsec:specific_case} above),
Framework~\ref{def:general_algorithm} allows to incorporate other minimization algorithms 
(cf.\ \eqref{eq:general_gradient_step} below and, e.g., E, Han, \& Jentzen~\cite[Subsections~3.2, 5.1, and 5.2]{EHanJentzen2017})
such as the Adam optimizer
(cf.\ Kingma \& Ba~\cite{KingmaBa2015} and \eqref{eq:examples_setting_moment_estimation}--\eqref{eq:examples_setting_adam_grad_update} below)
than just the plain vanilla stochastic gradient 
descent method (see, e.g., \eqref{eq:specific_gradient}--\eqref{eq:sgd_step} 
in Framework~\ref{def:specific_case} in Subsection~\ref{subsec:specific_case} above), 
and 
Framework~\ref{def:general_algorithm} allows 
to deploy more sophisticated machine learning techniques like batch normalization 
(cf.\ Ioffe \& Szegedy~\cite{IoffeSzegedy2015}
and \eqref{eq:general_batch_normalization} below).
In Section~\ref{sec:examples} below we illustrate the general description in Framework~\ref{def:general_algorithm}
by several examples.

\begin{algo}[General Case]
\label{def:general_algorithm}
Let 
  $T \in (0,\infty)$,  
  $N, d, \varrho, \varsigma, \nu \in \N$, 
let
  $f\colon[0,T]\times\R^d\times\R\times\R^d\times\R^{d\times d}\to\R$ 
  and $g\colon\R^d\to\R$ be functions,
let 
$ 
  ( \Omega, \F, \P, ( \bF_t )_{ t \in [0,T] } ) 
$ 
be a filtered probability space, 
let 
$
  W^{m,j} \colon [0,T] \times \Omega \to \R^d 
$, 
$
  m \in \N_0
$, 
$ 
  j \in \N
$,
be independent standard $ ( \bF_t )_{ t \in [0,T] } $-Brownian motions on 
$(\Omega,\F,\P)$, 
let 
$
  \xi^{m,j}\colon\Omega\to\R^d
$, 
$
  m \in \N_0 
$, 
$ j \in \N $,
be i.i.d.\ $ \bF_0 $/$ \B(\R^d) $-measurable functions,
let
  $t_0, t_1, \ldots, t_N\in [0,T]$ be 
  real numbers with 
  $0 = t_0 < t_1 < \ldots < t_N = T$, 
let  
  $H\colon [0,T]^2\times\R^d\times\R^d\to\R^d$ 
  and
  $\sigma\colon\R^d\to\R^{d\times d}$ 
  be functions,
for every $\theta\in\R^{\nu}$ 
let 
  $ \U^{\theta}\colon\R^d \to \R $ and 
  $ \V^{\theta}\colon\R^d \to\R^d $ 
  be functions, 
for every 
  $ m \in \N_0 $, $ j \in \N $ 
  let 
  $\X^{m,j}\colon \{0,1,\ldots,N\}\times\Omega\to\R^d$ 
  be a stochastic process which satisfies for all 
  $n\in\{0,1,\ldots,N-1\}$ that 
  $\X^{m,j}_0 = \xi^{ m, j } $ and 
  \begin{equation}\label{eq:FormalXapprox}
   \X^{m,j}_{n+1} 
   = 
   H(t_n,t_{n+1},\X^{m,j}_n,W^{m,j}_{t_{n+1}}-W^{m,j}_{t_n}), 
  \end{equation}  
for every 
  $\theta\in\R^{\nu}$, 
  $ j \in \N $, 
  ${\bf s}\in\R^{\varsigma}$, 
  $n\in\{0,1,\ldots,N-1\}$
  let $\G^{\theta,j,{\bf s}}_n\colon(\R^d)^{\N_0}\to\R^{d\times d}$ 
  and $\A^{\theta,j,{\bf s}}_n\colon(\R^d)^{\N_0}\to\R^d$ 
  be functions, 
for every
  $\theta\in\R^{\nu}$, 
  $ m \in \N_0 $,
  $ j \in \N $,
  ${\bf s}\in\R^{\varsigma}$ 
  let $\Y^{\theta,m,j,{\bf s}}\colon \{0,1,\ldots,N\}\times\Omega\to\R$ 
  and $\cZ^{\theta,m,j,{\bf s}}\colon \{0,1,\ldots,N\}\times\Omega\to\R^{d}$ 
  be stochastic processes which satisfy  
  $\Y^{\theta,m,j,{\bf s}}_0 = \U^{\theta}(\xi^{ m, j } )$ and 
  $\cZ^{\theta,m,j,{\bf s}}_0 = \V^{\theta}(\xi^{ m, j } )$ 
  and which satisfy for all $n\in\{0,1,\ldots, N-1\}$ 
  that
  \begin{equation}\label{eq:FormalYapprox}
   \begin{split}
   & \Y^{\theta,m,j,{\bf s}}_{n+1} 
   = 
   \Y^{\theta,m,j,{\bf s}}_n  
   + 
   (t_{n+1}-t_n)\, \Bigl[\tfrac12\operatorname{Trace}\Bigl(
   \sigma(\X^{m,j}_n)\sigma(\X^{m,j}_n)^{\ast} 
   \G^{\theta,j,{\bf s}}_n\bigl((\X^{m,i}_n)_{ i \in \N }\bigr)
   \Bigr)
   \\
   & + 
   f\bigl(t_n,\X^{m,j}_n,\Y^{\theta,m,j}_n,\cZ^{\theta,m,j,{\bf s}}_n,
   \G^{\theta,j,{\bf s}}_n\bigl((\X^{m,i}_n)_{ i \in \N }\bigr)\bigr) 
   \Bigr]
   + 
   \langle
    \cZ^{\theta,m,j,{\bf s}}_n, 
    \X^{m,j}_{n+1}-\X^{m,j}_n
   \rangle_{\R^d}
  \end{split}
 \end{equation}
 and
 \begin{equation}
   \cZ^{\theta,m,j,{\bf s}} 
   = \A^{\theta,j,{\bf s}}_n\bigl((\X^{m,i}_n)_{ i \in \N }\bigr)\,(t_{n+1}-t_n) 
   + \G^{\theta,j,{\bf s}}_n\bigl((\X^{m,i}_n)_{ i \in \N }\bigr)\,(\X^{m,j}_{n+1}-\X^{m,j}_n), 
  \label{eq:FormalZapprox}
 \end{equation}
let $(J_m)_{m\in\N_0}\subseteq\N$ be a sequence, 
for every $m\in\N_0$, ${\bf s}\in\R^{\varsigma}$ 
  let 
  $\phi^{m,{\bf s}}\colon\R^{\nu}\times\Omega\to\R$ be the function which 
  satisfies for all 
  $(\theta,\omega)\in\R^{\nu}\times\Omega$ that 
  \begin{align}\label{eq:general_functionToMinimize}
   \phi^{m,{\bf s}}(\theta,\omega) = \frac{1}{J_m}\sum_{j=1}^{J_m}
   \big|
     \Y^{\theta,m,j,{\bf s}}_N(\omega)-g(\X^{m,j}_N(\omega))
   \big|^2,
  \end{align}
for every $m\in\N_0$, ${\bf s}\in\R^{\varsigma}$ let 
  $\Phi^{m,{\bf s}}\colon\R^{\nu}\times\Omega\to\R^{\nu}$ a function 
  which satisfies for all $\omega\in\Omega$,  
  $\theta\in\{\eta\in\R^{\nu}\colon \phi^{m,{\bf s}}(\cdot,\omega)\colon\R^{\nu}\to\R~\text{is differentiable at}~\eta\}$ 
  that
  \begin{align}
   \Phi^{m,{\bf s}}(\theta,\omega) = (\nabla_{\theta}\phi^{m,{\bf s}})(\theta,\omega),
  \end{align}
let $\S\colon\R^{\varsigma}\times\R^{\nu}\times(\R^d)^{\{0,1,\ldots,N-1\}\times\N}\to\R^{\varsigma}$ 
  be a function, 
for every $m\in\N_0$
  let $\psi_m\colon\R^{\varrho}\to\R^{\nu}$ 
  and $\Psi_m\colon\R^{\varrho}\times\R^{\nu}\to\R^{\varrho}$
  be functions, 
let 
  $\Theta\colon\N_0\times\Omega\to\R^{\nu}$, 
  $\bS\colon\N_0\times\Omega\to\R^{\varsigma}$, 
  and 
  $\Xi\colon\N_0\times\Omega\to\R^{\varrho}$ 
  be stochastic processes 
  which satisfy for all $m\in\N_0$ that 
  \begin{equation}\label{eq:general_batch_normalization} 
   \bS_{m+1} = \S\bigl(\bS_m, \Theta_{m}, 
   (\X_n^{m,i})_{(n,i)\in\{0,1,\ldots,N-1\}\times\N}\bigr),  
  \end{equation}
  \begin{equation}
   \Xi_{m+1} = \Psi_{m}(\Xi_{m},\Phi^{m,\bS_{m+1}}(\Theta_m)),  
   \qquad
   \text{and}
   \qquad
   \Theta_{m+1} = \Theta_{m} - \psi_{m}(\Xi_{m+1}) 
   \label{eq:general_gradient_step}. 
  \end{equation}
\end{algo}

Under suitable further assumptions, we think in 
the case of sufficiently large 
  $N, \nu\in\N$, 
  $m\in\N_0$ 
in Framework \ref{def:general_algorithm}
of the random variable 
$\Theta_m\colon\Omega\to\R^{\nu}$ 
as an appropriate approximation of
a local minimum point of the expected loss function and 
we think in the case of sufficiently large 
  $N, \nu\in\N$, 
  $m\in\N_0$ 
in Framework \ref{def:general_algorithm} 
of the random function 
\begin{equation}
 \R^d \ni x \mapsto \U^{\Theta_m}(x,\omega) \in \R^d 
\end{equation}
for $\omega\in\Omega$ as an 
appropriate approximation of the function 
\begin{equation}
 \R^d \ni x \mapsto u(0,x) \in\R
\end{equation}
where $u\colon [0,T]\times\R^d\to\R$ 
is an at most polynomially growing continuous function which 
satisfies for all $(t,x)\in [0,T)\times\R^d$ that 
$u|_{[0,T)\times\R^d}\in C^{1,2}([0,T)\times\R^d)$, 
$u(T,x) = g(x)$, and 
\begin{align}\label{eq:general_terminal_value_problem}
 \tfrac{\partial u}{\partial t}(t,x) 
 = 
 f\bigl(t, x, u(t,x), (\nabla_x u)(t,x), (\operatorname{Hess}_x u)(t,x)\bigr)
\end{align}
(cf. \eqref{eq:specific_fullyNonlinearPDE} in Subsection~\ref{subsec:specific_case}). 
This terminal value problem 
can in a straight-forward manner be transformed into an initial value problem. 
This is the subject of the following elementary remark. 

\begin{remark}\label{rem:initial_vs_terminal_value}
Let 
  $d\in\N$, 
  $T\in (0,\infty)$, 
let 
  $f\colon [0,T]\times\R^d\times\R\times\R^d\times\R^{d\times d}\to\R$
  and 
  $g\colon \R^d\to\R$ 
  be functions, 
let 
  $u\colon [0,T]\times\R^d\to\R$ 
  be a continuous function which satisfies for 
  all $(t,x)\in [0,T)\times\R^d$ that 
  $u(T,x) = g(x)$, 
  $u|_{[0,T)\times\R^d}\in C^{1,2}([0,T)\times\R^d,\R)$, 
  and 
  \begin{align}\label{eq:fullyNonlinearPDEInitialValueProblemForLittleU}
   \tfrac{\partial u}{\partial t}(t,x) 
   = 
   f\bigl( t, x, u(t,x), (\nabla_x u)(t,x), (\operatorname{Hess}_x u)(t,x) \bigr), 
  \end{align}
and 
let 
  $F\colon [0,T]\times\R^d\times\R\times\R^d\times\R^{d\times d}\to\R$ 
  and 
  $U\colon [0,T]\times\R^d\to\R$ 
  be the functions which satisfy for all 
  $(t,x,y,z,\gamma)\in [0,T]\times\R^d\times\R\times\R^d\times\R^{d\times d}$ that 
  $U(t,x) = u(T-t,x)$ and
  \begin{align}\label{eq:definitionOfF}
  F\bigl(t,x,y,z,\gamma\bigr) = -f\bigl(T-t, x,y,z,\gamma\bigr). 
  \end{align}
Then it holds that $U\colon [0,T]\times\R^d\to\R$ is a continuous function 
which satisfies for all 
$(t,x)\in (0,T]\times\R^d$ that 
$ U(0,x) = g(x)$, $U|_{(0,T]\times\R^d}\in C^{1,2}((0,T]\times\R^d,\R)$, 
and 
\begin{align}
 \label{eq:fullyNonlinearPDEInitialValueProblem}
  \tfrac{\partial U}{\partial t}(t,x) 
  = 
  F\bigl(t, x, U(t,x), (\nabla_x U)(t,x), (\operatorname{Hess}_x U)(t,x)\bigr).
\end{align}
\end{remark}
\begin{proof}[Proof of Remark~\ref{rem:initial_vs_terminal_value}]
First, note that the hypothesis that $u\colon [0,T]\times\R^d\to\R$ 
is a continuous function ensures that $U\colon [0,T]\times\R^d\to\R$ 
is a continuous function. Next, note that for all $x\in\R^d$ it holds that 
\begin{equation}\label{eq:InitialConditionForU}
 U(0,x) = u(T,x) = g(x). 
\end{equation}
Moreover, observe that the chain rule, 
\eqref{eq:fullyNonlinearPDEInitialValueProblemForLittleU}, and 
\eqref{eq:definitionOfF} ensure that for all  
$(t,x)\in (0,T]\times\R^d$ it holds that  
$U|_{(0,T]\times\R^d}\in C^{1,2}((0,T]\times\R^d,\R)$, $U(0,x)=g(x)$, and 
\begin{align}\label{eq:transformationCalculated}
 \tfrac{\partial U}{\partial t}(t,x) 
 & = \tfrac{\partial}{\partial t} \bigl[ u(T-t,x) \bigr]
 = -(\tfrac{\partial u}{\partial t}) ( T-t, x ) \nonumber \\
 & = - f\bigl(T-t, x, u(T-t,x), (\nabla_x u)(T-t,x), (\operatorname{Hess}_x u)(T-t,x) \bigr) \nonumber \\
 & = - f\bigl(T-t, x, U(t,x), (\nabla_x U)(t,x), (\operatorname{Hess}_x U)(t,x) \bigr) \nonumber \\
 & = F\bigl(t, x, U(t,x), (\nabla_x U)(t,x), (\operatorname{Hess}_x U)(t,x) \bigr).  
\end{align}
Combining the fact that $U\colon [0,T]\times\R^d\to\R$ is a continuous function 
with \eqref{eq:InitialConditionForU} and \eqref{eq:transformationCalculated} 
completes the proof of Remark~\ref{rem:initial_vs_terminal_value}.
\end{proof}

\section{Examples}\label{sec:examples}
 In this section we employ the deep 2BSDE method 
(see Framework~\ref{def:specific_case} and Framework~\ref{def:general_algorithm} above)
to approximate the solutions of several example PDEs 
such as Allen-Cahn equations, 
a Hamilton-Jacobi-Bellman (HJB) equation, 
a Black-Scholes-Barenblatt equation,
and nonlinear expectations of $ G $-Brownian motions. 
More specifically, 
in Subsection~\ref{subsec:example_allen_cahn_plain_sgd_no_bn} 
we employ an implementation of the deep 2BSDE method in Framework~\ref{def:specific_case} 
to approximate a $20$-dimensional Allen-Cahn equation, 
in Subsection~\ref{subsec:example_bsb}
we employ an implementation of the deep 2BSDE method in Framework~\ref{def:general_algorithm} 
to approximate a $ 100 $-dimensional Black-Scholes-Barenblatt equation, 
in Subsection~\ref{subsec:example_square_gradient}
we employ an implementation of the deep 2BSDE method in Framework~\ref{def:general_algorithm} 
to approximate 
a $ 100 $-dimensional Hamilton-Jacobi-Bellman equation, 
in Subsection~\ref{subsec:example_allen_cahn} 
we employ an implementation of the deep 2BSDE method in Framework~\ref{def:general_algorithm}
to approximate 
a $ 50 $-dimensional Allen-Cahn equation,
and 
in Subsection~\ref{subsec:example_gbm} 
we employ implementations of the deep 2BSDE method in
Framework~\ref{def:general_algorithm} to approximate 
nonlinear expectations of $ G $-Brownian motions in $ 1 $ and $ 100 $ space-dimensions. 
The {\sc Python} code used for the implementation of the deep 2BSDE method 
in Subsection~\ref{subsec:example_allen_cahn_plain_sgd_no_bn} can be found in 
Subsection~\ref{subsec:plainSGDCode} below. 
The {\sc Python} code used for the implementation of the deep 2BSDE method 
in Subsection~\ref{subsec:example_bsb} can be found in 
Subsection~\ref{subsec:generalCode} below. 
All of the numerical experiments presented below have been 
performed in {\sc Python 3.6} using {\sc TensorFlow} 1.2 or {\sc TensorFlow} 1.3, respectively, 
on a {\sc Lenovo X1 Carbon} with a $2.40$ Gigahertz (GHz) 
{\sc Intel} i7 microprocessor with $8$ Megabytes (MB) RAM. 
 
\subsection{Allen-Cahn equation with plain gradient descent 
and no batch normalization}
\label{subsec:example_allen_cahn_plain_sgd_no_bn}

In this subsection we use the deep 2BSDE method 
in Framework~\ref{def:specific_case} 
to approximatively calculate the solution of a $20$-dimensional 
Allen-Cahn equation with a cubic nonlinearity 
(see~\eqref{eq:example_allen_cahn_plain} below).

Assume 
  Framework \ref{def:specific_case},  
assume 
  that 
  $T = \tfrac{3}{10}$, 
  $\gamma = \tfrac{1}{1000}$, 
  $d = 20$, 
  $N = 20$, 
  $\xi = 0 \in \R^{20}$, 
  $
  \nu \geq ( 5 N d + N d^2 + 1 ) (d + 1)
  $,
assume for every 
  $ \theta = ( \theta_1, \dots, \theta_{ \nu } ) \in \R^{ \nu }$, 
  $x\in\R^d$ 
that 
  \begin{equation}
   \G^{\theta}_0(x) = 
   \begin{pmatrix}
    \theta_{d+2}  	& \theta_{d+3} 	& \ldots & \theta_{2d + 1} \\
    \theta_{2d+2} 	& \theta_{2d+3}	& \ldots & \theta_{3d + 1} \\
    \vdots 		& \vdots	& \vdots & \vdots \\
    \theta_{d^2+2} 	& \theta_{d^3+3}& \ldots & \theta_{d^2 + d + 1} 
   \end{pmatrix}
   \in\R^{d\times d}
   \quad\text{and}\quad
   \A^{\theta}_0(x) = 
   \begin{pmatrix}
    \theta_{d^2+d+2} \\
    \theta_{d^2+d+3} \\
    \vdots \\ 
    \theta_{d^2+2d+1}
   \end{pmatrix}
    \in\R^d,
  \end{equation}
for every $k\in\N$ let 
  $ \mathcal{R}_{k} \colon \R^k \to \R^k $ be the function 
  which satisfies for all $ x = ( x_1, \dots, x_k ) \in \R^k $ that
  \begin{equation}
    \mathcal{R}_k( x ) 
    =
    \left(
      \max\{ x_1, 0 \}
      ,
      \dots
      ,
      \max\{ x_k, 0 \}
    \right)
    ,
  \end{equation}
  for every 
$ \theta = ( \theta_1, \dots, \theta_{ \nu } ) \in \R^{ \nu } $, 
$ v \in \N_0 $,
$ k, l \in \N $
with 
$
  v + k (l + 1 ) \leq \nu
$
let 
$ M^{ \theta, v }_{ k, l } \colon \R^l \to \R^k $ 
be the affine linear function which satisfies for all 
$ x = ( x_1, \dots, x_l ) $ that
\begin{equation}
 M^{ \theta, v }_{ k, l }( x )
  =
  \left(
    \begin{array}{cccc}
      \theta_{ v + 1 }
    &
      \theta_{ v + 2 }
    &
      \dots
    &
      \theta_{ v + l }
    \\
      \theta_{ v + l + 1 }
    &
      \theta_{ v + l + 2 }
    &
      \dots
    &
      \theta_{ v + 2 l }
    \\
      \theta_{ v + 2 l + 1 }
    &
      \theta_{ v + 2 l + 2 }
    &
      \dots
    &
      \theta_{ v + 3 l }
    \\
      \vdots
    &
      \vdots
    &
      \vdots
    &
      \vdots
    \\
      \theta_{ v + ( k - 1 ) l + 1 }
    &
      \theta_{ v + ( k - 1 ) l + 2 }
    &
      \dots
    &
      \theta_{ v + k l }
    \end{array}
  \right)
  \left(
    \begin{array}{c}
      x_1
    \\
      x_2
    \\
      x_3
    \\
      \vdots 
    \\
      x_l
    \end{array}
  \right)
  +
  \left(
    \begin{array}{c}
      \theta_{ v + k l + 1 }
    \\
      \theta_{ v + k l + 2 }
    \\
      \theta_{ v + k l + 3 }
    \\
      \vdots 
    \\
      \theta_{ v + k l + k }
    \end{array}
  \right),
\end{equation}
assume for all 
$ \theta \in \R^{ \nu } $,
$ n \in \{ 1,2,\ldots, N-1 \} $,
$ x \in \R^d $
that
\begin{equation}
   \A^{ \theta }_n
    =
    M^{ \theta, [ ( 2N + n ) d + 1] ( d + 1 ) }_{ d, d } 
    \circ 
    \mathcal{R}_d
    \circ 
    M^{ \theta, [ ( N + n ) d + 1] ( d + 1 ) }_{ d, d } \\
    \circ 
    \mathcal{R}_d
    \circ 
    M^{ \theta, (n d + 1) ( d + 1 ) }_{ d, d } 
\end{equation}
and
\begin{equation}
  \G^{ \theta }_n
    =
    M^{ \theta, ( 5 N d + n d^2 + 1 ) ( d + 1 ) }_{ d^2, d } 
    \circ 
    \mathcal{R}_d
    \circ 
    M^{ \theta, [ ( 4 N + n) d + 1 ] ( d + 1 ) }_{ d, d }
    \circ 
    \mathcal{R}_d
    \circ 
    M^{ \theta, [ ( 3 N + n) d  + 1 ] ( d + 1 ) }_{ d, d },
\end{equation}
assume for all 
  $i\in\{0,1,\ldots,N\}$, 
  $\theta\in\R^{\nu}$, 
  $t\in [0,T)$, 
  $x,z\in\R^d$, 
  $y\in\R$,  
  $S\in\R^{d\times d}$
  that 
  $t_i = \tfrac{iT}{N}$, 
  $g(x) = \bigl[2+\tfrac{2}{5}\|x\|_{\R^d}^2\bigr]^{-1}$, 
  and 
  \begin{align}
   f\bigl(t,x,y,z,S\bigr) = -\tfrac12\operatorname{Trace}(S) - y + y^3, 
  \end{align}
and let $u\colon [0,T]\times\R^d\to\R$ be an at most polynomially 
growing continuous function which satisfies for all $(t,x)\in [0,T)\times\R^d$ 
that $u(T,x) = g(x)$, 
$u|_{[0,T)\times\R^d}\in C^{1,2}([0,T)\times\R^d,\R)$, 
and 
\begin{align}\label{eq:example_plain_fully_nonlinear_PDE}
 \tfrac{\partial u}{\partial t}(t,x) 
 = 
 f(t,x,u(t,x),(\nabla_x u)(t,x), (\operatorname{Hess}_x u)(t,x)). 
\end{align}
The solution $u\colon [0,T]\times\R^d\to\R$ 
of the PDE~\eqref{eq:example_plain_fully_nonlinear_PDE} 
satisfies for all $(t,x)\in [0,T)\times\R^d$ that 
$u(T,x) = [2+\tfrac25 \|x\|_{\R^d}]^{-1}$ and 
\begin{align}\label{eq:example_allen_cahn_plain}
 \tfrac{\partial u}{\partial t}(t,x) 
 + \tfrac12(\Delta_x u)(t,x) 
 + u(t,x) - [u(t,x)]^3 
 = 0. 
\end{align}
In Table~\ref{tab:table_AllenCahnPlainSGD.tex} 
we use 
{\sc Python} code~\ref{code:deepPDEmethodPlainSGDNoBN}
in Subsection~\ref{subsec:plainSGDCode} below
to approximatively calculate 
  the mean of $\Theta^{(1)}_m$, 
  the standard deviation of $\Theta^{(1)}_m$, 
  the relative $ L^1 $-approximation error associated to $\Theta^{(1)}_m$, 
  the uncorrected sample standard deviation of the relative approximation error associated to $ \Theta^{(1)}_m $, 
  the mean of the loss function associated to $\Theta_m$, 
  the standard deviation of the loss function associated to $\Theta_m$, 
  and the average runtime in seconds needed for calculating one realization of $\Theta^{(1)}_m$ 
against $m\in\{0,1000,2000,3000,4000,5000\}$ based on $10$ independent realizations 
(10 independent runs of {\sc Python} code~\ref{code:deepPDEmethodPlainSGDNoBN}
in Subsection~\ref{subsec:plainSGDCode} below). 
In addition, Figure~\ref{fig:figure_AllenCahnPlain} depicts approximations of the relative $ L^1 $-approximation error 
and approximations of the mean of the loss function associated to $ \Theta^{(1)}_m $ against 
$ m \in \{ 0, 1, 2, \ldots, 5000 \} $ 
based on $10$ independent realizations 
(10 independent runs of {\sc Python} code~\ref{code:deepPDEmethodPlainSGDNoBN}
in Subsection~\ref{subsec:plainSGDCode} below). 
In the approximative calculations of the relative $ L^1 $-approximation errors 
in Table~\ref{fig:figure_AllenCahnPlain} and 
Figure~\ref{fig:figure_AllenCahnPlain}
the value 
$ 
  u(0,\xi) 
$ 
of the solution $ u $ of the PDE~\eqref{eq:example_allen_cahn_plain} 
has been replaced by the value $ 0.30879 $ 
which, in turn, has been calculated through the Branching diffusion method 
(see {\sc Matlab} code~\ref{code:branchingMatlab} in Appendix~\ref{subsec:BranchingMatlab} below). 

\begin{center}
\begin{table}
\begin{center}
\begin{tabular}{|c|c|c|c|c|c|c|c|}
\hline
Number&Mean&Standard&Rel.\ $L^1$-&Standard&Mean&Standard&Runtime\\
of&of $ \mathcal{U}^{ \Theta_m } $&deviation&approx.&deviation&of the&deviation&in sec.\\
iteration&&of $ \mathcal{U}^{ \Theta_m } $&error&of the&loss&of the&for one \\
steps &&&&relative&function&loss&realiz.\\
&&&&approx.&&function&of $ \mathcal{U}^{ \Theta_m } $\\
&&&&error&&&\\
\hline
    0 & 
-0.02572 & 
0.6954 & 
2.1671 & 
1.24464 & 
0.50286 & 
0.58903 & 
3\\
 1000 & 
0.19913 & 
0.1673 & 
0.5506 & 
0.34117 & 
0.02313 & 
0.01927 & 
6\\
 2000 & 
0.27080 & 
0.0504 & 
0.1662 & 
0.11875 & 
0.00758 & 
0.00672 & 
8\\
 3000 & 
0.29543 & 
0.0129 & 
0.0473 & 
0.03709 & 
0.01014 & 
0.01375 & 
11\\
 4000 & 
0.30484 & 
0.0054 & 
0.0167 & 
0.01357 & 
0.01663 & 
0.02106 & 
13\\
 5000 & 
0.30736 & 
0.0030 & 
0.0093 & 
0.00556 & 
0.00575 & 
0.00985 & 
15\\
\hline
\end{tabular}
\end{center}
\caption{
Numerical simulations of the deep2BSDE method in 
Framework 
\ref{def:specific_case} in the case of 
the $20$-dimensional Allen-Cahn equation 
\eqref{eq:example_allen_cahn_plain}
(cf.\ {\sc Python} code \ref{code:deepPDEmethodPlainSGDNoBN}
in Subsection \ref{subsec:plainSGDCode} below).
In the approximative calculations of the 
relative $L^1$-approximation errors 
the value $u(0, \xi)$ has been replaced by the value 0.30879 which has been
calculated through the Branching diffusion method 
(cf.\ {\sc Matlab} code 
\ref{code:branchingMatlab} in Subsection 
\ref{subsec:BranchingMatlab} below).
\label{tab:table_AllenCahnPlainSGD.tex}}
\end{table}
\end{center}

\begin{figure}
\includegraphics{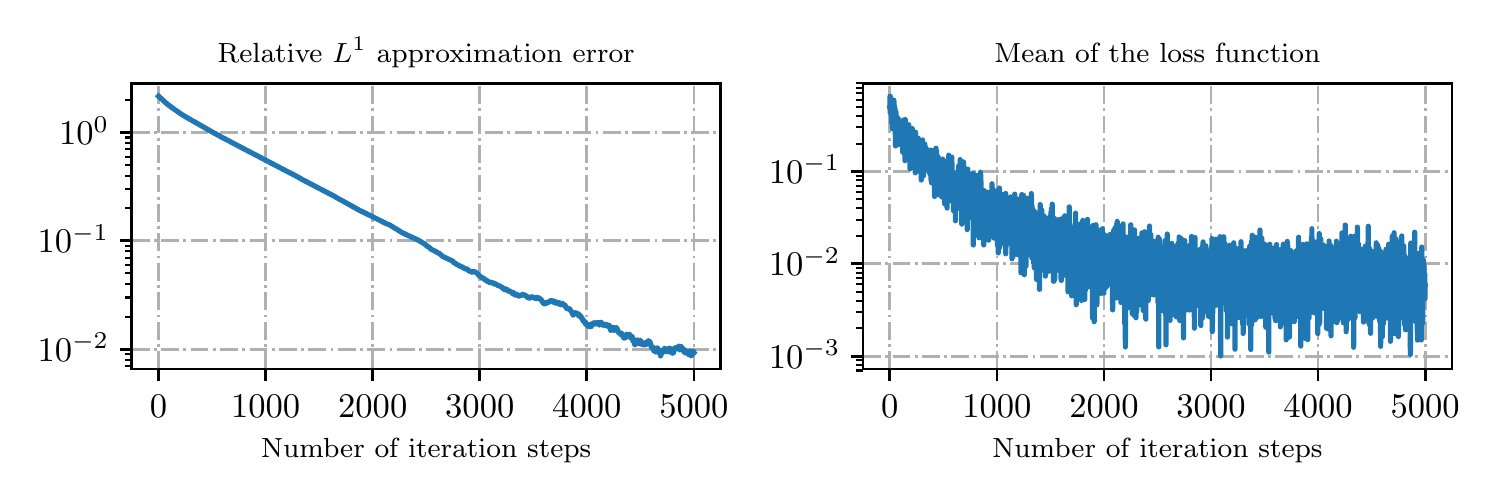}
\input{AllenCahnPlain_Caption}
\end{figure}

\subsection[Setting for the deep 2BSDE method with batch normalization and Adam]{Setting for the deep 2BSDE method 
with batch normalization and the Adam optimizer}
\label{subsec:example_setting}
Assume Framework~\ref{def:general_algorithm}, 
let
  $\varepsilon\in (0,\infty)$, 
  $\beta_1 = \tfrac{9}{10}$, 
  $\beta_2 = \tfrac{999}{1000}$, 
  $(\gamma_m)_{m\in\N_0}\subseteq (0,\infty)$, 
let 
  $\operatorname{Pow}_r \colon \R^{\nu}\to\R^{\nu}$, $r\in (0,\infty)$,  
  be the functions which satisfy for all $r\in (0,\infty)$, 
  $x=(x_1,\ldots,x_{\nu})\in\R^{\nu}$ that 
  \begin{align}
   \operatorname{Pow}_r(x) = (|x_1|^r,\ldots,|x_{\nu}|^r), 
  \end{align}
let 
  $u \colon [0,T]\times \R^d\to\R$
  be an at most polynomially growing continuous function 
  which satisfies 
  for all $(t,x)\in [0,T)\times \R^d$ that 
  $u(T,x) = g(x)$, 
  $u|_{ [0,T) \times \R^d }\in C^{1,2}([0,T)\times\R^d,\R)$, 
  and  
  \begin{align}\label{eq:example_PDE_general}
   \tfrac{ \partial u }{ \partial t }( t, x )
   = f(t,x,u(t,x),(\nabla_x u)(t,x),(\operatorname{Hess}_x u)(t,x)), 
  \end{align}
assume for all 
$m\in\N_0$, 
$i\in\{0,1,\ldots,N\}$
that 
$ J_m = 64 $,
$ t_i = \tfrac{iT}{N} $, 
and 
$ \varrho = 2 \nu $, 
and assume 
  for all $m\in\N_0$, $x=(x_1,\ldots,x_{\nu}),~y=(y_1,\ldots,y_{\nu})\in\R^{\nu}$, 
  $\eta = ( \eta_1 , \ldots , \eta_{\nu} )\in \R^{\nu}$ 
  that 
  \begin{align}\label{eq:examples_setting_moment_estimation}
   \Psi_m ( x , y , \eta ) 
   = 
   (\beta_1 x + (1-\beta_1) \eta, \beta_2 y + (1-\beta_2) \operatorname{Pow}_2(\eta))
  \end{align}
  and 
  \begin{align}\label{eq:examples_setting_adam_grad_update}
   \psi_m ( x,y ) = 
   \biggl(
   \Bigl[
   \sqrt{\tfrac{|y_1|}{1-\beta_2^m}} + \varepsilon
   \Bigr]^{-1}
   \frac{\gamma_m x_{1}}{1-\beta_1^m},
   \ldots, 
   \Bigl[
   \sqrt{\tfrac{|y_{\nu}|}{1-\beta_2^m}} + \varepsilon
   \Bigr]^{-1}
   \frac{\gamma_m x_{\nu}}{1-\beta_1^m}
   \biggr). 
  \end{align}

\begin{remark} 
Equations~\eqref{eq:examples_setting_moment_estimation} and \eqref{eq:examples_setting_adam_grad_update} 
describe the Adam optimizer; 
cf.~Kingma~\& Ba \cite{KingmaBa2015} 
and lines~181--186 in {\sc Python} code~\ref{code:deepPDEmethod}
in Subsection~\ref{subsec:generalCode} below.
The default choice 
in {\sc TensorFlow} 
for the real number $ \varepsilon \in (0,\infty) $ 
in \eqref{eq:examples_setting_adam_grad_update}
is $ \varepsilon = 10^{-8} $ 
but according to the comments 
in the file {\ttfamily adam.py} 
in {\sc TensorFlow} 
there are situations in which other choices may be more appropriate. 
In Subsection \ref{subsec:example_allen_cahn} we took $\varepsilon=1$ 
(in which case one has to add the argument {\ttfamily epsilon=1.0}
to {\ttfamily tf.train.AdamOptimizer} in lines 181--183 
in {\sc Python} code \ref{code:deepPDEmethod} in Subsection \ref{subsec:generalCode} below) 
whereas we used the default value 
$\varepsilon = 10^{-8}$ in 
Subsections 
\ref{subsec:example_bsb},
\ref{subsec:example_square_gradient},
and 
\ref{subsec:example_gbm}.
\end{remark}

\subsection{A $100$-dimensional Black-Scholes-Barenblatt equation}
\label{subsec:example_bsb}
In this subsection we use the deep 2BSDE method 
in Framework \ref{def:general_algorithm} to approximatively 
calculate the solution of a $100$-dimensional Black-Scholes-Barenblatt equation 
(see 
Avellaneda, Levy, \& Par\'as~\cite{AvellanedaLevyParas1995} 
and \eqref{eq:example_bsb} below). 

Assume the setting of  
Subsection~\ref{subsec:example_setting}, 
assume 
  $d=100$, 
  $T=1$, 
  $N=20$, 
  $\varepsilon = 10^{-8}$, 
assume for all $\omega\in\Omega$ that 
  $\xi(\omega) = (1,\nicefrac12,1,\nicefrac12,\ldots,1,\nicefrac12)\in\R^d$,
let 
 $r = \frac{ 5 }{ 100 } $, 
 $\sigma_{\text{max}} = \frac{ 4 }{ 10 } $, 
 $\sigma_{\text{min}} = \frac{ 1 }{ 10 } $,  
 $\sigma_c = \frac{ 4 }{ 10 } $, 
let 
 $\bar\sigma\colon\R\to\R$ be the function which satisfies 
 for all $x\in\R$ that  
 \begin{align}
  \bar\sigma(x) = 
  \begin{cases}
   \sigma_{\text{max}} & \colon x\geq 0 \\
   \sigma_{\text{min}} & \colon x < 0
  \end{cases},
 \end{align}
assume for all 
 $s,t\in [0,T]$, 
 $x=(x_1,\ldots,x_d), 
 w=(w_1,\ldots,w_d), 
 z=(z_1,\ldots,z_d)\in\R^d$,
 $y\in\R$, 
 $S=(S_{ij})_{(i,j)\in\{1,\ldots,d\}^2}\in\R^{d\times d}$
that 
 $\sigma(x) = \sigma_c \text{diag}(x_1,\ldots,x_d)$, 
 $H(s,t,x,w) = x + \sigma(x)w$,
 $g(x) = \|x\|_{\R^d}^2$, 
 and 
 \begin{align}
  f(t,x,y,z,S) 
  = 
  -\frac12\sum_{i=1}^d 
  | x_i |^2
  \,
  | 
    \bar\sigma( S_{ii} )
  |^2 
  S_{ii}
  + r(y - \langle x, z\rangle_{\R^d}).
 \end{align}
The solution $u\colon [0,T]\times\R^d\to\R$ of the 
PDE~\eqref{eq:example_PDE_general} then
satisfies for all $(t,x)\in [0,T)\times\R^d$ that 
$u(T,x) = \|x\|_{\R^d}^2$ and 
\begin{equation}
\label{eq:example_bsb}
  \tfrac{\partial u}{\partial t}(t,x)    
  + \tfrac12
  {\textstyle 
  \sum\limits_{i=1}^d
  }
  | x_i |^2
  \big|
    \bar\sigma\big(
      \tfrac{ \partial^2 u }{ \partial x_i^2 }( t, x) 
    \big)
  \big|^2
  \tfrac{\partial^2 u}{\partial x_i^2}(t,x)
  =
  r \bigl(u(t,x)-\langle x, (\nabla_x u)(t,x)\rangle_{\R^d}\bigr). 
\end{equation}
\begin{center}
\begin{table}
\begin{center}
\begin{tabular}{|c|c|c|c|c|c|c|c|}
\hline
Number&Mean&Standard&Rel.\ $L^1$-&Standard&Mean&Standard&Runtime\\
of&of $ \mathcal{U}^{ \Theta_m } $&deviation&approx.&deviation&of the&deviation&in sec.\\
iteration&&of $ \mathcal{U}^{ \Theta_m } $&error&of the&empirical&of the&for one \\
steps &&&&relative&loss&empirical&realiz.\\
&&&&approx.&function&loss&of $ \mathcal{U}^{ \Theta_m } $\\
&&&&error&&function&\\
\hline
    0 & 
0.522 & 
0.2292 & 
0.9932 & 
0.00297 & 
5331.35 & 
101.28 & 
25\\
  100 & 
56.865 & 
0.5843 & 
0.2625 & 
0.00758 & 
441.04 & 
90.92 & 
191\\
  200 & 
74.921 & 
0.2735 & 
0.0283 & 
0.00355 & 
173.91 & 
40.28 & 
358\\
  300 & 
76.598 & 
0.1636 & 
0.0066 & 
0.00212 & 
96.56 & 
17.61 & 
526\\
  400 & 
77.156 & 
0.1494 & 
0.0014 & 
0.00149 & 
66.73 & 
18.27 & 
694\\
\hline
\end{tabular}
\end{center}
\caption{
Numerical simulations of the deep2BSDE method 
in Framework 
\ref{def:general_algorithm} in the case of 
the $100$-dimensional
Black-Scholes-Barenblatt equation 
\eqref{eq:example_bsb}
(cf.\ {\sc Python} code \ref{code:deepPDEmethod}
in Subsection \ref{subsec:generalCode} below).
In the approximative calculations of the relative 
$L^1$-approximation errors the value $u(0,(1,\nicefrac12,1,\nicefrac12,\ldots,1,\nicefrac12))$ 
has been replaced by the value 77.1049
which has been calculated by means of Lemma 
\ref{lem:bsb_analytic}.
\label{tab:table_BSB.tex}}
\end{table}
\end{center}
\begin{figure}
 \includegraphics{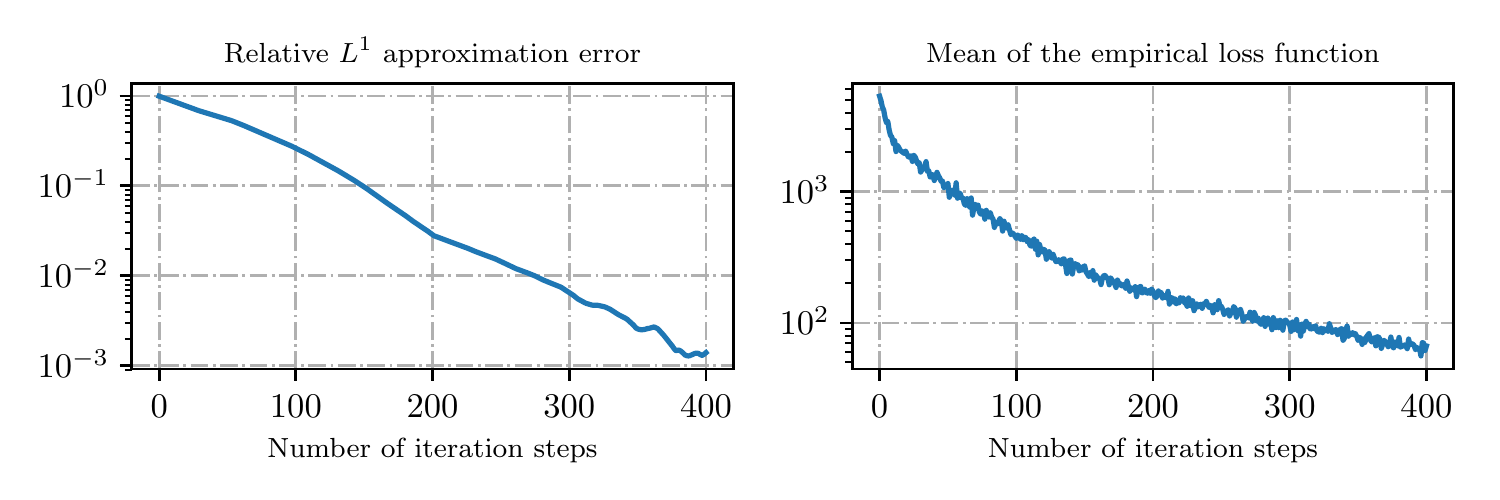}
 \input{BSB_Caption} 
\end{figure}
In Table~\ref{tab:table_BSB.tex} 
we use {\sc Python} code \ref{code:deepPDEmethod} 
in Subsection \ref{subsec:generalCode} below 
to approximatively calculate  
  the mean $\U^{\Theta_m}(\xi)$, 
  the standard deviation of $\U^{\Theta_m}(\xi)$, 
  the relative $L^1$-approximation error associated to $\U^{\Theta_m}(\xi)$, 
  the uncorrected sample standard deviation of the relative approximation error associated to $\U^{\Theta_m}(\xi)$, 
  the mean of the empirical loss function associated to $\Theta_m$, 
  the standard deviation of the empirical loss function associated to $\Theta_m$, 
  and the average runtime in seconds needed for calculating one realization of $\U^{\Theta_m}(\xi)$ 
against $m\in\{0,100,200,300,400\}$ based on $10$ realizations 
($10$ independent runs of {\sc Python} code~\ref{code:deepPDEmethod} 
in Subsection \ref{subsec:generalCode} below).
In addition, Figure \ref{fig:figure_bsb} depicts approximations of 
the relative $L^1$-approximation error 
and approximations of the mean of the empirical loss function 
associated to $\Theta_m$ 
against $m\in \{0,1,\ldots,400\}$ based on $10$ independent 
realizations ($10$ independent runs of {\sc Python} code 
\ref{code:deepPDEmethod}). 
In the approximative calculations of the relative 
$L^1$-approximation errors in Table~\ref{tab:table_BSB.tex} 
and Figure~\ref{fig:figure_bsb} the value 
$u(0,(1,\nicefrac12,1,\nicefrac12,\ldots,1,\nicefrac12))$ of the solution $u$ of the PDE 
\eqref{eq:example_bsb} has been replaced by the value 
$77.1049$ which, in turn, has been calculated 
by means of Lemma \ref{lem:bsb_analytic} below. 
\begin{lemma}
\label{lem:bsb_analytic}
Let $c,\sigma_{\max},r, T\in (0,\infty)$, $\sigma_{\min}\in (0,\sigma_{\max})$, $d\in\N$, 
let $\bar\sigma\colon\R\to\R$ be the function which satisfies 
for all $x\in\R$ that 
\begin{equation}
\label{eq:def_bar_sigma}
 \bar{\sigma}(x) = \begin{cases}
                    \sigma_{\max} & \colon x\geq 0 \\
                    \sigma_{\min} & \colon x < 0
                   \end{cases},
\end{equation}
and let 
  $g\colon\R^d\to\R$ 
  and 
  $u\colon [0,T]\times\R^d\to\R$
  be the functions which satisfy for all 
 $ t \in [0,T] $, $ x = (x_1,\dots,x_d) \in \R^d$ 
  that $g(x) = c\|x\|^2_{\R^d} = c \sum_{ i = 1 }^d | x_i |^2 $ and 
\begin{equation}
\label{eq:def_u_lemma}
  u(t,x) = \exp\!\left( [ r + | \sigma_{ \max } |^2 ] ( T - t ) \right) g(x).
\end{equation}
Then it holds for all 
$ t \in [0,T] $, $ x = (x_1,\dots,x_d) \in \R^d$ 
that 
  $u\in C^{\infty}([0,T]\times\R^d,\R)$, 
  $u(T,x) = g(x)$,  
  and 
\begin{equation}
  \tfrac{ \partial u }{ \partial t }(t,x)  
  + 
  \tfrac12
  {\textstyle 
    \sum\limits_{i=1}^d
  }
  | x_i |^2
  \big|
    \bar\sigma\big(
      \tfrac{ \partial^2 u }{ \partial x_i^2 }( t, x) 
    \big)
  \big|^2
  \tfrac{ \partial^2 u }{ \partial x_i^2 }(t,x)
  =  
  r 
  \big(
    u(t,x) - \langle x, ( \nabla_x u )(t,x) \rangle_{\R^d}
  \big)
  .
\end{equation}
\end{lemma}
\begin{proof}[Proof of Lemma~\ref{lem:bsb_analytic}]
Observe that the function $ u $ is clearly infinitely often differentiable. 
Next note that \eqref{eq:def_u_lemma} ensures that for all 
$ t \in [0,T] $, 
$ x = (x_1,\dots,x_d) \in \R^d $ 
it holds that
\begin{equation}
  u(t,x) = 
  \exp\!\left( - t [ r + | \sigma_{ \max } |^2 ] + T [ r + | \sigma_{ \max } |^2 ] \right) 
  g(x) .
\end{equation}
Hence, we obtain that 
for all
$ t \in [0,T] $, 
$ x = (x_1,\dots,x_d) \in \R^d $,
$ i \in \{ 1, 2, \dots, d \} $
it holds that
\begin{equation}
\label{eq:lem_first_t}
 \tfrac{\partial u}{\partial t}(t,x) 
 = 
 - [ r + | \sigma_{\max} |^2 ] u(t,x) , 
\end{equation}
\begin{equation}
\begin{split}
  \langle x, (\nabla_x u)(t,x) \rangle_{\R^d} 
& = 
  \exp\!\left( - t [ r + | \sigma_{ \max } |^2 ] + T [ r + | \sigma_{ \max } |^2 ] \right) 
  \left< 
    x, ( \nabla g )( x )
  \right>_{ \R^d }
\\ & =
  \exp\!\left( - t [ r + | \sigma_{ \max } |^2 ] + T [ r + | \sigma_{ \max } |^2 ] \right) 
  \left< 
    x, 2 c x
  \right>_{ \R^d }
\\ & =
  2
  c
  \exp\!\left( - t [ r + | \sigma_{ \max } |^2 ] + T [ r + | \sigma_{ \max } |^2 ] \right) 
  \left\| x \right\|^2_{ \R^d }
  =
  2 u(t,x) , 
\end{split}
\end{equation}
and
\begin{equation}
\label{eq:lem_2nd_u}
\begin{split} 
  \tfrac{\partial^2 u}{\partial x_i^2}(t,x)
& =
  2 c
  \exp\!\left( - t [ r + | \sigma_{ \max } |^2 ] + T [ r + | \sigma_{ \max } |^2 ] \right) 
  > 0
  .
\end{split}
\end{equation}
Combining this with \eqref{eq:def_bar_sigma} demonstrates that
for all
$ t \in [0,T] $, 
$ x = (x_1,\dots,x_d) \in \R^d $,
$ i \in \{ 1, 2, \dots, d \} $
it holds that
\begin{equation}
\begin{split}
  \bar\sigma\bigl(
    \tfrac{\partial^2 u}{\partial x_i^2}(t,x)
  \bigr) 
  = \sigma_{\max} 
  .
\end{split}
\end{equation}
This and \eqref{eq:lem_first_t}--\eqref{eq:lem_2nd_u} ensure that
for all
$ t \in [0,T] $, 
$ x = (x_1,\dots,x_d) \in \R^d $
it holds that
\begin{equation}
\begin{split}
&
  \tfrac{ \partial u }{ \partial t }(t,x)  
  + 
  \tfrac12
  {\textstyle 
    \sum\limits_{i=1}^d
  }
  | x_i |^2
  \big|
    \bar\sigma\big(
      \tfrac{ \partial^2 u }{ \partial x_i^2 }( t, x) 
    \big)
  \big|^2
  \tfrac{ \partial^2 u }{ \partial x_i^2 }(t,x)
  -
  r 
  \big(
    u(t,x) - \langle x, ( \nabla_x u )(t,x) \rangle_{\R^d}
  \big)
\\
&
=
  - \left[ r + | \sigma_{\max} |^2 \right] u(t,x)  
  + 
  \tfrac12
  {\textstyle 
    \sum\limits_{i=1}^d
  }
  | x_i |^2
  \big|
    \bar\sigma\big(
      \tfrac{ \partial^2 u }{ \partial x_i^2 }( t, x) 
    \big)
  \big|^2
  \tfrac{ \partial^2 u }{ \partial x_i^2 }(t,x)
  -
  r 
  \big(
    u(t,x) - 2 u(t,x)
  \big)
\\ &
=
  - \left[ r + | \sigma_{\max} |^2 \right] u(t,x)  
  + 
  \tfrac12
  {\textstyle 
    \sum\limits_{i=1}^d
  }
  | x_i |^2
  | \sigma_{ \max } |^2
  \tfrac{ \partial^2 u }{ \partial x_i^2 }(t,x)
  +
  r 
  u(t,x)
\\ &
=
  \tfrac12
  {\textstyle 
    \sum\limits_{i=1}^d
  }
  | x_i |^2
  | \sigma_{ \max } |^2
  \tfrac{ \partial^2 u }{ \partial x_i^2 }(t,x)
  - | \sigma_{\max} |^2 u(t,x)  
=
  | \sigma_{ \max } |^2
  \left[ 
    \tfrac12
    {\textstyle 
      \sum\limits_{i=1}^d
    }
    | x_i |^2
    \tfrac{ \partial^2 u }{ \partial x_1^2 }(t,x)
    - 
    u(t,x)  
  \right]
\\ & =
  | \sigma_{ \max } |^2
  \left[ 
    \tfrac12
    \left\| x \right\|^2_{ \R^d }
    \tfrac{ \partial^2 u }{ \partial x_1^2 }(t,x)
    - 
    u(t,x)  
  \right]
\\ &
=
  | \sigma_{ \max } |^2
  \left[ 
    c
    \left\| x \right\|^2_{ \R^d }
    \exp\!\left( - t [ r + | \sigma_{ \max } |^2 ] + T [ r + | \sigma_{ \max } |^2 ] \right) 
    - 
    u(t,x)  
  \right]
  = 0
  .
\end{split}
\end{equation}
The proof of Lemma~\ref{lem:bsb_analytic} is thus completed.
\end{proof}

\subsection{A $100$-dimensional Hamilton-Jacobi-Bellman equation}
\label{subsec:example_square_gradient}
In this subsection we use the deep 2BSDE method 
in Framework \ref{def:general_algorithm} to 
approximatively calculate the solution of a $100$-dimensional 
Hamilton-Jacobi-Bellman equation with a nonlinearity that is 
quadratic in the gradient (see, e.g., \cite[Section~4.3]{EHanJentzen2017} 
and \eqref{eq:example_hamilton_jacobi_bellman} 
below).

Assume the setting of Subsection \ref{subsec:example_setting},  
assume
  $d=100$, 
  $T=1$, 
  $N=20$, 
  $\varepsilon = 10^{-8}$, 
assume for all $\omega\in\Omega$ that 
  $\xi(\omega) = (0,0,\ldots,0)\in\R^d$,   
and assume for all 
  $m\in\N_0$, 
  $s,t\in [0,T]$, 
  $x,w,z\in\R^d$, 
  $y\in\R$, 
  $S\in\R^{d\times d}$ 
that 
  $\sigma(x)=\sqrt{2}\,\operatorname{Id}_{\R^d}$,  
  $H(s,t,x,w) = x + \sqrt{2}\,w$, 
  $\gamma_m = \frac{1}{100}$,  
  $g(x) = \ln( \frac12 [ 1 + \|x\|_{\R^d}^2 ] )$, 
  and 
  \begin{align}
    f(t,x,y,z,S) = -\operatorname{Trace}(S) + \|z\|_{\R^d}^2.
  \end{align} 
The solution $u\colon [0,T]\times\R^d\to\R$ of the PDE \eqref{eq:example_PDE_general} 
then satisfies for all $(t,x)\in [0,T)\times\R^d$ that 
\begin{align}\label{eq:example_hamilton_jacobi_bellman}
 \tfrac{\partial u}{\partial t}(t,x) + (\Delta_x u)(t,x) = \|\nabla_x u(t,x)\|_{\R^d}^2.
\end{align}
In Table~\ref{tab:table_HamiltonJacobiBellman.tex} 
we use an adapted version of {\sc Python} code~\ref{code:deepPDEmethod} 
in Subsection~\ref{subsec:generalCode} below to approximatively calculate 
  the mean of $\U^{\Theta_m}(\xi)$, 
  the standard deviation of $\U^{\Theta_m}(\xi)$, 
  the relative $L^1$-approximation error associated to $\U^{\Theta_m}(\xi)$, 
  the uncorrected sample standard deviation of the relative approximation error associated to $\U^{\Theta_m}(\xi)$, 
  the mean of the empirical loss function associated to $\U^{\Theta_m}(\xi)$, 
  the standard deviation of the empirical loss function associated to $\U^{\Theta_m}(\xi)$, 
  and the average runtime in seconds needed for calculating one 
  realization of $\U^{\Theta_m}(\xi)$ 
against $m\in\{0,500,1000,1500,2000\}$ based on $10$ 
independent realizations ($10$ independent runs). 
In addition, Figure~\ref{fig:figure_HamiltonJacobiBellman} depicts 
approximations of the mean of the relative $L^1$-approximation error and
approximations of the mean of the empirical loss function associated to $\Theta_m$ 
against $m\in\{0,1,\ldots,2000\}$ based on $10$ independent realizations ($10$ independent runs). 
In the calculation of the relative $L^1$-approximation errors 
in 
Table \ref{tab:table_HamiltonJacobiBellman.tex} and 
Figure \ref{fig:figure_HamiltonJacobiBellman}
the value $u(0,(0,0,\ldots,0))$ of the solution 
of the PDE \eqref{eq:example_hamilton_jacobi_bellman} has 
been replaced by the value $4.5901$ which, in turn, was calculated 
by means of Lemma 4.2 in \cite{EHanJentzen2017} 
(with $d=100$, $T=1$, $\alpha=1$, $\beta=-1$, 
$g=\R^d\ni x\mapsto \ln(\tfrac12[1+\|x\|_{\R^d}^2]) \in \R $ in the 
notation of Lemma 4.2 in \cite{EHanJentzen2017}) 
and the classical Monte Carlo method 
(cf.\ {\sc Matlab} code~\ref{code:MC_HJB} 
in Appendix~\ref{subsec:MC_HJB} below). 
\begin{center}
\begin{table}
\begin{center}
\begin{tabular}{|c|c|c|c|c|c|c|c|}
\hline
Number&Mean&Standard&Rel.\ $L^1$-&Standard&Mean&Standard&Runtime\\
of&of $ \mathcal{U}^{ \Theta_m } $&deviation&approx.&deviation&of the&deviation&in sec.\\
iteration&&of $ \mathcal{U}^{ \Theta_m } $&error&of the&empirical&of the&for one \\
steps &&&&relative&loss&empirical&realiz.\\
&&&&approx.&function&loss&of $ \mathcal{U}^{ \Theta_m } $\\
&&&&error&&function&\\
\hline
    0 & 
0.6438 & 
0.2506 & 
0.8597 & 
0.05459 & 
8.08967 & 
1.65498 & 
24\\
  500 & 
2.2008 & 
0.1721 & 
0.5205 & 
0.03750 & 
4.44386 & 
0.51459 & 
939\\
 1000 & 
3.6738 & 
0.1119 & 
0.1996 & 
0.02437 & 
1.46137 & 
0.46636 & 
1857\\
 1500 & 
4.4094 & 
0.0395 & 
0.0394 & 
0.00860 & 
0.26111 & 
0.08805 & 
2775\\
 2000 & 
4.5738 & 
0.0073 & 
0.0036 & 
0.00159 & 
0.05641 & 
0.01412 & 
3694\\
\hline
\end{tabular}
\end{center}
\caption{
Numerical simulations of the deep2BSDE method 
in Framework 
\ref{def:general_algorithm} in the case of 
the $100$-dimensional 
Hamilton-Jacobi-Bellman equation 
\eqref{eq:example_hamilton_jacobi_bellman}.
In the approximative calculations of the relative 
$L^1$-approximation errors the value $u(0,(0,0,\ldots,0))$ has 
been replaced by the value 4.5901
which has been calculated by means of the classical 
Monte Carlo method 
(cf.\ {\sc Matlab} code 
\ref{code:MC_HJB} in Appendix
\ref{subsec:MC_HJB} below).
\label{tab:table_HamiltonJacobiBellman.tex}}
\end{table}
\end{center}
\begin{figure}
\includegraphics{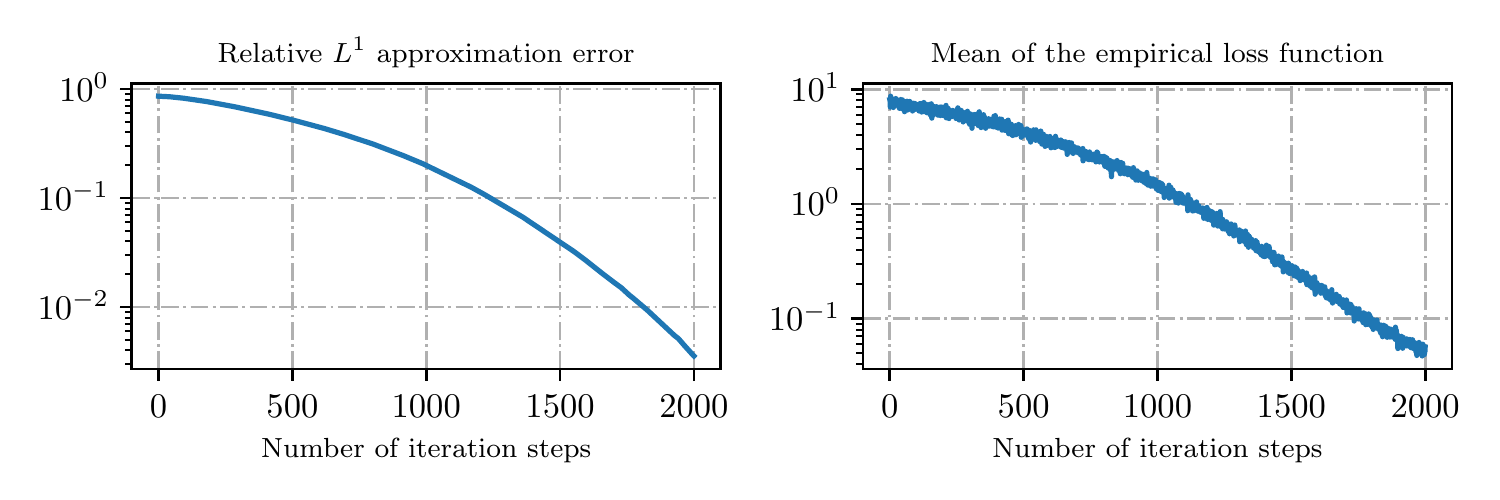}
\input{HamiltonJacobiBellman_Caption}
\end{figure}

\subsection{A $50$-dimensional Allen-Cahn equation}
\label{subsec:example_allen_cahn}
In this subsection we use the deep 2BSDE method 
in Framework \ref{def:general_algorithm} 
to approximatively calculate the solution of a $50$-dimensional 
Allen-Cahn equation with a cubic nonlinearity (see 
\eqref{eq:example_pde_allen_cahn} below).

Assume the setting of Subsection \ref{subsec:example_setting}, 
assume
  $T=\tfrac{3}{10}$, 
  $N=20$, 
  $d=50$, 
  $\varepsilon = 1$, 
assume for all $\omega\in\Omega$ that 
  $\xi(\omega) = (0,0,\ldots,0)\in\R^{50}$,   
and assume for all 
  $m\in\N_0$, 
  $s,t\in [0,T]$,
  $x,w,z\in\R^d$
  $y\in\R$, 
  $S\in\R^{d\times d}$
  that 
  $\sigma(x) = \sqrt{2}\operatorname{Id}_{\R^d}$, 
  $H(s,t,x,w) = x + \sigma(x)w = x + \sqrt{2}w$,
  $g(x)=[2+\tfrac{2}{5}\|x\|_{\R^d}^2]^{-1}$, 
  $
    f(t,x,y,z,S) = -\operatorname{Trace}(S) - y + y^3
  $,
and
  \begin{align}
    \gamma_m = \tfrac{1}{10} \cdot \left[ \tfrac{ 9 }{ 10 } \right]^{ \lfloor \frac{ m }{ 1000 } \rfloor }
    .
  \end{align} 
The solution $u$ to the PDE \eqref{eq:example_PDE_general} then satisfies 
for all $(t,x)\in [0,T)\times\R^d$ that $u(T,x) = g(x)$ and
\begin{align}\label{eq:example_pde_allen_cahn}
 \tfrac{ \partial u }{ \partial t }( t, x ) 
 + 
 \Delta u ( t, x ) 
 + 
 u ( t, x ) 
 - 
 \left[ u ( t, x ) \right]^3
 = 0.
\end{align}
In Table~\ref{tab:table_AllenCahn50.tex} we use an 
adapted version of {\sc Python} code \ref{code:deepPDEmethod} 
in Subsection \ref{subsec:generalCode} below to 
approximatively calculate 
  the mean $\U^{\Theta_m}(\xi)$, 
  the standard deviation of $\U^{\Theta_m}(\xi)$, 
  the relative $L^1$-approximation error associated to $\U^{\Theta_m}(\xi)$, 
  the uncorrected sample standard deviation of the relative approximation error associated to $\U^{\Theta_m}(\xi)$, 
  the mean of the empirical loss function associated to $\U^{\Theta_m}(\xi)$, 
  the standard deviation of the empirical loss function associated to $\U^{\Theta_m}(\xi)$, 
  and the average runtime in seconds needed for calculating one realization of $\U^{\Theta_m}(\xi)$ 
against
  $m\in\{0,500,1000,1500,2000\}$
based on $10$ independent realizations ($10$ independent runs). 
In addition, Figure~\ref{fig:figure_AllenCahn50} depicts approximations 
of the relative $L^1$-approximation error and 
approximations of the mean of the empirical loss function associated 
to $\Theta_m$ against $m\in\{0,1,\ldots,2000\}$ based on $10$ independent realizations ($10$ independent runs). 
In the approximate calculations of the relative $L^1$-approximation errors 
in Table \ref{tab:table_AllenCahn50.tex} and Figure \ref{fig:figure_AllenCahn50} 
the value $u(0,(0,0,\ldots,0))$ of the solution $u$ of the PDE \eqref{eq:example_pde_allen_cahn}
has been replaced by the value $0.09909$ which, in turn, has been 
calculated through the Branching diffusion method 
(cf. {\sc Matlab} code \ref{code:branchingMatlab} in Subsection~\ref{subsec:BranchingMatlab} 
below). 

\begin{center}
\begin{table}
\begin{center}
\begin{tabular}{|c|c|c|c|c|c|c|c|}
\hline
Number&Mean&Standard&Rel.\ $L^1$-&Standard&Mean&Standard&Runtime\\
of&of $ \mathcal{U}^{ \Theta_m } $&deviation&approx.&deviation&of the&deviation&in sec.\\
iteration&&of $ \mathcal{U}^{ \Theta_m } $&error&of the&empirical&of the&for one \\
steps &&&&relative&loss&empirical&realiz.\\
&&&&approx.&function&loss&of $ \mathcal{U}^{ \Theta_m } $\\
&&&&error&&function&\\
\hline
    0 & 
0.5198 & 
0.19361 & 
4.24561 & 
1.95385 & 
0.5830 & 
0.4265 & 
22\\
  500 & 
0.0943 & 
0.00607 & 
0.06257 & 
0.04703 & 
0.0354 & 
0.0072 & 
212\\
 1000 & 
0.0977 & 
0.00174 & 
0.01834 & 
0.01299 & 
0.0052 & 
0.0010 & 
404\\
 1500 & 
0.0988 & 
0.00079 & 
0.00617 & 
0.00590 & 
0.0008 & 
0.0001 & 
595\\
 2000 & 
0.0991 & 
0.00046 & 
0.00371 & 
0.00274 & 
0.0003 & 
0.0001 & 
787\\
\hline
\end{tabular}
\end{center}
\caption{
Numerical simulations of the deep2BSDE method 
in Framework 
\ref{def:general_algorithm} in the case of 
the $50$-dimensional
Allen-Cahn equation 
\eqref{eq:example_pde_allen_cahn}.
In the approximative calculations of the relative 
$L^1$-approximation errors the value $u(0,(0,0,\ldots,0))$ has 
 been replaced by the value 0.09909
which has been calculated through the Branching diffusion 
method (cf. {\sc Matlab} code 
\ref{code:branchingMatlab} in Subsection 
\ref{subsec:BranchingMatlab} below).
\label{tab:table_AllenCahn50.tex}}
\end{table}
\end{center}

\begin{figure}
 \includegraphics{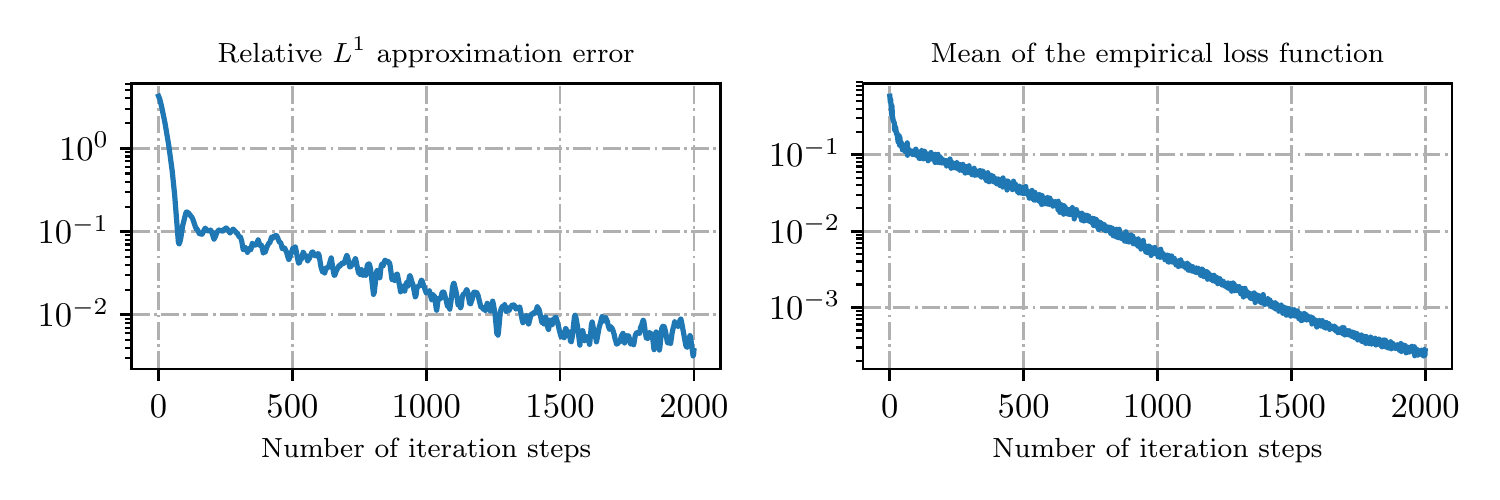}
 \input{AllenCahn50_Caption}
\end{figure}

\subsection{$ G $-Brownian motions in $ 1 $ and $ 100 $ space-dimensions}
\label{subsec:example_gbm}
In this subsection we use the deep 2BSDE method in 
Framework~\ref{def:general_algorithm} 
to approximatively calculate nonlinear expectations 
of a test function on a $ 100 $-dimensional $G$-Brownian motion 
and of a test function on a $ 1 $-dimensional $G$-Brownian motion. 
In the case of the $ 100 $-dimensional $ G $-Brownian motion 
we consider a specific test function such that 
the nonlinear expectation of this function on the $ 100 $-dimensional 
$ G $-Brownian motion admits an explicit analytic solution 
(see Lemma~\ref{lem:gbm_analytic} below). 
In the case of the $ 1 $-dimension $ G $-Brownian motion 
we compare the numerical results of the deep 2BSDE method 
with numerical results obtained by a finite difference approximation method.

Assume the setting of  
Subsection~\ref{subsec:example_setting}, 
assume 
  $T=1$, 
  $N=20$, 
  $\varepsilon = 10^{-8}$, 
let 
 $\sigma_{\text{max}} = 1$, 
 $\sigma_{\text{min}} = \tfrac{1}{\sqrt{2}}$, 
let 
 $\bar\sigma\colon\R\to\R$ be the function which satisfies 
 for all $x\in\R$ that  
 \begin{align}
  \bar\sigma(x) = 
  \begin{cases}
   \sigma_{\text{max}} & \colon x\geq 0 \\
   \sigma_{\text{min}} & \colon x < 0
  \end{cases},
 \end{align}
assume for all 
 $s,t\in [0,T]$, 
 $x=(x_1,\ldots,x_d),
 w=(w_1,\ldots,w_d),
 z=(z_1,\ldots,z_d)\in\R^d$, 
 $y\in\R$, 
 $S=(S_{ij})_{(i,j)\in\{1,\ldots,d\}^2}\in\R^{d\times d}$
that 
 $\sigma(x) = \operatorname{Id}_{ \R^d }$, 
 $H(s,t,x,w) = x + w$,
 $g(x) = \|x\|_{\R^d}^2$, 
 and 
 \begin{align}
  f(t,x,y,z,S) 
  = 
  -\tfrac12\sum_{i=1}^d \bigl[\bar\sigma(S_{ii})\bigr]^2S_{ii}.
 \end{align}
The solution $u\colon [0,T]\times\R^d\to\R$ of the 
PDE~\eqref{eq:example_PDE_general} then
satisfies for all $(t,x)\in [0,T)\times\R^d$ that 
$ u(T,x) = g(x) $ 
and 
\begin{align}\label{eq:example_gbm}
  \tfrac{\partial u}{\partial t}(t,x)  
  + \tfrac12
  {\textstyle 
  \sum\limits_{i=1}^d
  }
  \big|
    \bar\sigma\big(
      \tfrac{ \partial^2 u }{ \partial x_i^2 }( t, x) 
    \big)
  \big|^2
  \tfrac{\partial^2 u}{\partial x_i^2}(t,x)
  =
  0
  . 
\end{align}
\begin{center}
\begin{table}
\begin{center}
\begin{tabular}{|c|c|c|c|c|c|c|c|}
\hline
Number&Mean&Standard&Rel.\ $L^1$-&Standard&Mean&Standard&Runtime\\
of&of $ \mathcal{U}^{ \Theta_m } $&deviation&approx.&deviation&of the&deviation&in sec.\\
iteration&&of $ \mathcal{U}^{ \Theta_m } $&error&of the&empirical&of the&for one \\
steps &&&&relative&loss&empirical&realiz.\\
&&&&approx.&function&loss&of $ \mathcal{U}^{ \Theta_m } $\\
&&&&error&&function&\\
\hline
    0 & 
0.46 & 
0.35878 & 
0.99716 & 
0.00221 & 
26940.83 & 
676.70 & 
24\\
  500 & 
164.64 & 
1.55271 & 
0.01337 & 
0.00929 & 
13905.69 & 
2268.45 & 
757\\
 1000 & 
162.79 & 
0.35917 & 
0.00242 & 
0.00146 & 
1636.15 & 
458.57 & 
1491\\
 1500 & 
162.54 & 
0.14143 & 
0.00074 & 
0.00052 & 
403.00 & 
82.40 & 
2221\\
\hline
\end{tabular}
\end{center}
\caption{
Numerical simulations of the deep2BSDE method in Framework 
\ref{def:general_algorithm} in the case of 
the $100$-dimensional
$G$-Brownian motion
(cf.\ \eqref{eq:example_gbm} 
and \eqref{eq:specifications_GBM100}).
In the approximative calculations of 
the relative $L^1$-approximation errors the value 
$u(0,(1,\nicefrac12,1,\nicefrac12,\ldots,1,\nicefrac12))$
has been replaced by the value 
$162.5$ which has been calculated 
by means of Lemma 
\ref{lem:gbm_analytic}.
\label{tab:table_GBM100.tex}}
\end{table}
\end{center}
\begin{figure}
 \includegraphics{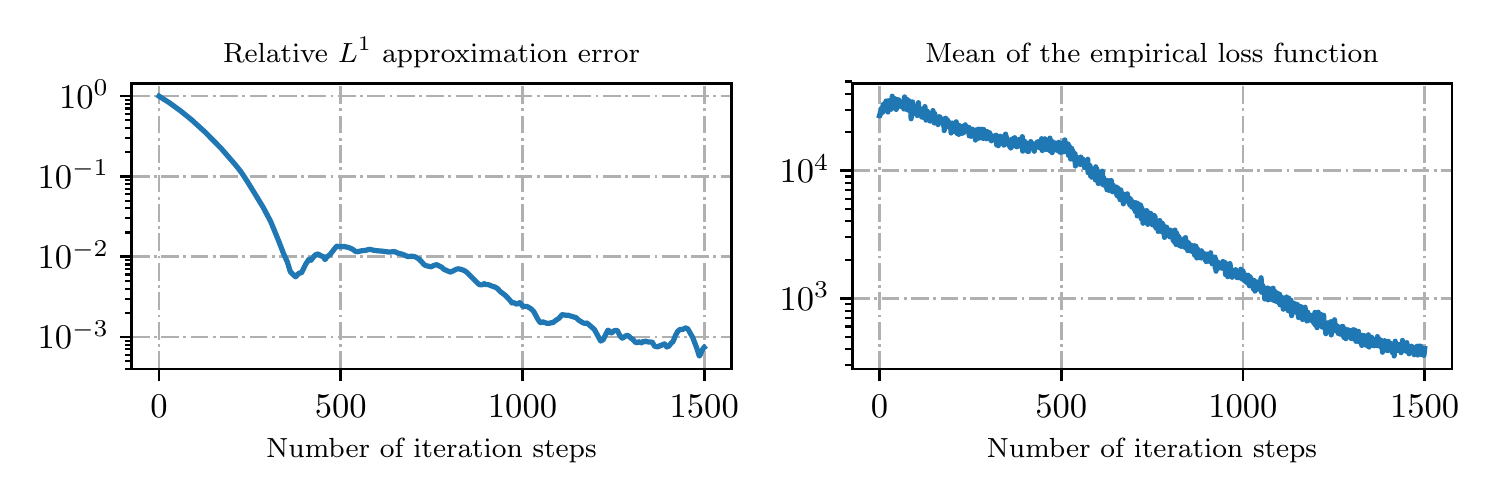}
 \input{GBM100_Caption} 
\end{figure}
In Table~\ref{tab:table_GBM100.tex} 
we use an adapted version of {\sc Python} code~\ref{code:deepPDEmethod} 
to approximatively calculate 
  the mean $\U^{\Theta_m}(\xi)$, 
  the standard deviation of $\U^{\Theta_m}(\xi)$, 
  the relative $L^1$-approximation error associated to $\U^{\Theta_m}(\xi)$, 
  the uncorrected sample standard deviation of the relative approximation error 
  associated to $\U^{\Theta_m}(\xi)$, 
  the mean of the empirical loss function associated to $\U^{\Theta_m}(\xi)$, 
  the standard deviation of the empirical loss function associated to $\U^{\Theta_m}(\xi)$, 
  and the average runtime in seconds needed for calculating one realization of $\U^{\Theta_m}(\xi)$ 
against $m\in\{0,500,1000,1500,2000\}$ based on $10$ realizations ($10$ independent runs) in the case 
where for all $ x \in \R^d $, $ m \in \N_0 $, $ \omega \in \Omega $ it holds that
\begin{equation}
\label{eq:specifications_GBM100}
\begin{split}
   d = 100, 
   \;\;
   g(x) = \| x \|_{ \R^d }^2 
   ,
   \;\;
   \gamma_m = \left[ \tfrac{ 1 }{ 2 } \right]^{ \lfloor \frac{ m }{ 500 } \rfloor }
   ,
   \;\;
   \text{and}
   \;\;
  \xi( \omega ) = 
  (1,\tfrac12,1,\tfrac12,\ldots,1,\tfrac12) \in \R^d 
  .
\end{split}
\end{equation}
In addition, Figure~\ref{fig:figure_gbm100} depicts approximations of the relative $L^1$-approximation error 
associated to $\U^{\Theta_m}(\xi)$ and approximations of mean of the empirical loss function associated to $\Theta_m$ 
against $m\in \{0,1,\ldots,2000\}$ based on $10$ independent realizations ($10$ independent runs) 
in the case of \eqref{eq:specifications_GBM100}. 
In the approximative calculations of the relative 
$ L^1 $-approximation errors in Table~\ref{tab:table_GBM100.tex}
and Figure~\ref{fig:figure_gbm100}
the value 
$ u( 0, (1,\nicefrac12,1,\nicefrac12,\ldots,1,\nicefrac12)) $ of the solution $ u $ of the PDE 
(cf.\ \eqref{eq:example_gbm} and \eqref{eq:specifications_GBM100})
has been replaced by the value 
$ 162.5 $ which, in turn, has been calculated 
by means of Lemma~\ref{lem:gbm_analytic}~below 
(with $ c = 1 $, $ \sigma_{ \max } = 1 $, $ T = 1 $, $ \sigma_{ \min } = \nicefrac{ 1 }{ \sqrt{2} } $, 
$ d = 100 $ in the notation of Lemma~\ref{lem:gbm_analytic} below). 
\begin{lemma}\label{lem:gbm_analytic}
Let $c,\sigma_{\max}, T\in (0,\infty)$, $\sigma_{\min}\in (0,\sigma_{\max})$, $d\in\N$, 
let $\bar\sigma\colon\R\to\R$ be the function which satisfies 
for all $x\in\R$ that 
\begin{align}\label{lem:gbm_sigmabar}
 \bar{\sigma}(x) = \begin{cases}
                    \sigma_{\max} & \colon x\geq 0 \\
                    \sigma_{\min} & \colon x < 0
                   \end{cases},
\end{align}
and let 
  $g\colon\R^d\to\R$ 
  and 
  $u\colon [0,T]\times\R^d\to\R$
  be the functions which satisfy 
  for all 
   $t\in [0,T]$, 
   $x = (x_1,\ldots,x_d)\in \R^d$ 
  that 
  $g(x) = c\|x\|^2_{\R^d} = c\sum_{i=1}^d |x_i|^2$ and 
\begin{align}
\label{eq:def_u_lemma2}
 u(t,x) = g(x) + c d |\sigma_{\max}|^2 (T-t).
\end{align}
Then it holds for all 
  $t\in [0,T]$, 
  $x=(x_1,\ldots,x_d)\in\R^d$ 
  that 
  $u\in C^{\infty}([0,T]\times\R^d,\R)$, 
  $u(T,x) = g(x)$,  
  and 
  \begin{align}
  \tfrac{\partial u}{\partial t}(t,x)  
  +
  \tfrac12
  {\textstyle
    \sum\limits_{i=1}^d 
  }
  \big|
    \bar{\sigma}\big(
      \tfrac{ \partial^2 u }{ \partial x_i^2 }(t,x)
    \big)
  \big|^2
  \tfrac{\partial^2 u}{\partial x_i^2}(t,x)
  = 0 .
 \end{align}
\end{lemma}
\begin{proof}[Proof of Lemma~\ref{lem:gbm_analytic}]
Observe that the function $u$ is clearly infinitely often differentiable.
Next note that \eqref{eq:def_u_lemma2} ensures that
for all 
 $t\in [0,T]$, 
 $x=(x_1,\ldots,x_d)\in\R^d$, 
$i\in\{1,2,\ldots,d\}$ 
it holds that 
\begin{equation}\label{eq:lem_gbm_partial_u_time}
 \tfrac{\partial u}{\partial t}(t,x) 
 = 
 \tfrac{\partial}{\partial t}
 \bigl[ 
 g(x) + c d |\sigma_{\max}|^2 (T-t)
 \bigr]
 = 
 \tfrac{\partial}{\partial t}
 \bigl[ 
 c d |\sigma_{\max}|^2 (T-t)
 \bigr]
 =
 -c d |\sigma_{\max}|^2
\end{equation}
and
\begin{equation}\label{eq:lem_gbm_second_partial_u}
 \tfrac{\partial^2 u}{\partial x_i^2}(t,x) 
 = 
 \tfrac{\partial^2}{\partial x_i^2}
 \bigl[
 g(x) + c d |\sigma_{\max}|^2 (T-t)
 \bigr]
 = \tfrac{\partial^2 g}{\partial x_i^2}(x) = 2 c > 0.
 \end{equation}
Combining this with \eqref{lem:gbm_sigmabar} shows that 
for all 
  $t\in [0,T]$, 
  $x=(x_1,\ldots,x_d)\in\R^d$, 
  $i\in\{1,2,\ldots,d\}$ 
  it holds that
\begin{equation}
\bar\sigma\bigl(\tfrac{\partial^2 u}{\partial x_i^2}(t,x)\bigr) 
= \bar\sigma(2c) 
= \sigma_{\max}.
\end{equation}
This, \eqref{eq:lem_gbm_partial_u_time}, and \eqref{eq:lem_gbm_second_partial_u} 
yield that for all $t\in [0,T]$, $x=(x_1,\ldots,x_d)\in \R^d$ it holds that
\begin{equation}
\begin{split}
  \tfrac{\partial u}{\partial t}(t,x)
  +
  \tfrac12
  {\textstyle
    \sum\limits_{i=1}^d 
  }
  \big|
    \bar\sigma\big(
      \tfrac{\partial^2 u}{\partial x_i^2}(t,x)
     \big)
   \big|^2
   \tfrac{\partial^2 u}{\partial x_i^2}(t,x)
& =
  \tfrac{\partial u}{\partial t}(t,x)
  +
  \tfrac12
  {\textstyle
    \sum\limits_{i=1}^d 
  }
  |
    \sigma_{ \max } 
  |^2
  \tfrac{\partial^2 u}{\partial x_i^2}(t,x)
\\ & =
  \tfrac{\partial u}{\partial t}(t,x)
  +
  c d
  |
    \sigma_{ \max } 
  |^2
  =
  0
 .
\end{split}
\end{equation}
This completes the proof of Lemma~\ref{lem:gbm_analytic}.
\end{proof}
\begin{center}
\begin{table}
\begin{center}
\begin{tabular}{|c|c|c|c|c|c|c|c|}
\hline
Number&Mean&Standard&Rel.\ $L^1$-&Standard&Mean&Standard&Runtime\\
of&of $ \mathcal{U}^{ \Theta_m } $&deviation&approx.&deviation&of the&deviation&in sec.\\
iteration&&of $ \mathcal{U}^{ \Theta_m } $&error&of the&empirical&of the&for one \\
steps &&&&relative&loss&empirical&realiz.\\
&&&&approx.&function&loss&of $ \mathcal{U}^{ \Theta_m } $\\
&&&&error&&function&\\
\hline
    0 & 
0.4069 & 
0.28711 & 
0.56094 & 
0.29801 & 
29.905 & 
25.905 & 
22\\
  100 & 
0.8621 & 
0.07822 & 
0.08078 & 
0.05631 & 
1.003 & 
0.593 & 
24\\
  200 & 
0.9097 & 
0.01072 & 
0.00999 & 
0.00840 & 
0.159 & 
0.068 & 
26\\
  300 & 
0.9046 & 
0.00320 & 
0.00281 & 
0.00216 & 
0.069 & 
0.048 & 
28\\
  500 & 
0.9017 & 
0.00159 & 
0.00331 & 
0.00176 & 
0.016 & 
0.005 & 
32\\
\hline
\end{tabular}
\end{center}
\caption{
Numerical simulations of the deep2BSDE method in Framework 
\ref{def:general_algorithm} in the case of 
the $1$-dimensional
$G$-Brownian motion
(cf.\ \eqref{eq:example_gbm} and 
\eqref{eq:specifications_GBM1}).
In the approximative calculations of the relative 
$L^1$-approximation error the value $u(0,-2)$ has 
been replaced by the value 0.90471
which has been calculated through finite differences 
approximations
(cf.\ {\sc Matlab} code \ref{code:finiteMatlab} 
in Subsection \ref{subsec:code_finiteMatlab} below).}
\label{tab:table_GBM1.tex}
\end{table}
\end{center}
\begin{figure}
 \includegraphics{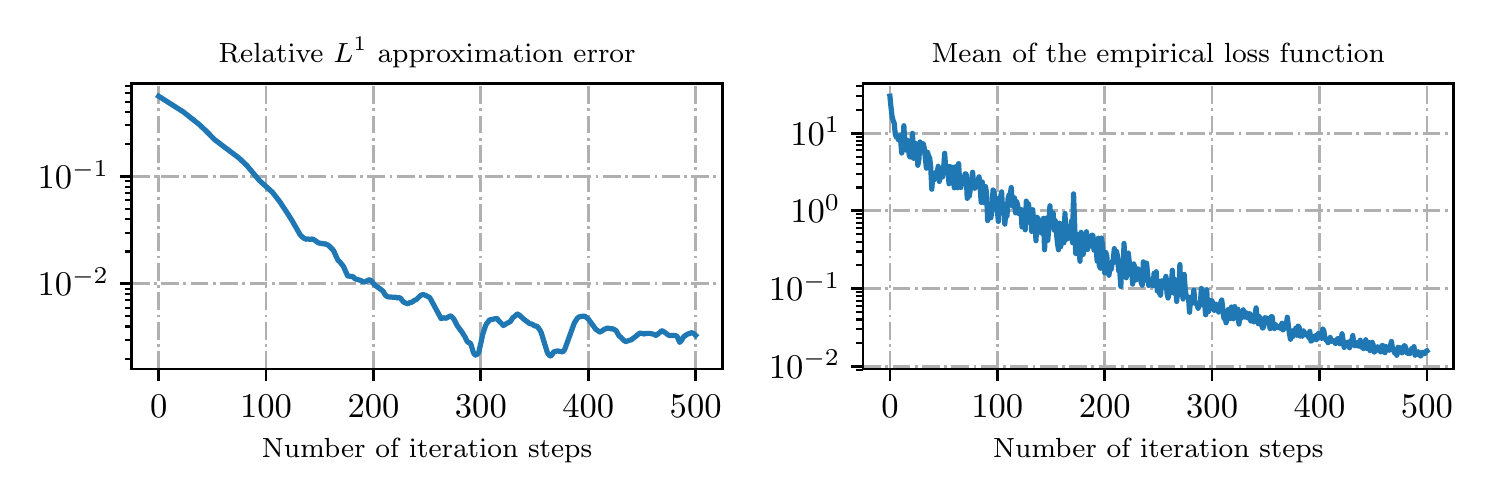}
 \input{GBM1_Caption} 
\end{figure}
In Table~\ref{tab:table_GBM1.tex} we use an adapted version of 
{\sc Python} code \ref{code:deepPDEmethod} in Subsection \ref{subsec:generalCode} below 
to approximatively calculate 
  the mean of $\U^{\Theta_m}(\xi)$, 
  the standard deviation of $\U^{\Theta_m}(\xi)$, 
  the relative $L^1$-approximation error associated to $\U^{\Theta_m}(\xi)$, 
  the uncorrected sample standard deviation of the relative approximation error associated to $\U^{\Theta_m}(\xi)$, 
  the mean of the empirical loss function associated to $\U^{\Theta_m}(\xi)$, 
  the standard deviation of the empirical loss function associated to $\U^{\Theta_m}(\xi)$, 
  and the average runtime in seconds needed for calculating one realization of $\U^{\Theta_m}(\xi)$ 
against $m\in\{0,100,200,300,500\}$ based on $10$ realizations ($10$ independent runs)
in the case where for all $ x \in \R^d $, $ m \in \N_0 $, $ \omega \in \Omega $ 
it holds that 
\begin{equation}
\label{eq:specifications_GBM1}
\begin{split}
 d = 1,
 \qquad
 g(x) = \tfrac{1}{ 1 + \exp( - x^2 ) },
 \qquad
 \gamma_m = \tfrac{1}{100},
 \qquad
 \text{and}
 \qquad
 \xi(\omega) = -2.  
\end{split}
\end{equation}
In addition, Figure \ref{fig:figure_gbm1} depicts approximations of the relative $L^1$-approximation error 
associated to $\U^{\Theta_m}(\xi)$ and approximations of the mean of empirical loss function associated to $\Theta_m$ 
for $m\in \{0,1,\ldots,500\}$ based on $10$ independent realizations ($10$ independent runs)
in the case of \eqref{eq:specifications_GBM1}.
In the approximative calculations of the relative 
$ L^1 $-approximation errors in Table~\ref{tab:table_GBM100.tex}
and Figure~\ref{fig:figure_gbm100}
the value 
$ u( 0, - 2 ) $ of the solution $ u $ of the PDE 
(cf.\ \eqref{eq:example_gbm} and \eqref{eq:specifications_GBM1})
has been replaced by the value 
$ 0.90471 $ which, in turn, has been calculated 
by means of finite differences approximations 
(cf.\ {\sc Matlab} code \ref{code:finiteMatlab} 
in Subsection \ref{subsec:code_finiteMatlab} below).

\begin{appendix}
\section{Source codes}\label{sec:source_code}
  \subsection{A {\sc Python} code for the deep 2BSDE method used in Subsection~\ref{subsec:example_allen_cahn_plain_sgd_no_bn}}
\label{subsec:plainSGDCode}
The following {\sc Python} code, {\sc Python} code~\ref{code:deepPDEmethodPlainSGDNoBN} below, 
is a simplified version of {\sc Python} code~\ref{code:deepPDEmethod} in Subsection~\ref{subsec:generalCode} below.
\begin{lstlisting}[language=Python,
caption={A {\sc Python} code for the deep 2BSDE method used in Subsection~\ref{subsec:example_allen_cahn_plain_sgd_no_bn}. 
This {\sc Python} code uses the plain stochastic gradient descent method and does not use batch normalization.},
label=code:deepPDEmethodPlainSGDNoBN]
#! /usr/bin/python3
    
"""
Plain deep2BSDE solver with hard-coded Allen-Cahn equation
"""

import tensorflow as tf
import numpy as np
import time, datetime

tf.reset_default_graph()
start_time = time.time()
    
name = 'AllenCahn'
    
# setting of the problem 
d = 20
T = 0.3
Xi = np.zeros([1,d])
    
# setup of algorithm and implementation    
N = 20 
h = T/N
sqrth = np.sqrt(h)
n_maxstep = 10000
batch_size = 1
gamma = 0.001
    
# neural net architectures 
n_neuronForGamma = [d, d, d, d**2]
n_neuronForA = [d, d, d, d]
    
# (adapted) rhs of the pde
def f0(t,X,Y,Z,Gamma):
    return -Y+tf.pow(Y, 3)
    
# terminal condition
def g(X):
    return 1/(1 + 0.2*tf.reduce_sum(tf.square(X), 
                        1, keep_dims=True))*0.5

# helper functions for constructing the neural net(s)
def _one_time_net(x, name, isgamma = False):
    with tf.variable_scope(name):
        layer1 = _one_layer(x, (1-isgamma)*n_neuronForA[1] \
                    +isgamma*n_neuronForGamma[1], name = 'layer1')
        layer2 = _one_layer(layer1, (1-isgamma)*n_neuronForA[2] \
                    +isgamma*n_neuronForGamma[2], name = 'layer2')            
        z = _one_layer(layer2, (1-isgamma)*n_neuronForA[3] \
                    +isgamma*n_neuronForGamma[3], activation_fn=None,
                    name = 'final')
        return z
             
def _one_layer(input_, output_size, activation_fn=tf.nn.relu,
        stddev=5.0, name='linear'):
    with tf.variable_scope(name):
        shape = input_.get_shape().as_list()
        w = tf.get_variable('Matrix', [shape[1], output_size],
                tf.float64,
                tf.random_normal_initializer(
                stddev=stddev/np.sqrt(shape[1]+output_size)))
        b = tf.get_variable('Bias', [1,output_size], tf.float64, 
                tf.constant_initializer(0.0))
        hidden = tf.matmul(input_, w) + b
        if activation_fn:
            return activation_fn(hidden)
        else:
            return hidden
                 
with tf.Session() as sess:
    # background dynamics   
    dW = tf.random_normal(shape=[batch_size, d], stddev = sqrth, 
                          dtype=tf.float64)
    
    # initial values of the stochastic processes
    X = tf.Variable(np.ones([batch_size, d]) * Xi,
                    dtype = tf.float64, 
                    name='X', 
                    trainable=False)
    Y0 = tf.Variable(tf.random_uniform([1], 
                    minval = -1, maxval = 1, 
                    dtype=tf.float64), 
                    name='Y0')
    Z0 = tf.Variable(tf.random_uniform([1,d],
                    minval = -.1, maxval = .1, 
                    dtype=tf.float64), 
                    name='Z0')
    Gamma0 = tf.Variable(tf.random_uniform([d, d], 
                    minval = -.1, maxval = .1, 
                    dtype=tf.float64), 
                    name='Gamma0')
    A0 = tf.Variable(tf.random_uniform([1,d], 
                    minval = -.1, maxval = .1, 
                    dtype=tf.float64), 
                    name='A0')
    allones = tf.ones(shape=[batch_size, 1], 
                    dtype=tf.float64, 
                    name='MatrixOfOnes')
    Y = allones * Y0
    Z = tf.matmul(allones,Z0)
    Gamma = tf.multiply(tf.ones([batch_size, d, d],
                                dtype = tf.float64), Gamma0)
    A = tf.matmul(allones,A0)
    
    # forward discretization
    with tf.variable_scope('forward'):
        for i in range(N-1):
            Y = Y + f0(i*h,X,Y,Z,Gamma)*h \
                  + tf.reduce_sum(dW*Z, 1, keep_dims=True)
            Z = Z + A * h \
                  + tf.squeeze(tf.matmul(Gamma,
                                         tf.expand_dims(dW, -1)))
            Gamma = tf.reshape(_one_time_net(X, name=str(i)+'Gamma', 
                                            isgamma=True)/d**2, 
                                            [batch_size, d, d])
            if i!=N-1:
                A = _one_time_net(X, name=str(i)+'A')/d 
            X = X + dW
            dW = tf.random_normal(shape=[batch_size, d], 
                stddev = sqrth, dtype=tf.float64)
            
        Y = Y + f0( (N-1)*h , X,Y,Z,Gamma)*h \
              + tf.reduce_sum(dW*Z, 1, keep_dims=True)
        X = X + dW
        loss_function = tf.reduce_mean(tf.square(Y-g(X)))
        
    # specifying the optimizer    
    global_step = tf.get_variable('global_step', [],
                    initializer=tf.constant_initializer(0),
                    trainable=False, dtype=tf.int32)
                  
    learning_rate = tf.train.exponential_decay(gamma, global_step, 
        decay_steps = 10000, decay_rate = 0.0, staircase = True)
                         
    trainable_variables = tf.trainable_variables()
    grads = tf.gradients(loss_function, trainable_variables)
    optimizer = tf.train.GradientDescentOptimizer(
        learning_rate=learning_rate)
    apply_op = optimizer.apply_gradients(
        zip(grads, trainable_variables),
        global_step=global_step, name='train_step')
     
    with tf.control_dependencies([apply_op]):
        train_op_2 = tf.identity(loss_function, name='train_op2')
                              
    # to save history
    learning_rates = []
    y0_values = []
    losses = []
    running_time = []
    steps = []
    sess.run(tf.global_variables_initializer())
                  
    try: 
        # the actual training loop                                    
        for _ in range(n_maxstep + 1):
            y0_value, step = sess.run([Y0, global_step])
            currentLoss, currentLearningRate = sess.run(
                [train_op_2, learning_rate])
                        
            learning_rates.append(currentLearningRate)
            losses.append(currentLoss)
            y0_values.append(y0_value)
            running_time.append(time.time()-start_time)
            steps.append(step)
                               
            if step % 100 == 0:            
                print("step: ", step,
                      " loss: ", currentLoss,
                      " Y0: " , y0_value,
                      " learning rate: ", currentLearningRate)
                  
        end_time = time.time()
        print("running time: ", end_time-start_time)        

    except KeyboardInterrupt:
        print("\nmanually disengaged")
     
# writing results to a csv file
output = np.zeros((len(y0_values),5))
output[:,0] = steps
output[:,1] = losses
output[:,2] = y0_values
output[:,3] = learning_rates
output[:,4] = running_time
    
np.savetxt("./" + str(name) + "_d" + str(d) + "_" + \
    datetime.datetime.now().strftime('%Y-%m-%d-%H:%M:%S')+ ".csv", 
    output, 
    delimiter = ",", 
    header = "step, loss function, Y0, learning rate, running time", 
    )
\end{lstlisting}
\subsection{A {\sc Matlab} code for the Branching diffusion 
method used in Subsection~\ref{subsec:example_allen_cahn_plain_sgd_no_bn}}
\label{subsec:BranchingMatlab}
The following {\sc Matlab} code is a slightly 
modified version of the {\sc Matlab} code 
in E, Han, \& Jentzen~\cite[Subsection~6.2]{EHanJentzen2017}.
\renewcommand{\lstlistingname}{{\sc Matlab} code}
\begin{lstlisting}[language=Matlab, 
caption={A {\sc Matlab} code for the Branching diffusion 
method used in Subsection~\ref{subsec:example_allen_cahn}.},
label=code:branchingMatlab]
function Branching()
  rng('default');

% Parameter
  T = 0.3;
  M = 10^7;
  t0 = 0;

  d = 20; m = d;
  beta = 1;
  mu = zeros(d,1);
  sigma = eye(d);
  a = [0 2 0 -1]';
  p = [0 0.5 0 0.5]';
  psi = @(x) 1./(1+0.2*norm(x)^2)*1/2;
  x0 = 0;

% Branching PDE Solver
  tic;
  [result,SD] = ...
    MCBranchingEvaluation(...
        mu, sigma, beta, p, a, t0, x0, T, psi, M);
  runtime = toc;
  disp(['BranchingP solver: u(0,x0) = ' num2str(result) ';']);
  disp(['Runtime = ' num2str(runtime) ';']);
end

  function [result,SD] = ...
  MCBranchingEvaluation(...
        mu, sigma, beta, p, a, t0, x0, T, psi, M)
    result = 0;
    SD = 0;
    for m=1:M
      Evl = BranchingEvaluation(...
        mu, sigma, beta, p, a, t0, x0, T, psi);
      result = result + Evl;
      SD = SD + Evl^2;
    end 
    SD = sqrt( (SD - result^2/M)/M );
    result = result/M;
  end

  function Evl = ...
  BranchingEvaluation(...
        mu, sigma, beta, p, a, t0, x0, T, psi)
    BP = BranchingProcess(mu, sigma, beta, p, t0, x0, T);
    Evl = 1;
    for k=1:size(BP{1},2)
      Evl = Evl * psi( BP{1}(:,k) );
    end  
    if norm(a-p) > 0
      for k=1:length(a)
        if p(k) > 0
          Evl = Evl * ( a(k)/p(k) )^( BP{2}(k) );
        elseif a(k) ~= 0
          error('a(k) zero but p(k) non-zero');
        end
      end
    end
  end

  function BP = ...
  BranchingProcess(mu, sigma, beta, p, t0, x0, T)
    BP = cell(2,1);
    BP{2} = p*0;
    tau = exprnd(1/beta);
    new_t0 = min( tau + t0, T );
    delta_t = new_t0 - t0;
    m = size(sigma,2);
    new_x0 = x0 + mu*delta_t + sigma*sqrt(delta_t)*randn(m,1);
    if tau >= T - t0
      BP{1} = new_x0;
    else 
      [tmp,which_nonlinearity] = max(mnrnd(1,p));
      BP{2}(which_nonlinearity) = BP{2}(which_nonlinearity) + 1;
      for k=1:which_nonlinearity-1
        tmp = BranchingProcess(...
			mu, sigma, beta, p, new_t0, new_x0, T);
        BP{1} = [ BP{1} tmp{1} ];
        BP{2} = BP{2} + tmp{2};
      end
    end
  end

\end{lstlisting}
\subsection{A {\sc Python} code for the deep 2BSDE 
method used in Subsection~\ref{subsec:example_bsb}}
\label{subsec:generalCode}
The following {\sc Python} code is based on the {\sc Python} code 
in E, Han, \& Jentzen~\cite[Subsection~6.1]{EHanJentzen2017}.
\renewcommand{\lstlistingname}{{\sc Python} code}
\begin{lstlisting}[language=Python,
caption={A {\sc Python} code for the deep 2BSDE 
method used in Subsection~\ref{subsec:example_bsb}.},
label=code:deepPDEmethod]
#!/usr/bin/python3
# -*- coding: utf-8 -*-

"""
Deep2BSDE solver with hard-coded Black-Scholes-Barenblatt equation.
"""

import time, datetime
import tensorflow as tf
import numpy as np
from tensorflow.python.training import moving_averages

start_time = time.time()    
tf.reset_default_graph()

name = 'BSB'
d = 100
batch_size = 64
T = 1.0
N = 20
h = T/N
sqrth = np.sqrt(h)
n_maxstep = 500
n_displaystep = 100
n_neuronForA = [d,d,d,d]
n_neuronForGamma = [d,d,d,d**2]
Xinit = np.array([1.0,0.5]*50)
mu = 0
sigma = 0.4
sigma_min = 0.1
sigma_max = 0.4
r = 0.05

_extra_train_ops = []

def sigma_value(W):
    return sigma_max * \
        tf.cast(tf.greater_equal(W, tf.cast(0,tf.float64)),
               tf.float64) + \
        sigma_min * tf.cast(tf.greater(tf.cast(0,tf.float64), W), 
               tf.float64)
    
def f_tf(t, X, Y, Z, Gamma):
    return -0.5*tf.expand_dims(tf.trace(
        tf.square(tf.expand_dims(X,-1)) * \
            (tf.square(sigma_value(Gamma))-sigma**2) * Gamma),-1) + \
        r * (Y - tf.reduce_sum(X*Z, 1, keep_dims = True)) 
            
def g_tf(X):        
    return tf.reduce_sum(tf.square(X),1, keep_dims = True)
        
def sigma_function(X):
    return sigma * tf.matrix_diag(X)
    
def mu_function(X):
    return mu * X
    
def _one_time_net(x, name, isgamma = False):
    with tf.variable_scope(name):
        x_norm =  _batch_norm(x, name='layer0_normalization')
        layer1 = _one_layer(x_norm, (1-isgamma)*n_neuronForA[1] \
                +isgamma*n_neuronForGamma[1], name = 'layer1')
        layer2 = _one_layer(layer1, (1-isgamma)*n_neuronForA[2] \
                +isgamma*n_neuronForGamma[2], name = 'layer2')            
        z = _one_layer(layer2, (1-isgamma)*n_neuronForA[3] \
                +isgamma*n_neuronForGamma[3], activation_fn=None, 
                name = 'final')
    return z
        
def _one_layer(input_, output_size, activation_fn=tf.nn.relu, 
               stddev=5.0, name='linear'):
    with tf.variable_scope(name):
        shape = input_.get_shape().as_list()
        w = tf.get_variable('Matrix', [shape[1], output_size],
                tf.float64,
                tf.random_normal_initializer(
                    stddev=stddev/np.sqrt(shape[1]+output_size)))
        hidden = tf.matmul(input_, w)
        hidden_bn = _batch_norm(hidden, name='normalization')
    if activation_fn:
        return activation_fn(hidden_bn)
    else:
        return hidden_bn
            
def _batch_norm(x, name):
    """Batch normalization"""
    with tf.variable_scope(name):
        params_shape = [x.get_shape()[-1]]
        beta = tf.get_variable('beta', params_shape, tf.float64,
                initializer=tf.random_normal_initializer(
                    0.0, stddev=0.1, dtype=tf.float64))
        gamma = tf.get_variable('gamma', params_shape, tf.float64,
                initializer=tf.random_uniform_initializer(
                    0.1, 0.5, dtype=tf.float64))    
        moving_mean = tf.get_variable('moving_mean', 
                params_shape, tf.float64,
                initializer=tf.constant_initializer(0.0, tf.float64),
                trainable=False)
        moving_variance = tf.get_variable('moving_variance', 
                params_shape, tf.float64,
                initializer=tf.constant_initializer(1.0, tf.float64),
                trainable=False)
        mean, variance = tf.nn.moments(x, [0], name='moments')
        _extra_train_ops.append(
                moving_averages.assign_moving_average(
                        moving_mean, mean, 0.99))  
        _extra_train_ops.append(
                moving_averages.assign_moving_average(
                        moving_variance, variance, 0.99))
        y = tf.nn.batch_normalization(
                        x, mean, variance, beta, gamma, 1e-6)
        y.set_shape(x.get_shape())
        return y
    
with tf.Session() as sess:
            
    dW = tf.random_normal(shape=[batch_size, d], 
                          stddev = 1, dtype=tf.float64)
    X = tf.Variable(np.ones([batch_size, d]) * Xinit,
            dtype=tf.float64, 
            trainable=False)
    Y0 = tf.Variable(tf.random_uniform(
            [1], 
            minval=0, maxval=1, dtype=tf.float64), 
            name='Y0');
    Z0 = tf.Variable(tf.random_uniform(
            [1, d], 
            minval=-.1, maxval=.1,  
            dtype=tf.float64), 
            name='Z0')
    Gamma0 = tf.Variable(tf.random_uniform(
            [d, d], 
            minval=-1, maxval=1,
            dtype=tf.float64), name='Gamma0') 
    A0 = tf.Variable(tf.random_uniform(
            [1, d], minval=-.1, maxval=.1,
            dtype=tf.float64), name='A0')
    allones = tf.ones(shape=[batch_size, 1], dtype=tf.float64, 
            name='MatrixOfOnes') 
    Y = allones * Y0
    Z = tf.matmul(allones, Z0)
    A = tf.matmul(allones, A0)
    Gamma = tf.multiply(tf.ones(shape=[batch_size, d, d], 
                                    dtype=tf.float64), Gamma0) 
    with tf.variable_scope('forward'):
        for t in range(0, N-1):            
            # Y update inside the loop
            dX = mu * X * h + sqrth * sigma * X * dW
            Y = Y + f_tf(t * h, X, Y, Z, Gamma)*h \
                  + tf.reduce_sum(Z*dX, 1, keep_dims = True)
            X = X + dX
            # Z update inside the loop
            Z = Z + A * h \
                  + tf.squeeze(tf.matmul(Gamma, 
                                tf.expand_dims(dX, -1), 
                                transpose_b=False))         
            A = _one_time_net(X, str(t+1)+"A" )/d
            Gamma = _one_time_net(X, str(t+1)+"Gamma", 
                                isgamma = True)/d**2
            Gamma = tf.reshape(Gamma, [batch_size, d, d])
            dW = tf.random_normal(shape=[batch_size, d], 
                                  stddev = 1, dtype=tf.float64)
        # Y update outside of the loop - terminal time step
        dX = mu * X * h + sqrth * sigma * X * dW
        Y = Y + f_tf((N-1) * h, X, Y, Z, Gamma)*h \
              + tf.reduce_sum(Z * dX, 1, keep_dims=True) 
        X = X + dX
        loss = tf.reduce_mean(tf.square(Y-g_tf(X)))
            
    # training operations
    global_step = tf.get_variable('global_step', [],
                    initializer=tf.constant_initializer(0),
                    trainable=False, dtype=tf.int32)
    
    learning_rate = tf.train.exponential_decay(1.0, global_step, 
                    decay_steps = 200, 
                    decay_rate = 0.5, 
                    staircase=True)
    
    trainable_variables = tf.trainable_variables()
    grads = tf.gradients(loss, trainable_variables)
    optimizer = tf.train.AdamOptimizer(
                    learning_rate=learning_rate
                    )
    apply_op = optimizer.apply_gradients(
            zip(grads, trainable_variables),
            global_step=global_step, name='train_step')
    train_ops = [apply_op] + _extra_train_ops
    train_op = tf.group(*train_ops)
                
    with tf.control_dependencies([train_op]):
        train_op_2 = tf.identity(loss, name='train_op2')
                
    # to save history
    learning_rates = []
    y0_values = []
    losses = []
    running_time = []
    steps = []
    sess.run(tf.global_variables_initializer())
    
    try: 
        
        for _ in range(n_maxstep+1):
            step, y0_value = sess.run([global_step, Y0])
            currentLoss, currentLearningRate = sess.run(
                    [train_op_2, learning_rate])
            
            steps.append(step)
            losses.append(currentLoss)
            y0_values.append(y0_value)
            learning_rates.append(currentLearningRate)
            running_time.append(time.time()-start_time)
                            
            if step % n_displaystep == 0:            
                print("step: ", step, 
                      " loss: ", currentLoss, 
                      " Y0: " , y0_value, 
                      " learning rate: ", currentLearningRate)
                
        end_time = time.time()
        print("running time: ", end_time-start_time)
        
    except KeyboardInterrupt:
        print("manually disengaged")

# writing results to a csv file    
output = np.zeros((len(y0_values),5))
output[:,0] = steps
output[:,1] = losses
output[:,2] = y0_values
output[:,3] = learning_rates
output[:,4] = running_time
np.savetxt("./" + str(name) + "_d" + str(d) + "_" + \
    datetime.datetime.now().strftime('%Y-%m-%d-%H:%M:%S') + ".csv", 
    output, 
    delimiter = ",", 
    header = "step, loss function, Y0, learning rate, running time"
    )
\end{lstlisting}

\subsection{A {\sc Matlab} code for the classical Monte Carlo method 
used in Subsection~\ref{subsec:example_square_gradient}}
\label{subsec:MC_HJB}
The following {\sc Matlab} code is a slightly 
modified version of the {\sc Matlab} code 
in E, Han, \& Jentzen~\cite[Subsection~6.3]{EHanJentzen2017}.
\renewcommand{\lstlistingname}{{\sc Matlab} code}
\begin{lstlisting}[language=Matlab,  
caption=
{A {\sc Matlab} code for the classical Monte Carlo method 
used in Subsection~\ref{subsec:example_square_gradient}.}, 
label=code:MC_HJB
]
rng('default');
M = 10^7;
d = 100;
MC = 0;
for m=1:M
    dW = randn(1,d);
    MC = MC + 2/(1+norm(sqrt(2)*dW)^2);
end
MC = -log(MC/M);
\end{lstlisting}
\subsection{A {\sc Matlab} code for the finite differences method used in 
Subsection~\ref{subsec:example_gbm}}
\label{subsec:code_finiteMatlab}
The following {\sc Matlab} code is inspired by
the {\sc Matlab} code in E et al.~\cite[{\sc MATLAB} code 7 in Section~3]{MultilevelPicard}.
\begin{lstlisting}[language=Matlab,
caption={A {\sc Matlab} code for the finite differences method used in 
Subsection~\ref{subsec:example_gbm}.},
label=code:finiteMatlab]
function GBrownianMotion()
    disp(num2str(finiteDiff_GBrownianMotion(0,-2,1000)))
end

function y = finiteDiff_GBrownianMotion(t0,x0,N)
    sigma_max = 1.0;
    sigma_min = 0.5;
    T = 1.0;
    sigma = 1;
    f = @(t,x,y,z,gamma) 0.5 * ((gamma>0).*sigma_max ...
    			+ (gamma<=0).*sigma_min).* gamma;
  
    g = @(x) 1./(1+exp(-x.^2));
    h=(T-t0)./N;
    t=t0:h:T;
    d1 = 0.05;
    d2=sigma*d1;
    x=x0+d1*(-N:1:N);
    M=(1/(h)*[[1, zeros(1,2.*N-2)];
    	full(gallery('tridiag',...
		ones(2.*N-2,1),-2*ones(2.*N-1,1),ones(2.*N-2,1)));
	      	[zeros(1,2*N-2),1]]);
    L=1/(2*d1)*([[-1, zeros(1,2.*N-2)];
    	full(gallery('tridiag',...
		ones(2.*N-2,1),zeros(2.*N-1,1),-ones(2.*N-2,1)));
		[zeros(1,2*N-2),1]]);
    y=g(x);
    for i=N:-1:1
	    x=x(1:(2*i-1))+d2;
        tmp = y(2:(2*i));
        z=y*L(1:(2*i+1),1:(2*i-1));
        gamma=y*M(1:(2*i+1),1:(2*i-1));
        y=tmp+h*f(t(i),x,tmp,z,gamma);
    end
end
\end{lstlisting}
\end{appendix}

\subsection*{Acknowledgements}
Sebastian Becker and Jiequn Han are gratefully acknowledged 
for their helpful and inspiring comments regarding 
the implementation of the deep 2BSDE method. 

\clearpage

\bibliographystyle{acm}
\bibliography{bibfile}

\begin{thebibliography}{100}

\bibitem{Amadori2003Nonlinear}
{\sc Amadori, A.~L.}
\newblock Nonlinear integro-differential evolution problems arising in option
  pricing: a viscosity solutions approach.
\newblock {\em Differential Integral Equations 16}, 7 (2003), 787--811.

\bibitem{AvellanedaLevyParas1995}
{\sc Avellaneda, M., ∗, A.~L., and Par\'as, A.}
\newblock Pricing and hedging derivative securities in markets with uncertain
  volatilities.
\newblock {\em Appl. Math. Finance 2}, 2 (1995), 73--88.

\bibitem{BallyPages2003}
{\sc Bally, V., and Pag\`es, G.}
\newblock A quantization algorithm for solving multi-dimensional discrete-time
  optimal stopping problems.
\newblock {\em Bernoulli 9}, 6 (2003), 1003--1049.

\bibitem{Bayraktar2009Valuation}
{\sc Bayraktar, E., Milevsky, M.~A., Promislow, S.~D., and Young, V.~R.}
\newblock Valuation of mortality risk via the instantaneous {S}harpe ratio:
  applications to life annuities.
\newblock {\em J. Econom. Dynam. Control 33}, 3 (2009), 676--691.

\bibitem{Bayraktar2008Pricing}
{\sc Bayraktar, E., and Young, V.}
\newblock Pricing options in incomplete equity markets via the instantaneous
  sharpe ratio.
\newblock {\em Ann. Finance 4}, 4 (2008), 399--429.

\bibitem{BenderDenk2007}
{\sc Bender, C., and Denk, R.}
\newblock A forward scheme for backward {SDE}s.
\newblock {\em Stochastic Process. Appl. 117}, 12 (2007), 1793--1812.

\bibitem{Bender2015Primal}
{\sc Bender, C., Schweizer, N., and Zhuo, J.}
\newblock A primal-–dual algorithm for {BSDEs}.
\newblock {\em Math. Finance 27}, 3 (2017), 866--901.

\bibitem{Bengio2009}
{\sc Bengio, Y.}
\newblock Learning deep architectures for {AI}.
\newblock {\em Foundations and Trends in Machine Learning 2}, 1 (2009), 1--127.

\bibitem{Bergman1995}
{\sc Bergman, Y.~Z.}
\newblock Option pricing with differential interest rates.
\newblock {\em The Review of Financial Studies 8}, 2 (1995), 475--500.

\bibitem{Bismut1973}
{\sc Bismut, J.-M.}
\newblock Conjugate convex functions in optimal stochastic control.
\newblock {\em J. Math. Anal. Appl. 44\/} (1973), 384--404.

\bibitem{Bouchard2015Lecture}
{\sc Bouchard, B.}
\newblock {\em Lecture notes on BSDEs: Main existence and stability results}.
\newblock PhD thesis, CEREMADE-Centre de Recherches en MAth{\'e}matiques de la
  DEcision, 2015.

\bibitem{BouchardElieTouzi2009}
{\sc Bouchard, B., Elie, R., and Touzi, N.}
\newblock Discrete-time approximation of {BSDE}s and probabilistic schemes for
  fully nonlinear {PDE}s.
\newblock In {\em Advanced financial modelling}, vol.~8 of {\em Radon Ser.
  Comput. Appl. Math.} Walter de Gruyter, Berlin, 2009, pp.~91--124.

\bibitem{BouchardTouzi2004}
{\sc Bouchard, B., and Touzi, N.}
\newblock Discrete-time approximation and {M}onte-{C}arlo simulation of
  backward stochastic differential equations.
\newblock {\em Stochastic Process. Appl. 111}, 2 (2004), 175--206.

\bibitem{BriandLabart2014}
{\sc Briand, P., and Labart, C.}
\newblock Simulation of {BSDEs by Wiener} chaos expansion.
\newblock {\em Ann. Appl. Probab. 24}, 3 (2014), 1129--1171.

\bibitem{cai2017approximating}
{\sc Cai, Z.}
\newblock Approximating quantum many-body wave-functions using artificial
  neural networks.
\newblock {\em arXiv:1704.05148\/} (2017), 8 pages.

\bibitem{Carleo2017Solving}
{\sc Carleo, G., and Troyer, M.}
\newblock Solving the quantum many-body problem with artificial neural
  networks.
\newblock {\em Science 355}, 6325 (2017), 602--606.

\bibitem{ChangLiuXiong2016}
{\sc Chang, D., Liu, H., and Xiong, J.}
\newblock A branching particle system approximation for a class of {FBSDE}s.
\newblock {\em Probab. Uncertain. Quant. Risk 1\/} (2016), Paper No. 9, 34.

\bibitem{Chassagneux2014}
{\sc Chassagneux, J.-F.}
\newblock Linear multistep schemes for {BSDE}s.
\newblock {\em SIAM J. Numer. Anal. 52}, 6 (2014), 2815--2836.

\bibitem{ChassagneuxCrisan2014}
{\sc Chassagneux, J.-F., and Crisan, D.}
\newblock Runge-{K}utta schemes for backward stochastic differential equations.
\newblock {\em Ann. Appl. Probab. 24}, 2 (2014), 679--720.

\bibitem{ChassagneuxRichou2015}
{\sc Chassagneux, J.-F., and Richou, A.}
\newblock Numerical stability analysis of the {E}uler scheme for {BSDE}s.
\newblock {\em SIAM J. Numer. Anal. 53}, 2 (2015), 1172--1193.

\bibitem{ChassagneuxRichou2016}
{\sc Chassagneux, J.-F., and Richou, A.}
\newblock Numerical simulation of quadratic {BSDE}s.
\newblock {\em Ann. Appl. Probab. 26}, 1 (2016), 262--304.

\bibitem{CheriditoSonerTouziVictoir2007}
{\sc Cheridito, P., Soner, H.~M., Touzi, N., and Victoir, N.}
\newblock Second-order backward stochastic differential equations and fully
  nonlinear parabolic {PDE}s.
\newblock {\em Comm. Pure Appl. Math. 60}, 7 (2007), 1081--1110.

\bibitem{chiaramontesolving}
{\sc Chiaramonte, M., and Kiener, M.}
\newblock Solving differential equations using neural networks.
\newblock {\em Machine Learning Project\/}, 5 pages.

\bibitem{Crepeyetal2013}
{\sc Cr\'epey, S., Gerboud, R., Grbac, Z., and Ngor, N.}
\newblock Counterparty risk and funding: the four wings of the {TVA}.
\newblock {\em Int. J. Theor. Appl. Finance 16}, 2 (2013), 1350006, 31.

\bibitem{CrisanManolarakis2010}
{\sc Crisan, D., and Manolarakis, K.}
\newblock Probabilistic methods for semilinear partial differential equations.
  {A}pplications to finance.
\newblock {\em M2AN Math. Model. Numer. Anal. 44}, 5 (2010), 1107--1133.

\bibitem{CrisanManolarakis2012}
{\sc Crisan, D., and Manolarakis, K.}
\newblock Solving backward stochastic differential equations using the cubature
  method: application to nonlinear pricing.
\newblock {\em SIAM J. Financial Math. 3}, 1 (2012), 534--571.

\bibitem{CrisanManolarakis2014}
{\sc Crisan, D., and Manolarakis, K.}
\newblock Second order discretization of backward {SDE}s and simulation with
  the cubature method.
\newblock {\em Ann. Appl. Probab. 24}, 2 (2014), 652--678.

\bibitem{CrisanManolarakisTouzi2010}
{\sc Crisan, D., Manolarakis, K., and Touzi, N.}
\newblock On the {M}onte {C}arlo simulation of {BSDE}s: an improvement on the
  {M}alliavin weights.
\newblock {\em Stochastic Process. Appl. 120}, 7 (2010), 1133--1158.

\bibitem{Darbon2016Algorithmus}
{\sc Darbon, J., and Osher, S.}
\newblock Algorithms for overcoming the curse of dimensionality for certain
  {H}amilton-{J}acobi equations arising in control theory and elsewhere.
\newblock {\em Res. Math. Sci. 3\/} (2016), Paper No. 19, 26.

\bibitem{dehghan2009numerical}
{\sc Dehghan, M., Nourian, M., and Menhaj, M.~B.}
\newblock Numerical solution of {H}elmholtz equation by the modified {H}opfield
  finite difference techniques.
\newblock {\em Numer. Methods Partial Differential Equations 25}, 3 (2009),
  637--656.

\bibitem{DelarueMenozzi2006}
{\sc Delarue, F., and Menozzi, S.}
\newblock A forward-backward stochastic algorithm for quasi-linear {PDE}s.
\newblock {\em Ann. Appl. Probab. 16}, 1 (2006), 140--184.

\bibitem{DouglasMaProtter}
{\sc Douglas, Jr., J., Ma, J., and Protter, P.}
\newblock Numerical methods for forward-backward stochastic differential
  equations.
\newblock {\em Ann. Appl. Probab. 6}, 3 (1996), 940--968.

\bibitem{EHanJentzen2017}
{\sc {E}, W., {Han}, J., and {Jentzen}, A.}
\newblock {Deep learning-based numerical methods for high-dimensional parabolic
  partial differential equations and backward stochastic differential
  equations}.
\newblock {\em arXiv:1706.04702\/} (2017), 39 pages.

\bibitem{LinearScaling}
{\sc E, W., Hutzenthaler, M., Jentzen, A., and Kruse, T.}
\newblock Linear scaling algorithms for solving high-dimensional nonlinear
  parabolic differential equations.
\newblock {\em arXiv:1607.03295\/} (2017), 18 pages.

\bibitem{MultilevelPicard}
{\sc E, W., Hutzenthaler, M., Jentzen, A., and Kruse, T.}
\newblock On multilevel picard numerical approximations for high-dimensional
  nonlinear parabolic partial differential equations and high-dimensional
  nonlinear backward stochastic differential equations.
\newblock {\em arXiv:1708.03223\/} (2017), 25 pages.

\bibitem{ekren2017portfolio}
{\sc Ekren, I., and Muhle-Karbe, J.}
\newblock Portfolio choice with small temporary and transient price impact.
\newblock {\em arXiv:1705.00672\/} (2017), 42 pages.

\bibitem{ElKarouiPengQuenez1997}
{\sc El~Karoui, N., Peng, S., and Quenez, M.~C.}
\newblock Backward stochastic differential equations in finance.
\newblock {\em Math. Finance 7}, 1 (1997), 1--71.

\bibitem{FahimTouziWarin2011}
{\sc Fahim, A., Touzi, N., and Warin, X.}
\newblock A probabilistic numerical method for fully nonlinear parabolic
  {PDE}s.
\newblock {\em Ann. Appl. Probab. 21}, 4 (2011), 1322--1364.

\bibitem{ForsythVetzal2001}
{\sc Forsyth, P.~A., and Vetzal, K.~R.}
\newblock Implicit solution of uncertain volatility/transaction cost option
  pricing models with discretely observed barriers.
\newblock {\em Appl. Numer. Math. 36}, 4 (2001), 427--445.

\bibitem{FuZhaoZhou}
{\sc Fu, Y., Zhao, W., and Zhou, T.}
\newblock Efficient spectral sparse grid approximations for solving
  multi-dimensional forward backward {SDE}s.
\newblock {\em Discrete Contin. Dyn. Syst. Ser. B 22}, 9 (2017), 3439--3458.

\bibitem{GeissLabart2016}
{\sc Geiss, C., and Labart, C.}
\newblock Simulation of {BSDE}s with jumps by {W}iener chaos expansion.
\newblock {\em Stochastic Process. Appl. 126}, 7 (2016), 2123--2162.

\bibitem{geiss2014decoupling}
{\sc Geiss, S., and Ylinen, J.}
\newblock Decoupling on the {W}iener space, related {B}esov spaces, and
  applications to {BSDEs}.
\newblock {\em arXiv:1409.5322\/} (2014), 101 pages.

\bibitem{GobetLabart2010}
{\sc Gobet, E., and Labart, C.}
\newblock Solving {BSDE} with adaptive control variate.
\newblock {\em SIAM J. Numer. Anal. 48}, 1 (2010), 257--277.

\bibitem{GobetLemor2008Numerical}
{\sc Gobet, E., and Lemor, J.-P.}
\newblock Numerical simulation of {BSDE}s using empirical regression methods:
  theory and practice.
\newblock {\em arXiv:0806.4447\/} (2008), 17 pages.

\bibitem{GobetLemorWarin2005}
{\sc Gobet, E., Lemor, J.-P., and Warin, X.}
\newblock A regression-based {M}onte {C}arlo method to solve backward
  stochastic differential equations.
\newblock {\em Ann. Appl. Probab. 15}, 3 (2005), 2172--2202.

\bibitem{GobetLopezSalasTurkedjiev2016}
{\sc Gobet, E., L\'opez-Salas, J.~G., Turkedjiev, P., and V\'azquez, C.}
\newblock Stratified regression {M}onte-{C}arlo scheme for semilinear {PDE}s
  and {BSDE}s with large scale parallelization on {GPU}s.
\newblock {\em SIAM J. Sci. Comput. 38}, 6 (2016), C652--C677.

\bibitem{GobetTurkedjiev2016}
{\sc Gobet, E., and Turkedjiev, P.}
\newblock Approximation of backward stochastic differential equations using
  {M}alliavin weights and least-squares regression.
\newblock {\em Bernoulli 22}, 1 (2016), 530--562.

\bibitem{GobetTurkedjiev2016MathComp}
{\sc Gobet, E., and Turkedjiev, P.}
\newblock Linear regression {MDP} scheme for discrete backward stochastic
  differential equations under general conditions.
\newblock {\em Math. Comp. 85}, 299 (2016), 1359--1391.

\bibitem{GuoZhangZhuo2015}
{\sc Guo, W., Zhang, J., and Zhuo, J.}
\newblock A monotone scheme for high-dimensional fully nonlinear {PDE}s.
\newblock {\em Ann. Appl. Probab. 25}, 3 (2015), 1540--1580.

\bibitem{GuyonLabordere2011}
{\sc Guyon, J., and Henry-Labord{\`e}re, P.}
\newblock The uncertain volatility model: a {M}onte {C}arlo approach.
\newblock {\em The Journal of Computational Finance 14}, 3 (2011), 37--61.

\bibitem{HanE2016Arxiv}
{\sc Han, J., and E, W.}
\newblock {Deep Learning Approximation for Stochastic Control Problems}.
\newblock {\em arXiv:1611.07422\/} (2016), 9 pages.

\bibitem{EHanJentzen2017overcoming}
{\sc Han, J., Jentzen, A., and E, W.}
\newblock Overcoming the curse of dimensionality: Solving high-dimensional
  partial differential equations using deep learning.
\newblock {\em arXiv:1707.02568\/} (2017), 13 pages.

\bibitem{Labordere2012}
{\sc Henry-Labord{\`e}re, P.}
\newblock Counterparty risk valuation: a marked branching diffusion approach.
\newblock {\em arXiv:1203.2369\/} (2012), 17 pages.

\bibitem{Labordereetal2016arXiv}
{\sc Henry-Labord{\`e}re, P., Oudjane, N., Tan, X., Touzi, N., and Warin, X.}
\newblock Branching diffusion representation of semilinear {PDE}s and {M}onte
  {C}arlo approximation.
\newblock {\em arXiv:1603.01727\/} (2016), 30 pages.

\bibitem{LabordereTanTouzi2014}
{\sc Henry-Labord{\`e}re, P., Tan, X., and Touzi, N.}
\newblock A numerical algorithm for a class of {BSDE}s via the branching
  process.
\newblock {\em Stochastic Process. Appl. 124}, 2 (2014), 1112--1140.

\bibitem{HuijskensRuijterOosterlee2016}
{\sc Huijskens, T.~P., Ruijter, M.~J., and Oosterlee, C.~W.}
\newblock Efficient numerical {F}ourier methods for coupled forward-backward
  {SDE}s.
\newblock {\em J. Comput. Appl. Math. 296\/} (2016), 593--612.

\bibitem{IoffeSzegedy2015}
{\sc Ioffe, S., and Szegedy, C.}
\newblock {\em Batch normalization: accelerating deep network training by
  reducing internal covariate shift}.
\newblock {Proceedings of The 32nd International Conference on Machine Learning
  (ICML)}, 2015.

\bibitem{KaratzasShreve2ndEdition}
{\sc Karatzas, I., and Shreve, S.~E.}
\newblock {\em Brownian motion and stochastic calculus}, second~ed., vol.~113
  of {\em Graduate Texts in Mathematics}.
\newblock Springer-Verlag, New York, 1991.

\bibitem{khoo2017solving}
{\sc Khoo, Y., Lu, J., and Ying, L.}
\newblock Solving parametric {PDE} problems with artificial neural networks.
\newblock {\em arXiv:1707.03351\/} (2017).

\bibitem{KingmaBa2015}
{\sc Kingma, D., and Ba, J.}
\newblock {\em {Adam: a method for stochastic optimization}}.
\newblock {Proceedings of the International Conference on Learning
  Representations (ICLR)}, 2015.

\bibitem{KloedenPlaten1992}
{\sc Kloeden, P.~E., and Platen, E.}
\newblock {\em Numerical solution of stochastic differential equations},
  vol.~23 of {\em Applications of Mathematics (New York)}.
\newblock Springer-Verlag, Berlin, 1992.

\bibitem{KongZhaoZhou2015}
{\sc Kong, T., Zhao, W., and Zhou, T.}
\newblock Probabilistic high order numerical schemes for fully nonlinear
  parabolic {PDE}s.
\newblock {\em Commun. Comput. Phys. 18}, 5 (2015), 1482--1503.

\bibitem{Krizhevsky2012}
{\sc Krizhevsky, A., Sutskever, I., and Hinton, G.~E.}
\newblock Imagenet classification with deep convolutional neural networks.
\newblock {\em Advances in Neural Information Processing Systems 25\/} (2012),
  1097--1105.

\bibitem{LabartLelong2013}
{\sc Labart, C., and Lelong, J.}
\newblock A parallel algorithm for solving {BSDE}s.
\newblock {\em Monte Carlo Methods Appl. 19}, 1 (2013), 11--39.

\bibitem{lagaris1998artificial}
{\sc Lagaris, I.~E., Likas, A., and Fotiadis, D.~I.}
\newblock Artificial neural networks for solving ordinary and partial
  differential equations.
\newblock {\em IEEE Transactions on Neural Networks 9}, 5 (1998), 987--1000.

\bibitem{laurent2014overview}
{\sc Laurent, J.-P., Amzelek, P., and Bonnaud, J.}
\newblock An overview of the valuation of collateralized derivative contracts.
\newblock {\em Review of Derivatives Research 17}, 3 (2014), 261--286.

\bibitem{LeCun2015}
{\sc LeCun, Y., Bengio, Y., and Hinton, G.}
\newblock Deep learning.
\newblock {\em Nature 521\/} (2015), 436--444.

\bibitem{Lecun98}
{\sc Lecun, Y., Bottou, L., Bengio, Y., and Haffner, P.}
\newblock Gradient-based learning applied to document recognition.
\newblock {\em Proceedings of the IEEE 86}, 11 (1998), 2278--2324.

\bibitem{LeeKang1990}
{\sc Lee, H., and Kang, I.~S.}
\newblock Neural algorithm for solving differential equations.
\newblock {\em J. Comput. Phys. 91}, 1 (1990), 110--131.

\bibitem{Leland1985}
{\sc Leland, H.~E.}
\newblock Option pricing and replication with transactions costs.
\newblock {\em J. Finance 40}, 5 (1985), 1283--1301.

\bibitem{LemorGobetWarin2006}
{\sc Lemor, J.-P., Gobet, E., and Warin, X.}
\newblock Rate of convergence of an empirical regression method for solving
  generalized backward stochastic differential equations.
\newblock {\em Bernoulli 12}, 5 (2006), 889--916.

\bibitem{LionnetDosReisSzpruch2015}
{\sc Lionnet, A., dos Reis, G., and Szpruch, L.}
\newblock Time discretization of {FBSDE} with polynomial growth drivers and
  reaction-diffusion {PDE}s.
\newblock {\em Ann. Appl. Probab. 25}, 5 (2015), 2563--2625.

\bibitem{MaProtterSanMartin2002}
{\sc Ma, J., Protter, P., San~Mart\'\i{n}, J., and Torres, S.}
\newblock Numerical method for backward stochastic differential equations.
\newblock {\em Ann. Appl. Probab. 12}, 1 (2002), 302--316.

\bibitem{MaProtterYong1994}
{\sc Ma, J., Protter, P., and Yong, J.~M.}
\newblock Solving forward-backward stochastic differential equations
  explicitly---a four step scheme.
\newblock {\em Probab. Theory Related Fields 98}, 3 (1994), 339--359.

\bibitem{MaYong1999}
{\sc Ma, J., and Yong, J.}
\newblock {\em Forward-backward stochastic differential equations and their
  applications}, vol.~1702 of {\em Lecture Notes in Mathematics}.
\newblock Springer-Verlag, Berlin, 1999.

\bibitem{Maruyama1955}
{\sc Maruyama, G.}
\newblock Continuous {M}arkov processes and stochastic equations.
\newblock {\em Rend. Circ. Mat. Palermo (2) 4\/} (1955), 48--90.

\bibitem{McKean1975}
{\sc McKean, H.~P.}
\newblock Application of {B}rownian motion to the equation of
  {K}olmogorov-{P}etrovskii-{P}iskunov.
\newblock {\em Comm. Pure Appl. Math. 28}, 3 (1975), 323--331.

\bibitem{Meade1994Numerical}
{\sc Meade, Jr., A.~J., and Fern\'andez, A.~A.}
\newblock The numerical solution of linear ordinary differential equations by
  feedforward neural networks.
\newblock {\em Math. Comput. Modelling 19}, 12 (1994), 1--25.

\bibitem{mehrkanoon2015learning}
{\sc Mehrkanoon, S., and Suykens, J.~A.}
\newblock Learning solutions to partial differential equations using {LS-SVM}.
\newblock {\em Neurocomputing 159\/} (2015), 105--116.

\bibitem{MilsteinOriginal1974}
{\sc Milstein, G.~N.}
\newblock Approximate integration of stochastic differential equations.
\newblock {\em Teor. Verojatnost. i Primenen. 19\/} (1974), 583--588.

\bibitem{MilsteinTretyakov2006}
{\sc Milstein, G.~N., and Tretyakov, M.~V.}
\newblock Numerical algorithms for forward-backward stochastic differential
  equations.
\newblock {\em SIAM J. Sci. Comput. 28}, 2 (2006), 561--582.

\bibitem{MilsteinTretyakov2007}
{\sc Milstein, G.~N., and Tretyakov, M.~V.}
\newblock Discretization of forward-backward stochastic differential equations
  and related quasi-linear parabolic equations.
\newblock {\em IMA J. Numer. Anal. 27}, 1 (2007), 24--44.

\bibitem{Moreau2017Trading}
{\sc Moreau, L., Muhle-Karbe, J., and Soner, H.~M.}
\newblock Trading with small price impact.
\newblock {\em Math. Finance 27}, 2 (2017), 350--400.

\bibitem{OksendalSDEs}
{\sc {\O}ksendal, B.}
\newblock {\em Stochastic differential equations}.
\newblock Universitext. Springer-Verlag, Berlin, 1985.
\newblock An introduction with applications.

\bibitem{PardouxPeng1990}
{\sc Pardoux, {\'E}., and Peng, S.}
\newblock Adapted solution of a backward stochastic differential equation.
\newblock {\em Systems Control Lett. 14}, 1 (1990), 55--61.

\bibitem{PardouxPeng1992}
{\sc Pardoux, E., and Peng, S.}
\newblock Backward stochastic differential equations and quasilinear parabolic
  partial differential equations.
\newblock In {\em Stochastic partial differential equations and their
  applications ({C}harlotte, {NC}, 1991)}, vol.~176 of {\em Lect. Notes Control
  Inf. Sci.} Springer, Berlin, 1992, pp.~200--217.

\bibitem{PardouxTang1999}
{\sc Pardoux, E., and Tang, S.}
\newblock Forward-backward stochastic differential equations and quasilinear
  parabolic {PDE}s.
\newblock {\em Probab. Theory Related Fields 114}, 2 (1999), 123--150.

\bibitem{Peng1991}
{\sc Peng, S.}
\newblock Probabilistic interpretation for systems of quasilinear parabolic
  partial differential equations.
\newblock {\em Stochastics Stochastics Rep. 37}, 1-2 (1991), 61--74.

\bibitem{Peng2004GExpectations}
{\sc Peng, S.}
\newblock Nonlinear expectations, nonlinear evaluations and risk measures.
\newblock In {\em Stochastic methods in finance}, vol.~1856 of {\em Lecture
  Notes in Math.} Springer, Berlin, 2004, pp.~165--253.

\bibitem{Peng2005NonlinearExpMarkov}
{\sc Peng, S.}
\newblock Nonlinear expectations and nonlinear {M}arkov chains.
\newblock {\em Chinese Ann. Math. Ser. B 26}, 2 (2005), 159--184.

\bibitem{peng2007g}
{\sc Peng, S.}
\newblock {$G$}-expectation, {$G$}-{B}rownian motion and related stochastic
  calculus of {I}t\^o type.
\newblock In {\em Stochastic analysis and applications}, vol.~2 of {\em Abel
  Symp.} Springer, Berlin, 2007, pp.~541--567.

\bibitem{Peng2010BookNonlinearExp}
{\sc Peng, S.}
\newblock Nonlinear expectations and stochastic calculus under uncertainty.
\newblock {\em arXiv:1002.4546\/} (2010), 149 pages.

\bibitem{Pham2015}
{\sc Pham, H.}
\newblock Feynman-{K}ac representation of fully nonlinear {PDE}s and
  applications.
\newblock {\em Acta Math. Vietnam. 40}, 2 (2015), 255--269.

\bibitem{Possamai2015Homogenization}
{\sc Possama\"\i, D., Mete~Soner, H., and Touzi, N.}
\newblock Homogenization and asymptotics for small transaction costs: the
  multidimensional case.
\newblock {\em Comm. Partial Differential Equations 40}, 11 (2015), 2005--2046.

\bibitem{ramuhalli2005finite}
{\sc Ramuhalli, P., Udpa, L., and Udpa, S.~S.}
\newblock Finite-element neural networks for solving differential equations.
\newblock {\em IEEE Transactions on Neural Networks 16}, 6 (2005), 1381--1392.

\bibitem{RasulovRaimoveMascagni2010}
{\sc Rasulov, A., Raimova, G., and Mascagni, M.}
\newblock Monte {C}arlo solution of {C}auchy problem for a nonlinear parabolic
  equation.
\newblock {\em Math. Comput. Simulation 80}, 6 (2010), 1118--1123.

\bibitem{ruder2016overview}
{\sc Ruder, S.}
\newblock An overview of gradient descent optimization algorithms.
\newblock {\em arXiv:1609.04747\/} (2016), 14 pages.

\bibitem{RuijterOosterlee2015}
{\sc Ruijter, M.~J., and Oosterlee, C.~W.}
\newblock A {F}ourier cosine method for an efficient computation of solutions
  to {BSDE}s.
\newblock {\em SIAM J. Sci. Comput. 37}, 2 (2015), A859--A889.

\bibitem{RuijterOosterlee2016}
{\sc Ruijter, M.~J., and Oosterlee, C.~W.}
\newblock Numerical {F}ourier method and second-order {T}aylor scheme for
  backward {SDE}s in finance.
\newblock {\em Appl. Numer. Math. 103\/} (2016), 1--26.

\bibitem{Ruszczynski2017Dual}
{\sc Ruszczynski, A., and Yao, J.}
\newblock A dual method for backward stochastic differential equations with
  application to risk valuation.
\newblock {\em arXiv:1701.06234\/} (2017), 22.

\bibitem{SirignanoDGM2017}
{\sc Sirignano, J., and Spiliopoulos, K.}
\newblock {DGM:} a deep learning algorithm for solving partial differential
  equations.
\newblock {\em arXiv:1708.07469\/} (2017), 7 pages.

\bibitem{SkorohodBranchingDiffusion1964}
{\sc Skorohod, A.~V.}
\newblock Branching diffusion processes.
\newblock {\em Teor. Verojatnost. i Primenen. 9\/} (1964), 492--497.

\bibitem{Tadmor2012}
{\sc Tadmor, E.}
\newblock A review of numerical methods for nonlinear partial differential
  equations.
\newblock {\em Bull. Amer. Math. Soc. (N.S.) 49}, 4 (2012), 507--554.

\bibitem{Thomee1997}
{\sc Thom\'ee, V.}
\newblock {\em Galerkin finite element methods for parabolic problems}, vol.~25
  of {\em Springer Series in Computational Mathematics}.
\newblock Springer-Verlag, Berlin, 1997.

\bibitem{Turkedjiev2015}
{\sc Turkedjiev, P.}
\newblock Two algorithms for the discrete time approximation of {M}arkovian
  backward stochastic differential equations under local conditions.
\newblock {\em Electron. J. Probab. 20\/} (2015), no. 50, 49.

\bibitem{PetersdorffSchwab2004}
{\sc von Petersdorff, T., and Schwab, C.}
\newblock Numerical solution of parabolic equations in high dimensions.
\newblock {\em M2AN Math. Model. Numer. Anal. 38}, 1 (2004), 93--127.

\bibitem{warin2017variations}
{\sc Warin, X.}
\newblock Variations on branching methods for non linear {PDEs}.
\newblock {\em arXiv:1701.07660\/} (2017), 25 pages.

\bibitem{Watanabe1965Branching}
{\sc Watanabe, S.}
\newblock On the branching process for {B}rownian particles with an absorbing
  boundary.
\newblock {\em J. Math. Kyoto Univ. 4\/} (1965), 385--398.

\bibitem{windcliff2007hedging}
{\sc Windcliff, H., Wang, J., Forsyth, P.~A., and Vetzal, K.~R.}
\newblock Hedging with a correlated asset: solution of a nonlinear pricing
  {PDE}.
\newblock {\em J. Comput. Appl. Math. 200}, 1 (2007), 86--115.

\bibitem{ZhangGunzburgerZhao2013}
{\sc Zhang, G., Gunzburger, M., and Zhao, W.}
\newblock A sparse-grid method for multi-dimensional backward stochastic
  differential equations.
\newblock {\em J. Comput. Math. 31}, 3 (2013), 221--248.

\bibitem{Zhang2004}
{\sc Zhang, J.}
\newblock A numerical scheme for {BSDE}s.
\newblock {\em Ann. Appl. Probab. 14}, 1 (2004), 459--488.

\bibitem{KongZhaoZhou2017}
{\sc Zhao, W., Zhou, T., and Kong, T.}
\newblock High order numerical schemes for second-order {FBSDE}s with
  applications to stochastic optimal control.
\newblock {\em Commun. Comput. Phys. 21}, 3 (2017), 808--834.

\end{thebibliography}

\end{document}